\documentclass[nochaptername]{birkmono}

\usepackage{amssymb} 
\usepackage{enumerate} 
\usepackage[all]{xy}
\SelectTips{cm}{}

\usepackage{makeidx}
\makeindex

\usepackage{dsfont}
\usepackage{pdfsync}
\usepackage[colorlinks]{hyperref} 

\theoremstyle{plain}
\newtheorem{theorem}[equation]{Theorem}
\newtheorem{proposition}[equation]{Proposition}
\newtheorem{lemma}[equation]{Lemma} 
\newtheorem{corollary}[equation]{Corollary}

\theoremstyle{definition}
\newtheorem{definition}[equation]{Definition}
\newtheorem{example}[equation]{Example}

\newtheorem{observation}[equation]{Observation}
\newtheorem{warning}[equation]{Warning}

\theoremstyle{remark} 
\newtheorem{remark}[equation]{Remark} 

\renewcommand{\labelenumi}{\rm (\arabic{enumi})}

\newcommand{\ann}{\operatorname{ann}}
\newcommand{\cl}{\operatorname{cl}}
\newcommand{\Char}{\operatorname{char}}
\newcommand{\Coker}{\operatorname{Coker}}

\renewcommand{\dim}{\operatorname{dim}}
\newcommand{\End}{\operatorname{End}}
\newcommand{\Ext}{\operatorname{Ext}}
\newcommand{\tExt}{\mathrm{E\widehat{\phantom{\dot{}}x}t}}

\newcommand{\fEnd}{{\operatorname{\mathcal{E}\!\!\;\mathit{nd}}}}
\newcommand{\fHom}{\operatorname{\mathcal{H}\!\!\;\mathit{om}}}
\newcommand{\grmod}{\operatorname{\mathsf{grmod}}}

\newcommand{\hocolim}{\operatorname{hocolim}}
\newcommand{\Hom}{\operatorname{Hom}}
\newcommand{\sHom}{\underline{\Hom}}
\newcommand{\id}{\operatorname{id}}
\newcommand{\Id}{\operatorname{Id}}
\renewcommand{\Im}{\operatorname{Im}}
\newcommand{\Inj}{\operatorname{\mathsf{Inj}}}
\newcommand{\Ker}{\operatorname{Ker}}
\newcommand{\KInj}[1]{\mathsf K(\Inj #1)}

\newcommand{\KacInj}[1]{{\mathsf K}_{\mathrm{ac}}(\Inj #1)}
\newcommand{\length}{\operatorname{length}}
\newcommand{\Loc}{\operatorname{Loc}}
\renewcommand{\min}{\operatorname{min}}
\renewcommand{\mod}{\operatorname{\mathsf{mod}}}
\newcommand{\Mod}{\operatorname{\mathsf{Mod}}}
\newcommand{\Proj}{\operatorname{Proj}}
\newcommand{\proj}{\operatorname{proj}}
\newcommand{\rad}{\operatorname{rad}}
\newcommand{\rank}{\operatorname{rank}}
\newcommand{\res}{\operatorname{res}}

\newcommand{\Spec}{\operatorname{Spec}}
\newcommand{\Max}{\operatorname{Max}}
\newcommand{\Mor}{\operatorname{Mor}}
\newcommand{\soc}{\operatorname{soc}}
\newcommand{\StMod}{\operatorname{\mathsf{StMod}}}
\newcommand{\stmod}{\operatorname{\mathsf{stmod}}}
\newcommand{\supp}{\operatorname{supp}}
\newcommand{\Supp}{\operatorname{Supp}}
\newcommand{\Thick}{\operatorname{Thick}}
\newcommand{\Tor}{\operatorname{Tor}}

\newcommand{\op}{\mathrm{op}}

\newcommand{\comp}{\mathop{\circ}}
\newcommand{\col}{\colon}

\newcommand{\ges}{{\scriptscriptstyle\geqslant}}
\newcommand{\hh}[1]{H^{*}#1} 
\newcommand{\HH}[2]{H^{#1}(#2)}
\newcommand{\kos}[2]{{#1}/\!\!/{#2}} 
\newcommand{\les}{{\scriptscriptstyle\leqslant}}
\newcommand{\lotimes}{\otimes^{\mathbf L}}
 
\newcommand{\lto}{\longrightarrow}
\newcommand{\smatrix}[1]{\left[\begin{smallmatrix}#1\end{smallmatrix}\right]}

\newcommand{\da}{{\downarrow}}
\newcommand{\ua}{{\uparrow}}
\newcommand{\thra}{{\twoheadrightarrow}}
\newcommand{\xla}{\xleftarrow}
\newcommand{\xra}{\xrightarrow}
\newcommand{\Iff}{\Longleftrightarrow}

\newcommand{\bik}{Benson/Iyengar/Krause}

\def\sfC{\mathcal{C}}

\def\mcI{\mathcal{I}}

\def\mcU{\mathcal{U}} 
\def\mcV{\mathcal{V}}
\def\mcW{\mathcal{W}} 
 
\def\mcZ{\mathcal{Z}}

\def\sfb{\mathsf b} 
\def\sfc{\mathsf c}

\def\sfA{\mathsf A} 
\def\Ab{\mathsf{Ab}} 
\def\sfB{\mathsf B} 
\def\sfC{\mathsf C}
\def\sfD{\mathsf D} 
 
\def\sfG{\mathsf G}
\def\sfK{\mathsf K} 
 
\def\sfM{\mathsf M}
\def\sfS{\mathsf S} 
\def\sfT{\mathsf T} 
\def\sfU{\mathsf U}

\def\bbA{\mathbb A} 
\def\bbC{\mathbb C} 
\def\bbF{\mathbb F} 
\def\bbP{\mathbb P} 
\def\bbQ{\mathbb Q}
\def\bbR{\mathbb R}
\def\bbZ{\mathbb Z} 

\newcommand{\bsa}{\boldsymbol{a}} 
 
\newcommand{\bsr}{\boldsymbol{r}}

\newcommand{\fa}{\mathfrak{a}} 
\newcommand{\fb}{\mathfrak{b}}
\newcommand{\fm}{\mathfrak{m}} 
\newcommand{\fn}{\mathfrak{n}} 
\newcommand{\fp}{\mathfrak{p}}
\newcommand{\fq}{\mathfrak{q}}

\newcommand{\eps}{\varepsilon}

\newcommand{\gam}{\varGamma}

\def\Si{\Sigma} 
\def\si{\sigma}

\def\one{\mathds 1}

\newcommand{\bloc}{{L}}

\begin{document}
\pagenumbering{roman}
\title{Representations of finite groups:\\ Local cohomology and support}

\author{Dave Benson, Srikanth B. Iyengar, Henning Krause}
\date{\today}

\maketitle

\newpage

\thispagestyle{empty}

\newpage

\setcounter{tocdepth}{1}
\setcounter{page}{5}

\tableofcontents
\cleardoublepage

\section*{Preface}
\addcontentsline{toc}{chapter}{Preface}
\markboth{Preface}{Preface}
\thispagestyle{empty}

These are the notes from an Oberwolfach Seminar which we ran from 23--29 May 2010. There were 24 graduate student and postdoctoral participants. Each morning consisted of three lectures, one from each of the organisers. The afternoons consisted of problem sessions, apart from Wednesday which was reserved for the traditional hike to St.~Roman. We have tried to be reasonably faithful to the lectures and problem sessions in these notes, and have added only a small amount of new material for clarification.

The seminar focused on recent developments in classification methods in commutative algebra, group representation theory and algebraic topology. These methods were initiated by Hopkins back in 1987 \cite{Hopkins:1987a}, with the classification of the thick subcategories of the derived category of bounded complexes of finitely generated projective modules over a commutative noetherian ring $R$, in terms of specialisation closed subsets of $\Spec R$. Neeman \cite{Neeman:1992a} (1992) clarified Hopkins' theorem and used analogous methods to classify the localising subcategories\index{localising subcategory} of the derived category of unbounded complexes of modules in terms of arbitrary subsets of $\Spec R$. In 1997, Benson, Carlson and Rickard \cite{Benson/Carlson/Rickard:1997a} proved the thick subcategory theorem for modular representation theory of a finite $p$-group $G$ over an algebraically closed field $k$ of characteristic $p$. Namely, the thick subcategories of the stable category of finitely generated $kG$-modules are classified by the specialisation closed subsets of the homogeneous non-maximal prime ideals in $H^*(G,k)$, the cohomology ring. The corresponding theorem for the localising subcategories of the stable category of
all $kG$-modules has only recently been achieved, in the paper \cite{Benson/Iyengar/Krause:bik3} by the three organisers of the seminar.

\medskip

\begin{center}
\renewcommand{\arraystretch}{1.3}
\begin{tabular}{|c|c|c|} \hline
& Thick subcategories & Localising subcategories \\
& of compact objects & of all objects \\ \hline
$\sfD(R)$ & Hopkins 1987 & Neeman 1992 \\ \hline
$\StMod(kG)$  & Benson, Carlson & Benson, Iyengar \\
 & and Rickard 1997 & and Krause 2008  \\ \hline
\end{tabular}\medskip
\end{center}

In the process of achieving the classification of the localising subcategories of $\StMod(kG)$, a general machinery was established for such classification theorems in a triangulated category; see \cite{Benson/Iyengar/Krause:2008a,Benson/Iyengar/Krause:bik2}. It is also worth mentioning at this stage the work of Hovey, Palmieri and Strickland \cite{Hovey/Palmieri/Strickland:1997a}, who did a great deal to clarify the appropriate settings for these theorems. 

The general setup involves a graded commutative noetherian ring $R$ acting on a compactly generated triangulated category with small coproducts $\sfT$. Write $\Spec R$ for the set of homogeneous prime ideals  of $R$. For each $\fp\in\Spec R$ there is a \emph{local cohomology functor}\index{local cohomology!functor} $\gam_\fp\colon T\to T$. The \emph{support}\index{support} of an object $X$ is defined to be the subset of $\Spec R$ consisting of those $\fp$ such that $\gam_\fp X$ is non-zero.

The object of the game is to establish conditions under which this notion of support classifies the localising subcategories of $\sfT$. This is given in terms of two conditions. The first is the \emph{local-global
principle}\index{local-global principle} that says for each object $X$ in $\sfT$, the localising subcategory of $\sfT$ generated by $X$ is the same as that generated by $\{\gam_\fp X\}$ as $\fp$ runs over the elements of $\Spec R$. The second is a minimality condition, which requires that each $\gam_\fp\sfT$ is either a minimal localising subcategory of $\sfT$ or it is zero. Under these two conditions, we say that $\sfT$ is \emph{stratified}\index{stratified} by the action of $R$, and then we obtain a classification theorem.

In the case of the derived category $\sfD(R)$, Neeman's classification made essential use of the existence of ``field objects''---for a prime ideal $\fp$ of $R$, the field object is the complex consisting of the field of fractions of $R/\fp$, concentrated in a single degree. One of the principle obstructions to carrying out the classification in the finite group case is a lack of field objects; the obstruction theory of Benson, Krause and Schwede \cite{Benson/Krause/Schwede:2004a,Benson/Krause/Schwede:2005a} can be used to show that the required field objects usually do not exist. Circumventing this involves an elaborate series of changes of category, and machinery for transferring stratification along such changes of category. For a general finite group, the strategy is first to
use Quillen stratification to reduce to elementary abelian $p$-groups, where there are still not enough field objects, but then to use a Koszul construction to reduce to an exterior algebra for which there are enough field objects. At this stage, a version of the Bernstein--Gelfand--Gelfand correspondence can be used to get to a graded polynomial ring, where the problem is solved. One consequence of this strategy is that we obtain classification theorems in a number of situations along the route.

In these notes we manage to give a complete proof in the case of characteristic two, where matters are considerably simplified by the fact that the group algebra of an elementary abelian $2$-group is already an exterior algebra.
We found it frustrating that in spite of having an entire week of lectures to explain the theory, we were not able to give a complete proof of the classification theorem for localising subcategories of $\StMod(kG)$, in odd characteristic.  An overview of the classification in general characteristic is given in Section \ref{sec:Wednesday3}, while the proof in characteristic two may be obtained by combining Theorems \ref{thm:2groups} and \ref{thm:main-theorem} with results from Section \ref{sec:Wednesday3}.

\subsection*{A guide to these notes}

In this volume, we have attempted to stick as closely as possible to the format of the Oberwolfach seminar. So the notes are divided into five chapters with four sections in each, corresponding to the five days with three lectures each morning and a problem session in the afternoon. The lecturing, and writing, styles of the three authors are  different, and we have not tried to alter that for the purposes of these notes. In particular, there is a small amount of repetition. But we have tried to be consistent about important details such as notation, and grading everything cohomologically rather than homologically.

Prerequisites for this seminar consist of a solid background in algebra, including the basic theory of rings and modules, Artin--Wedderburn theory, Krull--Remak--Schmidt theorem; basic commutative algebra from the  first chapters of the book of Atiyah and MacDonald; and basic homological algebra including derived functors, Ext and Tor. 
The appendix, describing the theory of support for modules over a commutative ring, is also necessary background material from commutative algebra that is not easy to find in the literature in the exact form in which we require it. The following books may be helpful.
\begin{enumerate}
\item[{[1]}] M. Atiyah and I. MacDonald, Commutative Algebra.
Addison-Wesley, 1969.
\item[{[2]}] D. J. Benson, Representations and cohomology of finite groups I, II, 
Cambridge Studies in Advanced Mathematics 30, 31.
Cambridge University Press, 2nd edition, 1998.
\item[{[3]}] W. Bruns and J. Herzog, Cohen--Macaulay rings,
Cambridge Studies in Advanced Mathematics 39.
Cambridge University Press, 2nd edition, 1998.
\item[{[4]}] R. Hartshorne, Local cohomology: A seminar given by A. Grothendieck 
(Harvard, 1961), Lecture Notes in Math. 41. 
Springer-Verlag, 1967.
\item[{[5]}] A. Neeman, Triangulated categories,
Annals of Mathematics Studies 148.
Princeton University Press, 2001.
\item[{[6]}] C. Weibel, Homological algebra, Cambridge Studies in
  Advanced Mathematics 38.  Cambridge University Press, 1994.
\end{enumerate}

About the exercises: These are from the problem sessions conducted during the seminar, though we have added a few more. Some are routine verifications/computations that have been omitted in the text, while others are quite substantial, and given with the implicit assumption (or hope) that, if necessary, readers would hunt for solutions in other sources.

\subsection*{Acknowledgments}
The first and second authors would like to thank the Humboldt Foundation for their generous support of the research that led to the seminar. The three of us thank the Mathematisches Forschungsinstitut Oberwolfach for hosting this meeting, and the participating students for making it a lively seminar; their names are listed below:

\bigskip

\begin{tabular}{l}
Beck, Kristen A. (Arlington, U.S.A.) \\ Burke, Jesse (Lincoln, U.S.A.) \\
Chen, Xiao-Wu (Paderborn, Germany) \\ Diveris, Kosmas (Syracuse, U.S.A.) \\
Eghbali, Majid (Halle, Germany)  \\ Henrich, Thilo (Bonn, Germany) \\
Hermann, Reiner (Bielefeld, Germany) \\ K\"ohler, Claudia (Paderborn, Germany) \\
Langer, Martin (Bonn, Germany) \\ Lassueur, Caroline L. (Lausanne, Switzerland) \\
Livesey, Michael (Aberdeen, Great Britain) \\ McKemey, Robert (Manchester, Great Britain) \\
Park, Sejong (Seoul, Korea) \\ Psaroudakis, Chrysostomos (Ioannina, Greece) \\
Purin, Marju (Syracuse, U.S.A.) \\ Reid, Fergus (Aberdeen, Great Britain) \\
Robertson, Marcy (Chicago, U.S.A.) \\ Sane, Sarang (Mumbai, India) \\
Scherotzke, Sarah (Paris, France) \\ Shamir, Shoham (Bergen, Norway) \\
Varbaro, Matteo (Genova, Italy) \\ Warkentin, Matthias (Chemnitz, Germany) \\
Witt, Emily E. (Ann Arbor, U.S.A.) \\ Xu, Fei (Bellaterra, Spain)
\end{tabular}

\newpage

\chapter{Monday}
\thispagestyle{empty}
\pagenumbering{arabic}
\setcounter{page}{1}
\label{ch:Monday}
The first section gives a rapid introduction to the subject matter of the seminar. In particular, readers will encounter many notions and constructions, which will be defined and developed only in later sections. It is not expected that everything is to be understood in the first reading. The second section, ~\ref{sec:Monday2}, is a more leisurely (or, a less rapid) discussion of some of the basic theory of group algebras, while the last one, Section~\ref{sec:Monday3}, is an introduction to triangulated categories. 

\section{Overview}
\label{sec:Monday1}

\subsection{Historical Perspective}
This first lecture begins with a brief historical perspective on modular representation theory of finite groups, to give a context for the main results presented in this seminar.

Representation theory of finite groups began in the nineteenth century with the work of Burnside, Frobenius, Schur and others on finite dimensional representations over $\bbR$ and $\bbC$. In this situation, we have the following theorem; see Theorem~\ref{thm:Maschke} for a more elaborate statement, and a proof.

\begin{theorem}[Maschke, 1899]\index{Maschke's Theorem} 
Every representation of a finite group $G$ over a field $k$ of characteristic zero (or more generally, characteristic not dividing the group order) is a direct sum of irreducible representations. Equivalently, every short exact sequence of $kG$-modules splits. \qed
\end{theorem}

If the characteristic of $k$ does not divide $|G|$, the group order, we talk of \emph{ordinary representation theory}.\index{ordinary representation theory} By contrast, \emph{modular representation theory}\index{modular representation theory} refers to the situation where the characteristic does divide $|G|$. In this
case, the regular representation of the group algebra $kG$ is never a direct sum of irreducible representations, because the augmentation map $kG\to k$ sending each group element to the identity is a surjective module homomorphism that does not split.

Tentative beginnings of modular representation theory were made 
by Dickson in the early twentieth century. But it was not until
the work of Richard Brauer from the 1940s to the 1970s that
the subject really took off the ground. Brauer introduced 
modular characters, defect groups, the Brauer homomorphism,
decomposition numbers, and so on; this is still the basic
language for modular representation theory.

For Brauer, one of the principal goals of modular representation
theory was to obtain information about structure of finite groups.
This was the era where a great deal of effort was going into
the classification of the finite simple groups, and Brauer's
methods were an integral part of this effort.

J.~A.~Green, in the decades spanning 
the nineteen sixties to the nineties, pioneered a
change of emphasis from characters to modules. He introduced
such tools as the Green correspondence, vertices and sources, 
Green's indecomposability theorem, the representation ring, etc.

This paved the way for Jon Carlson and others, starting in the 
nineteen eighties, to introduce support variety techniques, which
form the basis for the methods discussed in this seminar.
These techniques were very successful in modular representation
theory, and soon spread to adjacent fields such as restricted Lie
algebras, finite group schemes, commutative algebra, and 
stable homotopy theory. 

\subsection{Classification of finite dimensional modules}

In ordinary representation theory, one classifies finite dimensional 
modules by their characters. Two representations are isomorphic if and 
only if they have the same character. There are a finite number
of irreducible characters, corresponding to the simple modules,
and everything else is a sum of these. In particular, a module is
indecomposable if and only if it is irreducible (Maschke's theorem).

In modular representation theory, non-isomorphic modules can have
the same character. The best one can say is that two modules have
the same Brauer character if and only if they have the same 
composition factors with the same multiplicities. There are only
a finite number of simple modules, corresponding to the irreducible
Brauer characters, and they also correspond to the factors in the Wedderburn
decomposition of the semisimple algebra $kG/J(kG)$.

The Krull--Remak--Schmidt theorem \index{Krull--Remak--Schmidt Theorem} tells us that every finite dimensional
(or equivalently, finitely generated) $kG$-module decomposes into indecomposable factors, and the set of isomorphism classes of the factors (with multiplicities) is an invariant of the module.

Let $p$ be the characteristic of the field $k$. Then there are a finite number of isomorphism classes of indecomposable $kG$-modules if and only if the Sylow $p$-subgroups of $G$ are cyclic (D.~G.~Higman, 1954).
In this situation, we say that $kG$ has \emph{finite representation type}.\index{finite representation type}
The classification of the indecomposable modules for a group with cyclic Sylow $p$-subgroup, or more generally for a block with cyclic defect, was carried out by the work of Brauer, Thompson, Green, Dade and others.

If the Sylow $p$-subgroups of $G$ are non-cyclic, the indecomposables are ``unclassifiable'' except if $p=2$ and the Sylow $2$-subgroups are in the following list.

\begin{itemize}
\item dihedral:\index{dihedral group}
$D_{2^n}=\langle x,y\mid x^{2^{n-1}}=1,\ y^2=1,\ yxy=x^{-1}\rangle$, for $n\ge 2$

\item generalised quaternion:\index{generalised quaternion group}
$Q_{2^n}=\langle x,y\mid x^{2^{n-1}}=1,\ y^2=x^{2^{n-2}},\ yxy^{-1}=x^{-1}\rangle$, for 
$n\ge 3$

\item semidihedral:\index{semidihedral group}
$SD_{2^n}=\langle x,y\mid x^{2^{n-1}}=1,\ y^2=1,\ yxy=x^{2^{n-2}-1}\rangle$,
for $n\ge 4$.
\end{itemize}

In these cases we say that $kG$ has 
\emph{tame representation type},\index{tame representation type}
while in all remaining cases with non-cyclic Sylow $p$-subgroups,
we say that $kG$ has 
\emph{wild representation type}.\index{wild representation type} 
In fact, this
trichotomy between finite, tame and wild representation type occurs
in general for finite dimensional algebras over a field, by a theorem
of Drozd.\index{representation type}

\begin{example} 
\index{Klein four group!representations}
If $G=\bbZ/2\times\bbZ/2=\langle g,h\rangle$ and $k$ is an
algebraically closed field of characteristic two, the
classification of the indecomposable $kG$-modules is given by
the following list.

\begin{itemize}
\item dimension 1: just the trivial module

\item dimension $2n+1$ (for $n\ge 1$): two indecomposables denoted
$\Omega^n(k)$
and $\Omega^{-n}(k)$

\item dimension $2n$ (for $n\ge 1$): an infinite family of modules parametrized by points $\zeta\in\bbP^1(k)$, denoted $L_{\zeta^n}$.
\end{itemize}

For example, the infinite family of two dimensional modules in the above classification is described as follows. If
$(\lambda:\mu)$ is a point in $\bbP^1(k)$ then $M_{(\lambda:\mu)}$ is the representation given by the matrices
\[ 
g \mapsto \begin{pmatrix} 1& 0\\\lambda &1\end{pmatrix}\qquad
h\mapsto \begin{pmatrix} 1& 0\\\mu &1\end{pmatrix}\,. 
\]
It is easy to check that if $(\lambda:\mu)$ and $(\lambda':\mu')$ represent the same point in $\bbP^1(k)$ (i.e., if $\lambda\mu'=\lambda'\mu$) then the representations are isomorphic.
\end{example}

\subsection{Module categories}

Given that the indecomposable modules are usually unclassifiable, how do we
make progress understanding them? Are there organisational principles
that we can use? Can we make less refined classifications that
are still useful?

In categorical language, we study the category $\mod(kG)$ of finitely
generated $kG$-modules, and the larger category $\Mod(kG)$ of all
$kG$-modules. Our goal is to understand ``interesting'' subcategories.

The first thing we need to discuss is the projective and injective modules. Recall that $P$ is \emph{projective}\index{projective!module} if every epimorphism $M \to P$ splits, and $I$ is  \emph{injective}\index{injective!module} if every monomorphism $I \to M$ splits.

\begin{theorem}
\label{thm:kg-selfinjective}
A $kG$-module is projective if and only if it is injective. 
\end{theorem}
\begin{proof}

Here is a sketch of a proof. One of the exercises is to fill in the details. First one proves that $kG$ is injective as a $kG$-module (this is called \emph{self-injectivity}\index{self-injectivity}) by giving a $kG$-module isomorphism between $kG$ and its vector space dual. Since $kG$ is noetherian, direct sums of injective modules are injective, so free modules are injective. Direct summands of injectives are injective, so projectives are injective.

For the converse, we show that every module $M$ embeds in a free module. Namely, if we denote by $M{\downarrow_{\{1\}}\uparrow^G}$ the free module obtained by restricting to the trivial subgroup and then inducing back up to $G$, then there is a monomorphism $M \to M{\downarrow_{\{1\}}\uparrow^G}$. If $M$ is injective, then this map splits and $M$ is a direct summand of a free module, hence projective.
\end{proof}

\begin{remark}
It is also true that a $kG$-module is projective if and only if it is flat; we shall not make use of this fact. Furthermore, if $G$ is a $p$-group ($p$ is the characteristic of $k$), then a $kG$-module is projective if and only if it is free; see Proposition~\ref{prop:kg-local}.
\end{remark}

There are only a finite number of projective indecomposables, or equivalently 
injective indecomposables,
and for any given group these can be understood.
If $P$ is projective indecomposable then $P$ has a unique top composition
factor and a unique bottom composition factor, and they are isomorphic
simple modules $S$. The projective module
$P=P(S)$ is determined up to isomorphism by $S$, and this gives a one to one
correspondence between simples and projective indecomposables. Thus $P(S)$
is both the projective cover and the injective hull of $S$.

\subsection{The stable module category}
\index{stable module category}

Recall that the category
$\Mod(kG)$ has as its objects the $kG$-modules, 
as its arrows the module homomorphisms. 

It is often convenient to work ``modulo the projective modules'', which
leads to the \emph{stable module category} $\StMod(kG)$. This has
the same objects as $\Mod(kG)$, but the arrows are given by quotienting
out the module homomorphisms that factor through some projective module.
We write
\[ 
\sHom_{kG}(M,N) = \Hom_{kG}(M,N)/P\Hom_{kG}(M,N) 
\]
where $P\Hom_{kG}(M,N)$ denotes the linear subspace consisting of homomorphisms that factor through some  projective module. Note that $M\to N$ factors through some projective module if and only if it factors through the projective cover of $N$ ($M\to P(N)\to N$), and also if and only if it factors through the injective hull of $M$ ($M\to I(M) \to N$). This implies in particular that $P\Hom_{kG}(M,N)$ is a linear subspace of $\Hom_{kG}(M,N)$.

We write $\mod(kG)$ and $\stmod(kG)$\index{stmodkg@$\stmod(kG)$} for the full subcategories of finitely generated modules in $\Mod(kG)$ and $\StMod(kG)$\index{StModkg@$\StMod(kG)$} respectively. Note that by the discussion above, a homomorphism of finitely generated modules factors through a projective module if and only if it factors through a finitely generated projective module.

\begin{warning}
It is not true that if a homomorphism of $kG$-modules factors through
a finitely generated module and it factors through a projective
module then it factors through a finitely generated projective module.
\end{warning}

The categories $\Mod(kG)$, $\mod(kG)$ are abelian categories, but $\StMod(kG)$, $\stmod(kG)$ are not. Roughly speaking, the problem is that every homomorphism $M\to N$ of $kG$-modules is equivalent in $\StMod(kG)$ to the surjective homomorphism $M\oplus P(N)\to N$ and to the injective homomorphism $M\to I(M)\oplus N$, so we lose sight of kernels and cokernels.

Instead, $\StMod(kG)$ and $\stmod(kG)$ are examples of \emph{triangulated categories}. This is proved in detail in the book of Happel \cite{Happel:1988a}, and will be discussed in Section~\ref{sec:Monday3}. For now, we just mention that the distinguished triangles in $\StMod(kG)$ come from short exact sequences in $\Mod(kG)$ and
the shift in $\StMod(kG)$ is $\Omega^{-1}$.

\bigskip

Next, we discuss how you tell whether a $kG$-module is projective.
This is done through Chouinard's theorem and Dade's lemma.

\subsection{Chouinard's theorem}
\index{Chouinard's Theorem}
Let $p$ be a prime number and $k$ a field of characteristic $p$.

\begin{definition}
A finite group is an \emph{elementary abelian $p$-group}\index{elementary abelian $p$-group} if it is isomorphic to $(\bbZ/p)^r=\bbZ/p\times\dots \times \bbZ/p$ for some $r$. The number $r$ is called the \emph{rank}.\index{rank}
\end{definition}

\begin{theorem}[Chouinard \cite{Chouinard:1976a} (1976)]\label{th:Chouinard}
A $kG$-module is projective if and only if its restriction to every elementary abelian subgroup of $G$ is projective. \qed
\end{theorem}

\begin{remark}
In fact, Chouinard proved this theorem in the context where $k$ is an
arbitrary commutative rings of coefficients. In this case, one needs
to use elementary abelian subgroups at all primes dividing the group order.
We shall only make use of
the case where $k$ is a field of characteristic $p$, in which case
we only need the elementary abelian $p$-subgroups.
\end{remark}

\begin{example}
Let $G=Q_8$, the quaternions and $k$ a field of characteristic $2$. The only elementary abelian subgroup of $G$ is its centre $Z(G)=\{1,z\}$. In this case, Chouinard's theorem states that a $kG$-module $M$ is projective if and only if its restriction to $Z(G)$ is projective. If the module is finite dimensional, this is equivalent to the statement
that the rank of the matrix representing $1+z$ is as large as it can be, namely one half of the dimension of $M$. This can be seen using the theory of Jordan canonical forms.
\end{example}

\subsection{Dade's lemma}
\index{Dade's Lemma}

Let $E$ be an elementary abelian $p$-group of rank $r$:
\[ 
E=(\bbZ/p)^r=\langle g_1,\dots,g_r\rangle 
\]
and let $k$ be an algebraically closed field of characteristic $p$.

Write $x_i$ for the element $g_i-1$ in $J(kE)$\index{JkE@$J(kE)$}, the Jacobson radical of $kE$, so that 
\[ 
kE=k[x_1,\dots,x_r]/(x_1^p,\dots,x_r^p). 
\]

\begin{definition}
If $\alpha=(\alpha_1,\dots,\alpha_r)\in\bbA^r(k)\setminus\{0\}$ we set
\[ 
x_\alpha=\alpha_1 x_1+\dots+\alpha_r x_r \in J(kE) 
\]
so that $x_\alpha^p=0$, and $1+x_\alpha$ is a unit of order $p$ in $kG$. We call the group $\langle 1+x_\alpha\rangle$ generated by $1+x_{\alpha}$ a \emph{cyclic shifted subgroup}\index{cyclic shifted subgroup} of $E$.
\end{definition}

Recall that projective $kE$-modules (equivalently, injective modules) are free.

\begin{lemma}[Dade \cite{Dade:1978b} (1978)]
A finitely generated $kE$-module $M$ is projective if and only if its restriction to every cyclic shifted subgroup $\langle 1+x_\alpha\rangle$ is free. \qed
\end{lemma}

\begin{remark}
Using the theory of Jordan canonical forms, we can see that the restriction of $M$ to $\langle 1+x_\alpha\rangle$ is free if and only if the rank of the matrix representing $x_\alpha$ is as large as it can be, namely $\left(\frac{p-1}{p}\right).\dim_k(M)$.
\end{remark}

\subsection{Rank varieties}
\label{ssec:Rank varieties}
Although it is logically not necessary to discuss rank varieties for the purpose of understanding the proofs of the main results presented in this seminar, cohomological varieties, which are needed, are very difficult to compute without reference to rank varieties. We therefore include a discussion of rank varieties, first for finitely generated modules and then for infinitely generated modules.

Dade's lemma motivates the following definition.

\begin{definition}[Carlson \cite{Carlson:1981b,Carlson:1983a}]
Let $E$ be an elementary abelian $p$-group and $k$ a field of characteristic $p$. If $M$ is a finitely generated $kE$-module then the \emph{rank variety}\index{rank variety} of $M$ is
\[ 
V^r_E(M)=\{0\}\cup \{\alpha\ne 0 \mid M{\downarrow_{\langle 1+x_\alpha\rangle}}
\text{ is not free}\} \subseteq V^r_E=\bbA^r(k) 
\]
\end{definition}

Here are some properties of rank varieties; some of these are obvious, while others are difficult to prove.
A detailed account of the theory can be found in Chapter~5 of \cite{Benson:1991b}.

\begin{enumerate}
\item $V^r_E(M)$ is a closed homogeneous subvariety of $\bbA^r(k)$.
\item $V^r_E(M)=\{0\}$ if and only if $M$ is projective.
\item $V^r_E(M\oplus N)=V^r_E(M)\cup V^r_E(N)$.
\item $V^r_E(M\otimes_k N)=V^r_E(M)\cap V^r_E(N)$.
\item In any exact sequence $0\to M_1\to M_2\to M_3\to 0$ of $kE$-modules the variety of each module is contained in the union of the varieties of the other two.
\item The Krull dimension of $V^r_E(M)$ measures the rate of growth of the minimal 
projective resolution of $M$.
\item Given a closed homogeneous subvariety $V$ of $\bbA^r(k)$, there exists
a finitely generated $kE$-module $M$ such that $V^r_E(M)=V$.
\end{enumerate}

\subsection{Infinitely generated modules}
\index{infinitely generated modules}

Why study infinitely generated modules?\medskip

Here is an analogy. If you were only interested in finite CW complexes, or even
only interested in manifolds, you'd end up looking at homology theories
(ordinary homology, K-theory, etc.).
The representing objects for these are infinite CW complexes
such as $K(\pi,n)$, $BU(n)$, etc.

Similarly, in representation theory the representing objects for some
natural functors on finitely generated modules are 
infinitely generated. Relevant examples for us are Rickard's idempotent
modules and idempotent functors \cite{Rickard:1997a}, which we shall be
discussing later.

Another relevant motivation comes from commutative algebra. Over a 
commutative noetherian ring, the injective hull of a finitely generated
module is usually not finitely generated. Since the definition of local
cohomology functors in this context involves injective resolutions, 
we should expect to have to deal with infinitely generated modules.

More care is needed in dealing with infinitely generated modules, 
as many of the well known properties of finitely generated modules fail.
So for example a non-zero $kG$-module need not have any indecomposable 
direct summands. Even if it is a finite direct sum of indecomposables,
the Krull--Remak--Schmidt theorem need not hold. 

\begin{warning}
Dade's lemma as stated earlier in this lecture is false for infinitely generated $kG$-modules, as is explained in the following example.
\end{warning}

\begin{example} 
This is the \emph{generic module}\index{generic!module} for $(\bbZ/2)^r=\langle g_1,\dots,g_r\rangle$, for $r\ge 2$.

Let $K=k(t_1,\dots,t_r)$, a pure transcendental extension of $k$ of
transcendence degree $r$, and let $M=K\oplus K$ as a $k$-vector space,
with each $g_i$ acting on $M$ as the matrix
\[ 
\left(\begin{matrix}I&0\\t_i&I\end{matrix}\right). 
\]
Here, $I$ is the identity map on $K$, and $t_i$ is to be thought of as multiplication by $t_i$ on $K$. You are asked to prove in the exercises that $M$ is projective on all cyclic shifted subgroups, but it is not  
projective.

In fact, in a sense that we are about to explain, the variety of this module $M$ is \emph{the generic point}\index{generic!point} of $\bbA^r(k)$.
\end{example}

The following modification of Dade's lemma was proved by Benson, Carlson and Rickard \cite{Benson/Carlson/Rickard:1996a}.

\begin{lemma}
A $kE$-module $M$ is projective if and only if for all extension fields
$K$ of $k$ and all cyclic shifted subgroups $\langle 1+x_\alpha\rangle$ of
$KE$ the module $(K\otimes_k M){\downarrow_{\langle 1+x_\alpha\rangle}}$ 
is free.\qed
\end{lemma}

If $k$ is algebraically closed, write $\mcV_E^r(k)$ for the set of non-zero homogeneous irreducible closed subvarieties $V$ of $\bbA^r(k)$. Then each cyclic shifted subgroup of $KE$ is \emph{generic} for some $V\in\mcV^r_E(k)$ in a sense described in \cite{Benson/Carlson/Rickard:1996a}.

\begin{definition}
If $M$ is a, possibly infinite dimensional, $kE$-module then the \emph{rank variety}\index{rank variety} is no longer a single variety, but a set of varieties:
\[
\mcV_E^r(M)=\left\{V\in \mcV_E^r(k) \left |
\begin{gathered}
K \otimes_k M \text{ restricted to a generic cyclic} \\
\text{shifted subgroup for $V$ is not projective}.
\end{gathered} \right.\right\}
\]
\end{definition}
The following is a list of properties of the rank varieties $\mcV^r_E(M)$; cf. \S\ref{ssec:Rank varieties}. Again, some are obvious and some are difficult to prove; see \cite{Benson/Carlson/Rickard:1996a}.

\begin{enumerate}
\item $\mcV^r_E(M)=\varnothing$ if and only if $M$ is projective.
\item $\mcV^r_E(M\oplus N)=\mcV^r_E(M)\cup\mcV^r_E(N)$, and more generally
\[ 
\mcV^r_E(\bigoplus_\alpha M_\alpha)=\bigcup_\alpha\mcV^r_E(M_\alpha). 
\]
\item $\mcV^r_E(M\otimes_k N)=\mcV^r_E(M)\cap\mcV^r_E(N)$.
\item If $0\to M_1\to M_2\to M_3\to 0$ is a short exact sequence of
$kE$-modules then the variety of each module is contained in the
union of the varieties of the other two.
\item For a finitely generated $kE$-module $M$, $\mcV^r_E(M)$ is the set of 
closed homogeneous irreducible subvarieties of $V_E^r(M)$,
\item 
For any subset $\mcV\subseteq\mcV^r_E(k)$ there exists a $kE$-module $M$
such that $\mcV^r_E(M)=\mcV$.
\end{enumerate}

\subsection{Localising subcategories}
\index{localising subcategory}

Our goal in these lectures is to understand the subcategories of $\Mod(kG)$ in terms of varieties.
Let us state the main theorem for an elementary abelian group $E$.

\begin{definition}\label{def:Mod-loc}
Consider full subcategories $\sfC$ of $\Mod(kE)$ satisfying the following two properties:
\begin{enumerate}
\item $\sfC$ is closed under direct sums, and
\item $\sfC$ has the ``two in three property'':\index{two in three property}\smallskip

If $0\to M_1\to M_2\to M_3\to 0$ is an exact sequence of $kE$-modules and two of $M_1,M_2,M_3$ are in $\sfC$ then so is the third.
\end{enumerate}
\end{definition}

Such subcategories, if they are non-zero, contain the projective modules (Exercise!) and pass down to the \emph{localising subcategories} of $\StMod(kE)$.

The following is the main theorem of \cite{Benson/Iyengar/Krause:bik3}, reinterpreted for rank varieties in the case of an elementary abelian $p$-group.

\begin{theorem}
The non-zero subcategories of $\Mod(kE)$ satisfying the above two conditions are in bijection with subsets of the set $\mcV^r_E(k)$ of non-zero closed homogeneous irreducible subvarieties of $\bbA^r(k)$. Under this bijection, a subset $\mcV\subseteq\mcV^r_E(k)$ corresponds to the full subcategory consisting of those $kE$-modules $M$
satisfying $\mcV_E^r(M)\subseteq\mcV$. \qed
\end{theorem}

\subsection{Thick Subcategories}

The corresponding result for the fi\-nite\-ly generated module category $\mod(kE)$ was proved in
\cite{Benson/Carlson/Rickard:1997a}, and goes as follows.

\begin{definition}
A subset $\mcV$ of a set of varieties is \emph{specialisation closed}\index{specialisation!closed} if
whenever $V\in\mcV$ and $W\subseteq V$ we have $W\in\mcV$.
\end{definition}

\begin{definition}
Consider full subcategories $\sfC$ of $\mod(kE)$ satisfying the following
two properties:
\begin{enumerate}
\item $\sfC$ is closed under finite direct sums and summands, and
\item $\sfC$ has the ``two in three property''\index{two in three property} 
as before.
\end{enumerate}
\end{definition}

Such subcategories if they are non-zero, contain the projective modules and pass down to the \emph{thick subcategories}\index{thick subcategory} of $\stmod(kE)$.

\begin{theorem}
The non-zero subcategories of $\mod(kE)$  satisfying the above two conditions are in bijection with subsets of $\mcV^r_E(k)$ that are closed under specialisation. Under this bijection, a specialisation closed subset $\mcV\subseteq\mcV^r_E(k)$ corresponds to the full subcategory consisting of those finitely generated $kE$-modules $M$ with the property that every irreducible component of $V^r_E(M)$ is an element of $\mcV$. \qed
\end{theorem}

\section{Modules over group algebras}
\label{sec:Monday2}
This section begins afresh, as it were, and introduces basic concepts and constructions concerning modules over group algebra. Our basic reference for this material is \cite{Benson:1991a}; the commutative algebraists among readers may prefer to look also at \cite{Iyengar:2004}.

In what follows $G$ denotes a finite group and $k$ a field. We write $\Char k$ for the characteristic of $k$.
The \emph{group algebra} \index{group algebra} of $G$ over $k$ is the $k$-vector space
\[
kG = \bigoplus_{g\in G}kg
\]
with product induced from $G$. The identity element, denoted $1$, of the group is also the identity for $kG$ and $k$ is identified with the subring $k1$ of $kG$. It is a central subring, so $kG$ is a $k$-algebra. Evidently, the ring $kG$ is commutative if and only if the group $G$ is abelian; see also Exercise~3 at the end of this chapter.

This construction is functorial, on the category of groups.

\begin{example}
When $G=\langle g\mid g^{d}=1\rangle$, a cyclic group $\bbZ/d$ of order $d$, one has $kG= k[x]/(x^{d}-1)$, the polynomial ring in the variable $x$ modulo the ideal generated by $x^{d}-1$. 

More generally, if $G= \langle g_{1},\dots,g_{r}\mid g_{1}^{d_{1}}=1,\dots,g_{r}^{d_{r}}=1\rangle$, with $d_{i}\geq 1$, then
\[
kG\cong k[x_{1},\dots,x_{r}]/(x_{1}^{d_{1}}-1,\dots,x_{r}^{d_{r}}-1)\,,
\]
where $x_{i}$ corresponds to the element $g_{i}$.

A further specialisation plays an important role in these lectures: For $p$ a prime number, the abelian group $(\bbZ/p)^{r}$ is  an \emph{elementary abelian $p$-group of rank $r$}\index{elementary abelian $p$-group}. When $\Char k=p$, setting $z_{i}=x_{i}-1$, gives an isomorphism of $k$-algebras
\[
kG\cong k[z_{1},\dots,z_{r}]/(z_{1}^p,\dots,z_{r}^{p})\,.
\]
\end{example}

\begin{remark}
As a $k$-vector space, the rank of $kG$ equals $|G|$, the order of the group $G$; in particular, $kG$ is a finite dimensional algebra over $k$, so it is artinian and noetherian, both on the left and on the right.
\end{remark}

Now we move on to module theory over $kG$. From this perspective, there are many remarkable features that distinguish $kG$ for arbitrary (even finite dimensional) $k$-algebras. For a start one has:

\begin{remark}
\index{right $kG$-module}
Each right $kG$-module $M$ is, canonically, a left module with product defined by $gm =mg^{-1}$, for $m\in M$ and $g\in G$. Said otherwise, the map $g\mapsto g^{-1}$ gives an isomorphism between $kG$ and its opposite ring. Thus, the category of right modules is equivalent to the category of left modules.

For this reason, henceforth we focus on left modules. 
\end{remark}

\begin{definition}
\label{defn:diagonal-action}
\index{diagonal action}
Let $M$ and $N$ be (left) $kG$-modules. There is a \emph{diagonal action}\index{diagonal action} of $kG$ on the $k$-vector space $M\otimes_{k}N$ defined by 
\[
g(m\otimes n) = gm\otimes gn\quad\text{for all $g\in G$, $m\in M$ and $n\in N$}.
\]
One can verify that the diagonal action on $M\otimes_{k}N$  defines a $kG$-module structure as follows: Since $M,N$  are $kG$-modules, $M\otimes_{k}N$ is a module over $kG\otimes_{k}kG$, with 
\[
(g\otimes h)\cdot (m\otimes n)=(gm\otimes hn)\,.
\]
By functoriality of the construction of group algebras, the diagonal homomorphism of groups $G\to G\times G$, where $g\mapsto (g,g)$, induces a homomorphism of $k$-algebras $kG\to kG\otimes_{k}kG$; see Exercise~1. Restriction the $kG\otimes_{k}kG$ along this map gives the diagonal action on $M\otimes_{k}N$. 

In what follows $M\otimes_{k}N$ will always denote this $kG$-module. This is not to be confused with action of $kG$-action on $M\otimes_{k}N$ induced by $M$ (in which case, one gets a direct sum of $\dim_{k}N$ copies of $M$),  or the one induced by $N$. These are also important and arise naturally; see Definition~\ref{defn:induction-restriction}. 

We note that there is a $kG$-linear isomorphism
\[
M\otimes_{k}N \cong N\otimes_{k}M 
\]
defined by $m\otimes n\mapsto n\otimes m$.

In the same vein, $\Hom_{k}(M,N)$ is a $kG$-module with diagonal action, where for $g\in G$ and $\alpha\in\Hom_{k}(M,N)$ one has
\[
(g\cdot \alpha)(m) = g\alpha(g^{-1}m) \quad\text{for $m\in M$.}
\]
\end{definition}

\begin{definition}
\label{defn:invariants}
\index{invariant submodule}
For each $kG$-module $M$ the subset
\[
M^{G} = \{m\in M\mid \text{$gm=m$ for all $g\in G$}\}
\]
is a submodule, called the \emph{invariant submodule} of $M$.

The \emph{augmentation}\index{augmentation} of $kG$ is the homomorphism of $k$-algebras
\[
\eps\col kG\to k \quad\text{where}\quad \eps\big(\sum_{g\in G}c_{g}g\big) = \sum_{g\in G}c_{g}\,.
\]
That this is a  homomorphism of $k$-algebras can be checked directly, or by noting that it is induced by the constant homomorphism $G\to \{1\}$ of groups. The kernel of $\eps$ is the (two-sided) ideal $\bigoplus_{g}k(g-1)$, where the sum runs over all $g\in G\setminus\{1\}$.

The $kG$-module structure on $k$ thus obtained is called the \emph{trivial} \index{trivial module} one. It is immediate from the description of $\Ker(\eps)$ that there is identification
\[
\Hom_{kG}(k,M) = M^{G}\,.
\]
In particular, the functor defined on objects by $M\mapsto M^{G}$ is left exact; right exactness is equivalent to the projectivity of $k$ as a $kG$-module; see also Theorem~\ref{thm:Maschke}.
\end{definition}

\begin{remark}
\label{rem:adjunction-isomorphism}
\index{adjunction isomorphism}
It is easy to verify that for any $kG$-modules $M$ and $N$ one has
\[
\Hom_{kG}(M,N) = \Hom_{k}(M,N)^{G}\,.
\]
Recall that $kG$ acts on $\Hom_{k}(M,N)$ diagonally; see Definition~\ref{defn:diagonal-action}. 

For each $kG$-module $L$, the adjunction isomorphism of $k$-vector spaces:
\[
\Hom_{k}(L,\Hom_{k}(M,N))\cong \Hom_{k}(L\otimes_{k}M,N)
\]
is compatible with the $kG$-module structures; this can (and should) be verified directly. Applying $(-)^{G}$ to it yields the \emph{adjunction isomorphism}:
\begin{equation}
\label{eqn:adjunction-isomorphism}
\index{adjunction isomorphism}
\Hom_{kG}(L,\Hom_{k}(M,N))\cong \Hom_{kG}(L\otimes_{k}M,N)\,.
\end{equation}
\end{remark}

The adjunction isomorphism has the following remarkable consequence. 

\begin{proposition}
For any projective $kG$-module $P$ the $kG$-modules $M\otimes_{k}P$ and $P\otimes_{k}M$ are projective.
\end{proposition}

Note: The $kG$-modules $M\otimes_{k}kG$ and $kG\otimes_{k}M$ are even free; see Exercise~5.
 
\begin{proof}
The functors $\Hom_{k}(M,-)$ and $\Hom_{kG}(P,-)$ of $kG$-modules are exact and hence so is their composition $\Hom_{kG}(P,\Hom_{k}(M,-))$. The latter is isomorphic to $\Hom_{kG}(P\otimes_{k}M,-)$, by the adjunction isomorphism~\eqref{eqn:adjunction-isomorphism}, so one deduces that the $kG$-module $P\otimes_{k}M$ is projective. It remains to note that this module is isomorphic, as a $kG$-module, to $M\otimes_{k}P$.
\end{proof}

\begin{definition}
\label{defn:induction-restriction}
Let $H\le G$ be a subgroup and $\iota\col kH\subseteq kG$ the induced inclusion of $k$-algebras.

Restriction of scalars along $\iota$ endows each $kG$-module $M$ with a structure of a $kH$-module, denoted $M\da_{H}$. One thus gets a \emph{restriction}\index{restriction} functor
\[
(-)\da_{H}\col \Mod kG\to \Mod kH\,.
\]
This functor is evidently exact. Base change along $\iota$ induces a  functor
\[
(-)\ua^{G} = kG\otimes_{kH}- \col \Mod kH \to \Mod kG\,.
\]
This functor is called \emph{induction}\index{induction}; it is also exact, because $kG$ is a free $kH$-module; see Exercise~8.  One can, and does, also consider the \emph{coinduction}\index{coinduction} functor
\[
\Hom_{kH}(kG,-)\col \Mod kH \to \Mod kG\,.
\]
For any $kH$-module $L$ there is a natural isomorphism $\Hom_{kH}(kG,L)\cong L\ua^{G}$ of $kG$-modules. The version of the adjunction 
\begin{equation}
\label{eqn:frobenius-reciprocity}
\Hom_{kG}(L\ua^{G},M)\cong \Hom_{kH}(L,M\da_{H})\,.
\end{equation}
is called \emph{Frobenius reciprocity}\index{Frobenius!reciprocity}.
\end{definition}

\subsection{Structure of the ring $kG$}
Next we recall some of the salient features and results about the group algebra. The starting point is a souped-up version of Maschke's Theorem.

\begin{theorem}
\label{thm:Maschke}
\index{Maschke's Theorem}
\index{semi-simplicity}
The following conditions are equivalent.
\begin{enumerate}[\quad\rm(1)]
\item
The ring $kG$ is semi-simple.
\item
The trivial module $k$ is projective.
\item
The functor $(-)^{G}$ on $\Mod kG$ is exact.
\item
$\Char k$ does not divide $|G|$.
\end{enumerate}
\end{theorem}

\begin{proof}
The implication (1) $\implies$ (2) is clear.

(2) $\implies$ (4) Let $\eps\col kG\to k$ be the augmentation homomorphism, which defines the trivial action on $k$. When $k$ is projective there exists a homomorphism $\sigma\col k\to kG$ of $kG$-modules such that $\eps\circ\sigma=\id_{k}$. Write
\[
\sigma(1)=\sum_{g\in G}c_{g}g\,.
\]
For each $h\in G$ one has $\sigma(1)=\sigma(h^{-1}.1)=h^{-1}\sum_{g\in G}c_{g}g$. So comparing coefficients of the identity element in $G$ one gets that $c_{h}=c_{1}$. Thus,  $\sigma(1)=c_{1}\sum_{g\in G}g$ so that in $k$ there is an equality $1=\eps\sigma(1)=c_{1}|G|$; in particular $|G|\ne 0$.

(4) $\implies$ (3) Let $M$ be a $kG$-module. Consider the map 
\[
\rho\col M\to M \quad\text{where $\rho(m) = \frac 1{|G|}\sum_{g\in G}gm$.}
\]
A straightforward computation shows that this map is $kG$-linear and an identity when restricted to $M^{G}$, the submodule of invariants. The construction is evidently functorial, which means that $(-)^{G}$ is a direct summand of the identity functor, and since the latter is exact so is the former.

(3) $\implies$ (1) It suffices to prove that any epimorphism $M\thra N$ of $kG$-modules splits. Any such map induces an epimorphism $\Hom_{k}(N,M)\thra \Hom_{k}(N,N)$, and hence, since $(-)^{G}$ is exact, also an epimorphism
\[
\Hom_{kG}(N,M) = \Hom_{k}(N,M)^{G}\thra \Hom_{k}(N,N)^{G} = \Hom_{kG}(N,N)\,.
\]
Then the identity on $N$ lifts to a $kG$-linear map $N\to M$, as desired.
\end{proof}

In any characteristic, the group algebra $kG$ has the following properties:
\begin{itemize}
\item The Krull--Remak--Schmidt theorem holds in $\mod kG$.
\item The group algebra $kG$ is \emph{self-injective}, meaning, that it is injective as a module over itself. A proof of this assertion is sketched in Theorem~\ref{thm:kg-selfinjective}, and also part of Monday's exercises.
\item There are canonical bijections of isomorphism classes
\begin{gather*}
\left\{
\begin{gathered}
\text{Simple modules} 
\end{gathered}
\right\}
\leftrightarrow
\left\{
\begin{gathered}
\text{indecomposable} \\
\text{projectives}
\end{gathered}
\right\}
\leftrightarrow
\left\{
\begin{gathered}
\text{indecomposable} \\
\text{injectives}
\end{gathered}
\right\}
\end{gather*}
where Artin--Wedderburn theory gives the one on the left, and the one on the right holds by self-injectivity of $kG$.
\end{itemize}

Given these properties, it is not hard to prove the following characterization of $p$-groups; for a proof, see, for instance, \cite[(1.5) and (1.6)]{Iyengar:2004}.

\begin{proposition}
\label{prop:kg-local}
The following conditions are equivalent.
\begin{enumerate}[\quad\rm(1)]
\item
The ring $kG$ is local; i.e. it has a unique maximal ideal.
\item
The trivial module $k$ is the only simple module.
\item
$G$ is a $p$-group. \qed
\end{enumerate}
\end{proposition}

\subsection{Group cohomology}
\label{ssec:gc}
In the remainder of this section we assume $\Char k$ divides $|G|$; equivalently, $(-)^{G}$ is not exact; see Theorem~\ref{thm:Maschke}. For each $kG$-module $M$ and integer $n$, the $n$th \emph{cohomology of $G$ with coefficients in $M$} is the $k$-vector space:
\[
\HH n{G;M} = \Ext^{n}_{kG}(k,M)\,.
\]
We speak of $\HH n{G;k}$ as the cohomology of $G$. Yoneda composition induces on the graded $k$-vector space $\HH *{G;k}$ the structure of a $k$-algebra, and on $\HH*{G;M}$ the structure of a graded right $H^{*}(G,k)$-module.

Recall that a graded ring $R$ is said to be \emph{graded commutative} \index{graded commutative} if
\[
a\cdot b = (-1)^{|a||b|}b\cdot a\quad\text{for all $a,b\in A$.}
\]
Here is a remarkable feature of the cohomology algebra of a group:

\begin{proposition}
\label{prop:graded-commutative}
The $k$-algebra $H^{*}(G,k)$ is graded commutative. \qed
\end{proposition}

The key point is that $kG$ is in fact a Hopf algebra, with diagonal induced by the diagonal homomorphism $G\to G\times G$; see, for instance, \cite[Proposition 5.5 ]{Iyengar:2004}.

The cohomology algebra of an elementary abelian group is easy to compute, and has been known for long; see \cite[Chapter XII, \S7]{Cartan/Eilenberg:1956}. We give the argument for $2$-groups; see Section~\ref{sec:Tuesday3} for a different perspective on this computation and a more detailed discussion of the group cohomology algebra.

\begin{proposition}
\label{prop:gc-eab2}
\index{cohomology!elementary abelian group}
Let $G=(\bbZ/2)^{r}$ and $\Char k=2$. Then $H^{*}(G,k)$ is a polynomial algebra over $k$ on $r$ variables each of degree one:
\[
H^{*}(G,k)\cong k[y_{1},\dots,y_{r}] \quad \text{where $|y_{i}|=1$}\,.
\]
\end{proposition}

\begin{proof}
Consider first the case where $r=1$, so $G=\bbZ/2$. The claim is then that there is an isomorphism of $k$-algebras $\HH *{G;k}\cong k[y]$, where $|y|=1$. Indeed, the group algebra $kG$ is then isomorphic to $k[x]/(x^{2})$, and the  complex 
\[
\cdots \to kG\xra{ x } kG \xra{ x } kG \to 0\,,
\]
is a free resolution of the trivial $kG$-module $k$. It follows that $\Ext_{kG}^{n}(k,k)\cong k$ for each $n\geq 0$, and that the class of the extension 
\[
0\to k\xra{ \eta }kG\xra{ \eps }k\to 0\,,
\]
where $\eta(1)=x$ and $\eps$ is the augmentation, generates $H^{*}(G,k)$ as a $k$-algebra.

The structure of the cohomology algebra for $r\geq 1$ then follows from repeated applications of the K\"unneth isomorphism:
\[
\Ext_{A\otimes_{k}B}(k,k)\cong \Ext_{A}(k,k)\otimes_{k}\Ext_{B}(k,k)
\]
where $A$ and $B$ are augmented $k$-algebras; see \cite[Chapter~XI]{Cartan/Eilenberg:1956}.
\end{proof}

Here is the corresponding result for odd primes. It can be proved along the same lines as the preceding one.

\begin{proposition}
\label{prop:gc-eabp}
\index{cohomology!elementary abelian group}
Let $p$ be an odd prime, $G=(\bbZ/p)^{r}$ and $\Char k=p$. Then $H^{*}(G,k)$ is the tensor product of an exterior algebra in $r$ variables in degree one and a polynomial algebra in $r$ variables in degree two:
\[
H^{*}(G,k)\cong (\bigwedge\bigoplus_{i=1}^{r}ky_{i}) \otimes_{k} k[z_{1},\dots,z_{r}] \quad \text{where $|y_{i}|=1$ and $|z_{i}|=2$}\,.
\]
\end{proposition}

\section{Triangulated categories}
\label{sec:Monday3}
This lecture provides a quick introduction to triangulated categories. We give definitions and explain the basic concepts. The triangulated categories arising in this work are always \emph{algebraic}\index{triangulated category!algebraic} which means that they are equivalent to the stable category of some Frobenius category. We do not use this fact explicitly, but it helps sometimes to understand the triangulated structure.

Triangulated categories were introduced by Verdier in his thesis which was published posthumously \cite{Verdier:1997a}; it is still an excellent reference. Another source for the material in this section is \cite{Krause:2007a}.
For the material on exact categories and Frobenius categories see Happel~\cite{Happel:1988a} and also B\"uhler's survey article~\cite{Buehler:2010}.

\subsection{Triangulated categories}

Let $\sfT$ be an additive category together with a fixed equivalence $\Si\colon\sfT\xra{\sim}\sfT$, which one calls \emph{shift}\index{shift} or \emph{suspension}\index{suspension}. A \emph{triangle}\index{triangle} in $\sfT$ is a sequence
$(\alpha,\beta,\gamma)$ of morphisms
\[
X\xra{\alpha} Y\xra{\beta} Z\xra{\gamma}\Si X\,,
\] 
and a morphism between  triangles $(\alpha,\beta,\gamma)$ and $(\alpha',\beta',\gamma')$ is a triple $(\phi_1,\phi_2,\phi_3)$ of morphisms in $\sfT$ making the following diagram commutative.
\[
\xymatrix{
X\ar[r]^\alpha\ar[d]^{\phi_1}&Y\ar[r]^\beta\ar[d]^{\phi_2}&Z\ar[r]^\gamma\ar[d]^{\phi_3}&\Si
X\ar[d]^{\Si\phi_1}\\ X'\ar[r]^{\alpha'}&Y'\ar[r]^{\beta'}&Z'\ar[r]^{\gamma'}&\Si X'
}
\] 
The category $\sfT$ is called \emph{triangulated} if it is equipped with a class of distinguished triangles (called \emph{exact triangles})\index{triangle!exact} satisfying the following conditions.
\begin{enumerate}
\item[(T1)] Any triangle isomorphic to an exact triangle is exact. For each object $X$, the triangle $0\to X\xra{\id} X\to 0$ is exact. Each morphism $\alpha$ fits into an exact triangle $(\alpha,\beta,\gamma)$.
\item[(T2)] A triangle $(\alpha,\beta,\gamma)$ is exact if and only if
$(\beta,\gamma,-\Si\alpha)$ is exact.
\item[(T3)] Given  exact triangles $(\alpha,\beta,\gamma)$ and
$(\alpha',\beta',\gamma')$, each pair of morphisms $\phi_1$ and $\phi_2$ satisfying
$\phi_2 \alpha=\alpha' \phi_1$ can be completed to a morphism of triangles:
\[
\xymatrix{
X\ar[r]^\alpha\ar[d]^{\phi_1}&Y\ar[r]^\beta\ar[d]^{\phi_2}&Z\ar[r]^\gamma\ar[d]^{\phi_3}&\Si
X\ar[d]^{\Si\phi_1}\\ X'\ar[r]^{\alpha'}&Y'\ar[r]^{\beta'}&Z'\ar[r]^{\gamma'}&\Si X'
}
\] 
\item[(T4)] Given exact triangles $(\alpha_1,\alpha_2,\alpha_3)$, $(\beta_1,\beta_2,\beta_3)$, and $(\gamma_1,\gamma_2,\gamma_3)$ with $\gamma_1=\beta_1 \alpha_1$, there exists an exact triangle $(\delta_1,\delta_2,\delta_3)$ making the following diagram commutative.
\[
\xymatrix{X\ar[r]^{\alpha_1}\ar@{=}[d]&Y\ar[r]^{\alpha_2}\ar[d]^{\beta_1}&
U\ar[r]^{\alpha_3}\ar[d]^{\delta_1}& \Si X\ar@{=}[d]\\
X\ar[r]^{\gamma_1}&Z\ar[r]^{\gamma_2}\ar[d]^{\beta_2}&
V\ar[r]^{\gamma_3}\ar[d]^{\delta_2}&\Si X\ar[d]^{\Si\alpha_1}\\
&W\ar@{=}[r]\ar[d]^{\beta_3}& W\ar[d]^{\delta_3}\ar[r]^{\beta_3}&\Si Y\\ &\Si
Y\ar[r]^{\Si\alpha_2}&\Si U }
\]
\end{enumerate}

The axiom (T4) is known as \emph{octahedral axiom}\index{octahedral axiom} because the four exact triangles can be arranged in a diagram having the shape of an octahedron. The exact triangles $A\to B\to C\to\Si A$ are represented
by faces of the form
\[
\xymatrix{&C\ar[ld]|-{+}\\A\ar[rr]&&B\ar[lu]}
\]
and the other four faces are commutative triangles.

Let us give a more intuitive formulation of the octahedral axiom which is based on the notion of a homotopy cartesian square.  Call a commutative square
\[
\xymatrix{X\ar[r]^{\alpha'}\ar[d]^{\alpha''}&Y'\ar[d]^{\beta'}\\Y'' \ar[r]^{\beta''}&Z}
\] 
\emph{homotopy cartesian}\index{homotopy!cartesian} if there exists an exact triangle
\[
X\xra{\smatrix{\alpha'\\ \alpha''}}Y'\amalg Y''\xra{\smatrix{\beta'&-\beta''}}Z\xra{\gamma}\Si X.
\] 
The morphism $\gamma$ is called a \emph{differential} \index{homotopy!cartesian!differential} of the homotopy cartesian square.  Note that a differential of the homotopy cartesian square changes its sign if the
square is flipped along the main diagonal. 

Assuming (T1)--(T3), one can show that (T4) is equivalent to the following condition; see \cite[\S2.2]{Krause:2007a}.
\begin{enumerate}
\item[(T4$'$)] Every pair of morphisms $X\to Y$ and $X\to X'$ can be
completed to a morphism
\[
\xymatrix{
  X\ar[r]\ar[d]^{}&Y\ar[r]\ar[d]^{}&Z\ar[r]\ar@{=}[d]^{}&\Si
  X\ar[d]^{}\\ X'\ar[r]^{}&Y'\ar[r]^{}&Z\ar[r]^{}&\Si X' }
  \] 
between exact triangles such that the left hand square is homotopy cartesian and the composite $Y'\to Z\to\Si X$ is a differential.
\end{enumerate}

\subsection{Categories of complexes}

Let $\sfA$ be an additive category. A \emph{complex}\index{complex} in $\sfA$ is a sequence of morphisms
\[
\cdots \to X^{n-1}\xra{d^{n-1} }X^n\xra{d^n}X^{n+1}\to\cdots
\]
such that $d^{n}\comp d^{n-1}=0$ for all $n\in\bbZ$. A morphism $\phi\colon X\to Y$ between complexes consists of morphisms $\phi^n\colon X^n\to Y^n$ with $d_Y^{n}\comp\phi^n=\phi^{n+1}\comp d_X^n$ for all $n\in\bbZ$.  The complexes form a category which we denote by $\sfC(\sfA)$.

A morphism $\phi\colon X\to Y$ is \emph{null-homotopic} \index{morphisms!null-homotopic} if there are morphisms
$\rho^n\colon X^n\to Y^{n-1}$ such that $\phi^n=d_Y^{n-1}\comp \rho^{n}+\rho^{n+1}\comp d_X^n$ for all $n\in\bbZ$. 
Morphisms $\phi,\psi\col X\to Y$ are \emph{homotopic}\index{morphisms!homotopic} if $\phi-\psi$ is null-homotopic.

The null-homotopic morphisms form an \emph{ideal} $\mcI$ in $\sfC(\sfA)$, that is, for each pair
$X,Y$ of complexes a subgroup
\[
\mcI(X,Y)\subseteq\Hom_{\sfC(\sfA)}(X,Y)
\]
such that any composite $\psi\comp\phi$ of morphisms in $\sfC(\sfA)$ belongs to $\mcI$ if $\phi$ or $\psi$ belongs to $\mcI$.  The \emph{homotopy category}\index{homotopy!category} $\sfK(\sfA)$ is the quotient of $\sfC(\sfA)$
with respect to this ideal. Thus
\[
\Hom_{\sfK(\sfA)}(X,Y)=\Hom_{\sfC(\sfA)}(X,Y)/\mcI(X,Y)
\] 
for every pair of complexes $X,Y$.

Given any complex $X$, its \emph{suspension} \index{complex!suspension} or \emph{shift} \index{complex!shift} is the complex $\Si X$ with
\[
(\Si X)^n=X^{n+1}\quad\textrm{and}\quad d^n_{\Si X}=-d^{n+1}_X.
\] 
This yields an equivalence $\Si\colon\sfK(\sfA)\xra{\sim}\sfK(\sfA)$. The \emph{mapping cone}\index{mapping cone} of a morphism $\alpha\colon X\to Y$ of complexes is the complex $Z$ with $Z^n=X^{n+1}\amalg Y^n$ and differential 
\[
\begin{bmatrix}
-d^{n+1}_X & 0 \\ 
\alpha^{n+1}&d^n_Y
\end{bmatrix}
\]
The mapping cone fits into a \emph{mapping cone sequence}\index{mapping cone!sequence}
\[
X\xra{\alpha}Y\xra{\beta}Z\xra{\gamma}\Si X
\] 
which is defined in degree $n$ by the following sequence.
\[
X^n\xra{\alpha^n} Y^n\xra{\smatrix{0\\ \id}} X^{n+1}\amalg Y^{n} \xra{\smatrix{-\id&0}} X^{n+1}
\] 
By definition, a triangle in $\sfK(\sfA)$ is \emph{exact}\index{triangle!exact} if it is isomorphic to a mapping cone sequence as above. It is straightforward to verify  the axioms (T1)--(T4).
Thus $\sfK(\sfA)$ is a triangulated category.

\subsection{Exact categories}
\index{exact category}
Let $\sfA$ be an exact category in the sense of Quillen \cite{Quillen:1973a}.  Thus
$\sfA$ is an additive category, together with a distinguished class of sequences
\begin{equation*}
\label{eq:ses}
0\to X\xra{\alpha} Y\xra{\beta} Z\to 0
\end{equation*}
which are called \emph{exact}\index{exact category!exact sequence}.  The exact sequences satisfy a number of
axioms. In particular, the morphisms $\alpha$ and $\beta$ in each exact sequence as above form a \emph{kernel-cokernel pair}\index{exact category!kernel-cokener pair}, that is $\alpha$ is a kernel of $\beta$ and $\beta$ is a cokernel of $\alpha$. A morphism in $\sfA$ which arises as the kernel in some exact sequence is called \emph{admissible mono}\index{exact category!admissible mono}; a morphism arising as a cokernel is called \emph{admissible epi}\index{exact category!admissible epi}. A full subcategory $\sfB$ of $\sfA$ is \emph{extension-closed}\index{exact category!extension-closed subcategory} if every exact sequence in $\sfA$ with endterms in $\sfB$ belongs to $\sfB$.

\begin{remark}
(1) Any abelian category is exact with respect to the class of
all short exact sequences. 

(2) Any full and extension-closed subcategory $\sfB$ of an exact
category $\sfA$ is exact with respect to the class of sequences which
are exact in $\sfA$. 

(3) Any small exact category arises, up to an exact equivalence, as a
    full and extension-closed subcategory of a module category.
\end{remark}

\subsection{Frobenius categories}

Let $\sfA$ be an exact category. An object $P$ is called \emph{projective}\index{projective} if the induced map $\Hom_\sfA(P,Y)\to\Hom_\sfA(P,Z)$ is surjective for every admissible epi $Y\to Z$. Dually, an object $Q$ is
\emph{injective}\index{injective} if the induced map $\Hom_\sfA(Y,Q)\to\Hom_\sfA(X,Q)$ is surjective for every admissible mono $X\to Y$.  The category $\sfA$ has \emph{enough projectives}\index{enough projectives} if every object $Z$ admits an admissible epi $Y\to Z$ with $Y$ projective.  And $\sfA$ has \emph{enough injectives}\index{enough injectives} if every object $X$ admits an admissible mono $X\to Y$ with $Y$ injective. Finally, $\sfA$ is called a \emph{Frobenius category}\index{Frobenius category}, if $\sfA$ has enough projectives and enough injectives
and if both coincide.

\begin{example}
(1) Let $\sfA$ be an additive category. Then $\sfA$ is an exact category
    with respect to the class of all split exact sequences in
    $\sfA$. All objects are projective and injective, and $\sfA$ is a
    Frobenius category.

    (2) Let $\sfA$ be an additive category. The category $\sfC(\sfA)$
    of complexes is exact with respect to the class of all sequences
    $0\to X\to Y\to Z\to 0$ such that $0\to X^n\to Y^n\to Z^n\to 0$ is
    split exact for all $n\in\bbZ$. A typical projective and injective
    object is a complex of the form
\[
I_A\colon\;\; \cdots \to 0\to A\xra{\id} A\to 0\to\cdots
\]
for some $A$ in $\sfA$. There is an obvious admissible mono $X\to\prod_{n\in\bbZ}\Si^{-n}I_{X^n}$ and also an admissible epi $\coprod_{n\in\bbZ}\Si^{-n-1}I_{X^n}\to X$. Also,
\[
\coprod_{n\in\bbZ}\Si^{-n}I_{X^n}\cong \prod_{n\in\bbZ}\Si^{-n}I_{X^n}.
\]
Thus $\sfC(\sfA)$ is a Frobenius category. For a conceptual
explanation, see  Exercise~22 at the end of this chapter.

(3) Let $k$ be a field and $A$ a finite dimensional self-injective $k$-algebra. Then injective and
    projective $A$-modules coincide. Thus the category of $A$-modules is a Frobenius category.
\end{example}

\subsection{The stable category of a Frobenius category}
\label{ss:stab}

Let $\sfA$ be a Frobenius category. The \emph{stable category}\index{stable category}\index{Frobenius category!stable category} $\sfS(\sfA)$ is by definition the quotient of $\sfA$ with respect to the ideal $\mcI$ of morphisms which factor through an injective object.  Thus the objects of $\sfS(\sfA)$ are the same as in $\sfA$ and
\[
\Hom_{\sfS(\sfA)}(X,Y)=\Hom_\sfA(X,Y)/\mcI(X,Y)
\] 
for all $X,Y$ in $\sfA$. 

We define a triangulated structure for $\sfS(\sfA)$ as follows. Choose for each $X$ in $\sfA$ an exact sequence
\[
0\to X\to E\to \Si X\to 0
\] 
such that $E$ is injective. One obtains an equivalence $\Si\colon\sfS(\sfA)\xra{\sim}\sfS(\sfA)$ by sending
$X$ to $\Si X$. This equivalence serves as suspension.  Every exact sequence $0\to X\to Y\to Z\to 0$ fits into a commutative diagram with exact rows
\[
\xymatrix{ 
0\ar[r]&X \ar[r]^\alpha\ar@{=}[d]&Y\ar[r]^\beta\ar[d]&Z\ar[r]\ar[d]^\gamma&0\\
0\ar[r]&X\ar[r]&E\ar[r]&\Si X\ar[r]&0}
\] 
such that $E$ is injective. A triangle in $\sfS(\sfA)$ is by definition \emph{exact} if it
isomorphic to a sequence of morphisms 
\[ 
X\xra{\alpha} Y\xra{\beta} Z\xra{\gamma}\Si X
\] 
as above. 

\begin{proposition} 
\index{Frobenius category!triangulated}
The stable category of a Frobenius category is triangulated.
\end{proposition}

\begin{proof} 
It is easy to verify the axioms, once one observes that every morphism in $\sfS(\sfA)$ can be represented by an admissible mono in $\sfA$. Note that a homotopy cartesian square can be represented by a pull-back and push-out square. This gives a proof for (T4$'$). We refer to \cite[\S{I.2}]{Happel:1988a} for details.
\end{proof}

\begin{example}
\label{ex:frobstable}
(1) Let $A$ be a finite dimensional and self-injective algebra. Then the category of $A$-modules $\Mod A$ is a Frobenius category, and we write $\StMod A$ for the stable category $\sfS(\Mod A)$.

(2) The category of complexes $\sfC(\sfA)$ of an additive category $\sfA$ is a Frobenius category with respect to the degreewise split exact sequences. The morphisms factoring through an injective object are precisely the null-homotopic morphisms. Thus the stable category of $\sfC(\sfA)$ coincides with the homotopy category $\sfK(\sfA)$. Note that the triangulated structures which have been defined via mapping cones and via exact sequences in $\sfC(\sfA)$ coincide.
\end{example}

\subsection{Exact and cohomological functors}

An \emph{exact functor}\index{functor!exact} $\sfT\to\sfU$ between triangulated categories is a
pair $(F,\eta)$ consisting of a functor $F\colon\sfT\to\sfU$ and a natural
isomorphism $\eta\colon F\comp \Si_\sfT\to\Si_\sfU\comp F$ such that for
every exact triangle $X\xra{\alpha}Y\xra{\beta}Z\xra{\gamma}\Si X$ in $\sfT$ the
triangle
\[
FX\xra{F\alpha}FY\xra{F\beta}FZ\xra{\eta_X\comp F\gamma}\Si (FX)
\] is exact in $\sfU$.

\begin{example}
An additive functor $\sfA\to\sfB$ induces an exact functor  $\sfK(\sfA)\to\sfK(\sfB)$.
\end{example}

A functor $\sfT\to\sfA$ from a triangulated category $\sfT$ to an abelian category $\sfA$ is  \emph{cohomological}\index{functor!cohomological} if it sends each exact triangle in $\sfT$ to an exact sequence in $\sfA$.

\begin{example}
 For each object $X$ in $\sfT$, the representable functors
\[\Hom_\sfT(X,-)\colon\sfT\to\Ab\quad\textrm{and}\quad
\Hom_\sfT(-,X)\colon\sfT^\op\to\Ab\]
into the category $\Ab$ of abelian groups are cohomological functors.
\end{example}

\subsection{Thick subcategories}
\label{ss:tria}
Let $\sfT$ be a triangulated category. A non-empty full subcategory $\sfS$ is a
\emph{triangulated subcategory}\index{triangulated category!subcategory} if the following conditions hold.
\begin{enumerate}
\item[(S1)] $\Si^n X\in\sfS$ for all $X\in\sfS$ and $n\in\bbZ$.
\item[(S2)] Let $X\to Y\to Z\to\Si X$ be an exact triangle in $\sfT$.
  If two objects from $\{X,Y,Z\}$ belong to $\sfS$, then also the third.
\end{enumerate}
A triangulated subcategory $\sfS$ is \emph{thick}\index{triangulated category!thick subcategory} if in addition the following condition holds.
\begin{enumerate}
\item[(S3)] Every direct factor of an object in $\sfS$ belongs to $\sfS$,
that is, a decomposition $X=X'\amalg X''$ for $X\in\sfS$ implies
$X'\in\sfS$.
\end{enumerate}
A triangulated subcategory $\sfS$ inherits a canonical triangulated structure from $\sfT$. 

\begin{example}
Let $F\colon\sfT\to\sfU$ be an exact functor between triangulated categories. Then the full subcategory $\Ker F$, called the \emph{kernel}\index{functor!kernel} of $F$, consisting of the objects annihilated by $F$ forms a thick subcategory of $\sfT$.
\end{example}

\subsection{Derived categories}

Let $\sfA$ be an abelian category. Given a complex 
\[
\cdots \to X^{n-1}\xra{d^{n-1} }X^n\xra{d^n}X^{n+1}\to\cdots
\]
in $\sfA$, the \emph{cohomology}\index{cohomology} in degree $n$ is by definition the object 
\[
H^nX=\Ker d^n/\Im d^{n-1}.
\] 
A morphism $\phi\colon X\to Y$ of complexes induces, for each $n\in\bbZ$, a morphism 
\[
H^n\phi\colon H^nX\to H^nY\,;
\]
if these are all isomorphisms, then $\phi$ is said to be a \emph{quasi-isomorphism}\index{quasi-isomorphism}. Note that morphisms $\phi,\psi\colon X\to Y$ are homotopic, then $H^n\phi=H^n\psi$ for all $n$.

The \emph{derived category}\index{derived category} $\sfD(\sfA)$ of $\sfA$ is obtained from $\sfK(\sfA)$ by formally inverting all quasi-isomorphisms. To be precise, one defines
\[
\sfD(\sfA)=\sfK(\sfA)[S^{-1}]
\] 
as the localisation of $\sfK(\sfA)$ with respect to the class $S$ of all quasi-isomorphisms.  

\begin{proposition}
The derived category $\sfD(\sfA)$ carries a unique triangulated structure such that the canonical functor $\sfK(\sfA)\to\sfD(\sfA)$ is
exact.\qed
\end{proposition}

We identify any object $X$ in $\sfA$ with the complex having $X$ concentrated in degree zero. This yields a functor $\sfA\to\sfD(\sfA)$ which induces for  all objects $X,Y$ in $\sfA$ and $n\in\bbZ$ an
isomorphism
\[
\Ext^n_\sfA(X,Y)\xra{\sim} \Hom_{\sfD(\sfA)}(X,\Si^nY).
\]
For example, the functor sends each exact sequence $\eta \colon 0\to A\xra{\alpha} B\xra{\beta} C\to 0$ in $\sfA$ to an exact triangle $A\xra{\alpha} B\xra{\beta} C\xra{\gamma}\Si A$. The above isomorphism maps the class in $\Ext^1_\sfA(C,A)$ representing $\eta$ to $\gamma\colon C\to\Si A$.

\subsection{Compact objects}

Let $\sfT$ be a triangulated category and suppose that $\sfT$ admits set-indexed coproducts.  A \emph{localising subcategory}\index{localising subcategory} of $\sfT$ is a full triangulated subcategory that is closed under taking
coproducts.  We write $\Loc_\sfT(\sfC)$\index{Loc@$\Loc_\sfT(\sfC)$} for the smallest localising subcategory containing a given class of objects $\sfC$ in $\sfT$, and call it the localising subcategory \emph{generated}\index{localising subcategory!generated} by $\sfC$. In the same vein, we write $\Thick_{\sfT}(\sfC)$\index{Thick@$\Thick_\sfT(\sfC)$} for the smallest thick subcategory containing $\sfC$, and call it the thick subcategory generated by $\sfC$.

An object $X$ in $\sfT$ is \emph{compact}\index{compact object} if the functor $\Hom_{\sfT}(C,-)$ commutes with all coproducts. This means that each morphism $X\to\coprod_{i\in I}Y_i$ in $\sfT$  factors through $X\to\coprod_{i\in
  J}Y_i$ for some finite subset $J\subseteq I$. We write $\sfT^{\sfc}$ for the full subcategory of compact objects in $\sfT$. Note that $\sfT^c$ is a thick subcategory of $\sfT$.

The category $\sfT$ is \emph{compactly generated}\index{compactly generated}\index{triangulated category!compactly generated} if it is generated by a set of compact objects, that is, $\sfT=\Loc_\sfT(\sfC)$ for some set $\sfC\subseteq\sfT^c$. The following result provides a useful criterion for compact
generation. The proof uses the Brown representability theorem, which we will
learn about in Section~\ref{sec:Tuesday2}.

\begin{proposition}
\label{pr:compact-generation}
Let $\sfC$ be a set of compact objects of $\sfT$. Then $\Loc_\sfT(\sfC)=\sfT$ if and only if for each non-zero object $X\in\sfT$ there are $C\in\sfC$ and $n \in\bbZ$ such that $\Hom_\sfT(\Si^nC,X)\neq 0$.
\end{proposition}

\begin{proof}
Assume $\Loc_\sfT(\sfC)=\sfT$ and fix an object $X\in\sfT$. The objects $V\in\sfT$ satisfying $\Hom_\sfT(\Si^nV,X)= 0$ for all $n\in\bbZ$ form a localising subcategory of $\sfT$. If this localising subcategory contains $\sfC$, then $X=0$.

As to the converse, Corollary~\ref{cor:right-adjoint-exist} yields that the inclusion $\Loc_\sfT(\sfC)\to\sfT$ admits a right adjoint; we denote this by $\gam$. Given any object $X\in\sfT$, this yields a universal
morphism $\gam X\to X$. Completing this to an exact triangle $\gam X\to X\to X'\to$ produces an object $X'$ satisfying $\Hom_\sfT(V,X')=0$ for all $V\in\Loc_\sfT(\sfC)$. Thus $X'=0$ and
  therefore $X$ belongs to $\Loc_\sfT(\sfC)$.
\end{proof}

\begin{example}
\label{ex:stmodA}
(1) Let $A$ be any ring and $\sfD(\Mod A)$\index{Dmoda@$\sfD(\Mod A)$} its derived category. Since $\Hom_{\sfD(\Mod A)}(\Si^{n}A,X)=H^{-n}(X)$, it follows that $A$ viewed as complex concentrated in degree zero is a compact object; it is also a generator, by Proposition~\ref{pr:compact-generation}. Thus the derived category is compactly generated. The compact objects are described in Theorem~\ref{thm:compacts-DA}

(2) Let $A$ be a finite dimensional and self-injective algebra. Then
$\Mod A$ is a Frobenius category and the corresponding stable category
$\StMod A$ is compactly generated. An object is compact if and only if
it is isomorphic to a finitely generated $A$-module. Thus the
inclusion $\mod A\to \Mod A$ induces an equivalence $\stmod
A\xra{\sim}(\StMod A)^c$.

Indeed, $\Mod A$ with exact structure given by exact sequences of
modules is a Frobenius category, with projectives the projective
$A$-modules; see Example~\ref{ex:frobstable}. Its stable category is
thus triangulated with suspension $\Omega^{-1}$, by the discussion in
Section~\ref{ss:stab}. The simple $A$-modules form a set of compact
generators. This follows from
Proposition~\ref{pr:compact-generation}. The thick subcategory
generated by all simple $A$-modules coincides with $\stmod A$, and
this yields the description of the compact objects of $\StMod A$, by
Theorem~\ref{thm:neeman-compacts} below.
\end{example}

The result below, attributed to Ravenel~\cite{Ravenel:1984}, is
extremely useful in identifying compact objects in triangulated
categories; see \cite[Lemma~2.2]{Neeman:1992b} for a proof.

\begin{theorem}
\label{thm:neeman-compacts}
Let $C$ and $D$ be compact objects in a triangulated category $\sfT$. If $D$ is in $\Loc_{\sfT}(C)$, then it is already in $\Thick_{\sfT}(C)$. In particular, if $C$ is a compact generator for $\sfT$, then the class of compact objects in $\sfT$ is precisely $\Thick_{\sfT}(C)$. \qed
\end{theorem}

\section{Exercises}
\label{exer:Monday}
In the following exercises $k$ is a field, $G$ a finite group, and $M,N$ are $kG$-modules. Keep in mind that the  $kG$ action on $M\otimes_kN$ is via the diagonal. In what follows $\Char k$ denotes the characteristic of $k$.

\begin{enumerate}[\quad\rm(1)]

\item Let $H$ be a finite group. Prove that the canonical homomorphisms $G\to G\times H$ and $H\to G\times H$  of groups induce an isomorphism of $k$-algebras:
\[
kG\otimes_kkH \xra{\ \cong\ }k[G\times H]\,.
\]

\item Let $G=\langle g_1,\dots,g_r\rangle\cong (\bbZ/p)^r$ and set  $x_i=g_i-1$, in $kG$. Prove that if $\Char k=p$, then $kG$ is isomorphic as a $k$-algebra to 
\[
k[x_1,\dots,x_r]/(x_1^p,\dots,x_r^p)\,.
\]

\item 
Describe the centre of the ring $kG$, for a general group $G$.

\item Let $\pi\col kG\to k$ be the $k$-linear map defined on the basis $G$ by $\pi(1)=1$ and $\pi(g)=0$ for $g\ne 1$. Verify that the following map is a $kG$-linear isomorphism. 
\[
kG\to \Hom_k(kG,k)\quad \text{where $g\mapsto [h\mapsto \pi(g^{-1}h)]$}\,.
\]
This proves that $kG$ is a self-injective algebra.

\item Verify that the following maps are $kG$-linear isomorphisms.
\begin{alignat*}{2}
&M \to k\otimes_k M& &\text{where $m\mapsto 1\otimes m$}\,;\\
&M\da_1\ua^G \to kG\otimes_kM& \quad &\text{where $g\otimes m\mapsto g\otimes gm$}\,.
\end{alignat*}
The second isomorphism implies that $kG\otimes_kM$ is a free $kG$-module.

\item Verify that the following maps are $kG$-linear monomorphisms:
\begin{alignat*}{2}
&M\to kG\otimes_kM&\quad &\text{where $m\mapsto \sum_{g\in G}g\otimes m$}\,;\\
&M\to M\da_1\ua^G& \quad &\text{where $m\mapsto \sum_{g\in G}g\otimes g^{-1}m$}\,.
\end{alignat*}
Since $M\da_1\ua^G$ is free, it follows that each module embeds into a free one, in a canonical way.

\item Prove that a $kG$-module $M$ is projective if and only if it is
  injective.  Hint: use (4)---(6).  It is also true that $M$
  is projective if and only if it is flat.

\item 
Let $H$ be a subgroup of $G$. Prove that $kG$ is free as a $kH$-module, both on the left and on the right, and describe bases.

\item Let $G$ be a finite $p$-group and $\Char k =p$. For any non-zero element $m\in M$ the $\bbF_p$-subspace of $M$ spanned by $\{gm\mid g\in G\}$ is finite dimensional, and so has $p^n$ elements for some $n$. Show that some non-zero element of this set is fixed by $G$. Deduce that the trivial module is the only simple $kG$-module.

\item Let $G=(\bbZ/p)^r$ and $\Char k=p$. Describe $J(kG)$, the Jacobson radical\index{JkG@$J(kG)$} of $kG$, and show that $J(kG)/J^2(kG)$ is a vector space of dimension $r$ over $k$. Prove that there is a natural isomorphism of $k$-vector spaces
\[
H^1(G,k) \cong \Hom_{k}(J(kG)/J^2(kG),k)\,.
\]

\item Let $G=\langle g\mid g^{p^n}=1\rangle\cong\bbZ/p^n$ with $n>1$ and $\Char k=p$. Use Jordan canonical form to show that a finitely generated $kG$-module is free if and only if its restriction to the subgroup
\[ 
H=\langle g^{p^{n-1}}\rangle\cong \bbZ/p 
\]
is free. This is a case of Chouinard's theorem.

\medskip

In the following exercises, assume the field  \textbf{$k$ is algebraically closed}.

\item Write $(\bbZ/2)^2=\langle g_1,g_2\rangle$ and assume $\Char k=2$. For each $\lambda\in k$, compute the rank variety of the following module $M_\lambda$:
\begin{gather*} 
g_1\mapsto \begin{pmatrix} 1&0\\1&1\end{pmatrix}\qquad
g_2\mapsto\begin{pmatrix}1&0\\\lambda&1\end{pmatrix} 
\end{gather*}
Deduce that the $M_\lambda$ are non-isomorphic for different values of $\lambda$. 

\item Write $(\bbZ/3)^2=\langle g_1,g_2\rangle$ and assume $\Char k=3$. For each $\lambda\in k$, compute the rank variety of the following module $M_\lambda$:
\begin{gather*} 
g_1\mapsto \begin{pmatrix} 1&0&0\\1&1&0\\0&1&1\end{pmatrix}\qquad
g_2\mapsto\begin{pmatrix}1&0&0\\\lambda&1&0\\0&\lambda&1\end{pmatrix} 
\end{gather*}
Prove that the $M_\lambda$ are non-isomorphic for different values of $\lambda$.

\item Generalise the last question to $(\bbZ/p)^2$ in characteristic $p$, and deduce that there are infinitely many isomorphism classes of $p$-dimensional modules.

\item Write $(\bbZ/2)^4=\langle g_1,g_2,g_3,g_4\rangle$ and assume $\Char k=2$. Compute the rank variety of the following module:
\begin{gather*} 
g_1 \mapsto \begin{pmatrix} 1&0&0&0\\0&1&0&0\\1&0&1&0\\0&0&0&1 \end{pmatrix}
\qquad
g_2 \mapsto \begin{pmatrix} 1&0&0&0\\0&1&0&0\\0&1&1&0\\0&0&0&1 \end{pmatrix}
\\
g_3 \mapsto \begin{pmatrix} 1&0&0&0\\0&1&0&0\\0&0&1&0\\1&0&0&1 \end{pmatrix} 
\qquad
g_4 \mapsto \begin{pmatrix} 1&0&0&0\\0&1&0&0\\0&0&1&0\\0&1&0&1 \end{pmatrix}. 
\end{gather*}

\item Let $G=\langle g_1,\dots,g_r\rangle\cong(\bbZ/2)^r$ for some $r>1$ and assume $\Char k=2$. Set $K=k(t_1,\dots,t_r)$, a transcendental extension of $k$ and let $M=K\oplus K$ as a $k$-vector space, with $G$-action given by
\[
g_i\mapsto \begin{pmatrix}I&0\\t_iI&I\end{pmatrix}\,.
\]
Prove that $M$ is not free, but that its restriction to every cyclic
shifted subgroup of $G$ is free.\smallskip

\noindent Hint for the first part: if $i\ne j$
then $(g_i-1)(g_j-1)$ acts as zero.
 
\item 
Let $\sfT$ be a triangulated category. Prove that for each object $X$ in $\sfT$, the contravariant representable functor $\Hom_{\sfT}(-,X)$ and the covariant representable functor $\Hom_{\sfT}(X,-)$ are  cohomological.

\item Prove that the class of compact objects in a triangulated category admitting arbitrary coproducts form a thick subcategory.

\item 
Prove: In a triangulated category any coproduct of exact triangles is exact. 

\item
Prove that given an adjoint pair of functors between triangulated categories, one of the functors is exact if and only if the other is.

\item Let $\sfA$ be an additive category. Prove that $\sfA$ is a
  Frobenius category with exact structure given by the split short
  exact sequences in $\sfA$.

\item Let $\sfA$ be an additive category and $\sfC(\sfA)$ the category
  of complexes over $\sfA$, with morphisms the degree zero chain
  maps. Prove that $\sfC(\sfA)$ is a Frobenius category, with exact
  structure given by the degreewise split exact sequences.\smallskip

 \noindent Hint: One can deduce this from general principles as follows. The
  functor $F\colon \sfC(\sfA)\to\prod_{n\in\bbZ}\sfA$ forgetting the
  differential is exact. Here we use the exact structure on
  $\prod_{n\in\bbZ}\sfA$ coming from the additive structure. The
  functor $F$ has a left adjoint $F_\lambda$ and a right adjoint
  $F_\rho$. Thus $F_\lambda$ sends projectives to projectives, while
  $F_\rho$ sends injectives to injectives. For each complex $X$, the
  counit $F_\lambda F X\to X$ is an admissable epi, and the unit $X\to
  F_\rho FX$ is an admissable mono.

\item Let $F\colon\sfT\to\sfU$ be an exact functor between triangulated categories that admit set-indexed coproducts. Suppose also that $\sfT$ is compactly generated by a compact object $C$. Then $F$ is an equivalence of triangulated categories if and only if $F$ induces a bijection
\[
\Hom_\sfT(C,\Si^n C)\xra{\sim}\Hom_\sfU(FC,\Si^n FC)
\]
for all $n\in\bbZ$, and $\Loc_\sfU(FC)=\sfU$.\smallskip

\noindent Hint: See~\cite[Lemma~4.5]{Benson/Iyengar/Krause:bik3}.
\end{enumerate}
 
\chapter{Tuesday}
\thispagestyle{empty}
The highlights of this chapter are Hopkins' theorem on perfect complexes over commutative noetherian ring, which is the content of Section~\ref{sec:Tuesday1}, and its analogue in modular representation theory, proved by Benson, Carlson, and Rickard; this appears in Section~\ref{sec:Tuesday3}. This requires a discussion of appropriate notions of support; for commutative rings, this is based on the material from \ref{Appendix}, while for modules over group algebras, one requires more sophisticated tools, from homotopy theory, and these are discussed in Section~\ref{sec:Tuesday2}.

\section{Perfect complexes over commutative rings}
\label{sec:Tuesday1}
In this lecture $A$ will denote a commutative noetherian ring. Good examples to bear in mind are polynomial rings over fields, their quotients, and localisations. We write $\Mod A$ \index{ModA@$\Mod A$} for the category of $A$-modules, and $\mod A$ \index{modA@$\mod A$} for its full subcategory of finitely generated $A$-modules. 

Keeping with the convention in these notes, complexes of $A$-modules will be upper graded:
\[
\cdots \to M^{i-1}\xra{d^{i-1}} M^{i}\xra{d^{i}} M^{i+1}\to \cdots
\]
The $n$th \emph{suspension}\index{suspension} of $M$ is denoted $\Si^{n}M$; thus for each $i\in\bbZ$ one has
\[
(\Si^{n}M)^{i} = M^{n+i}\quad\text{and}\quad d^{\Si^{n}M} = (-1)^{n}d^{M}\,.
\]
We write $\HH *M$ for the total cohomology $\bigoplus_{n\in\bbZ}\HH nM$ of such a complex. It is viewed as a graded module over the ring $A$.

Let $\sfD(A)$\index{Dmoda@$\sfD(A)$} denote the derived category of the category of $A$-modules, viewed as a triangulated category with $\Si$ as the translation functor. Let $\sfD^{\sfb}(\mod A)$ \index{Dbmoda@$\sfD^{\sfb}(\mod A)$} denote the bounded derived category of finitely generated $A$-modules. Sometimes it is convenient to identify $\sfD^{\sfb}(\mod A)$ with the full subcategory
\[
\{M\in \sfD(A)\mid \text{the $A$-module $\HH *M$ is finitely generated}\}
\]
of $\sfD(A)$ consisting of complexes with finitely generated cohomology.

\subsection{Perfect complexes}

We focus on thick subcategories of $\sfD^{\sfb}(\mod A)$. Recall that a non-empty subcategory $\sfC\subseteq \sfD^{\sfb}(\mod A)$ is said to be \emph{thick} \index{triangulated category!thick subcategory} if it is closed under exact triangles, suspensions, and retracts.

Our interest in thick subcategories of $\sfD^{\sfb}(\mod A)$ stems from the fact that they are kernels of exact and cohomological functors. In addition most ``reasonable'' homological properties of complexes give rise to thick subcategories; you may take this as the definition of a reasonable property. A central problem which motivates the techniques and results discussed in this lectures is:

\medskip

\textbf{Problem:} Classify the thick subcategories of $\sfD^{\sfb}(\mod A)$.

\medskip

In the course of these lectures you will encounter solutions to this problem for important classes of rings. In this lecture, we present a classification theorem for a small piece of the bounded derived category of a commutative ring.

\begin{definition}
Let $M$ be a complex of $A$-modules. We write $\Thick(M)$\index{thick@$\Thick(M)$} for the smallest (with respect to inclusion) thick subcategory of $\sfD(A)$ containing $M$. 

An intersection of thick subcategories is again thick, so $\Thick(M)$ equals the intersection of all thick subcategories of $\sfD(A)$ containing $M$. The objects in $\Thick(M)$ are the complexes that can be \emph{finitely built}\index{finitely built} out of $M$ using finite direct sums, suspensions, exact triangles, and retracts. See \cite[\S2]{Avramov/Buchweitz/Iyengar/Miller:2010a} for a more precise and constructive definition. 

When $H^{*}(M)$ is finitely generated, $\Thick(M)$ is contained in $\sfD^{\sfb}(\mod A)$, for the latter is a thick subcategory of $\sfD(A)$.
\end{definition}

\begin{theorem}
\label{thm:compacts-DA}
Let $M$ be in $\sfD(A)$. The following conditions are equivalent:
\begin{enumerate}[\quad\rm(1)]
\item
$M$ is in $\Thick(A)$.
\item
$M$ is isomorphic in $\sfD(A)$ to a bounded complex $0\to P^{i}\to\cdots \to P^{s}\to 0$ of finitely generated projective $A$-modules.
\item $M$ is a compact object in $\sfD(A)$.
\end{enumerate}
\end{theorem}

In commutative algebra any bounded complex of projective modules, as in (2), is said to be \emph{perfect}\index{perfect complex}. We apply this terminology to any complex isomorphic to such a complex, and so, given the result above, speak of $\Thick(A)$ as the subcategory of perfect complexes in $\sfD(A)$.

\begin{proof}
Note that $\Hom_{\sfD(A)}(A,X)=H^{0}(X)$ for any complex $X$ of $A$-modules.

(1) $\implies$ (3) It follows from the identification above that $A$ is a compact object, for $H^{0}(-)$ commutes with arbitrary coproducts. The class of compact objects in any triangulated category, in particular, in $\sfD(A)$, form a thick subcategory---this was an exercise in Chapter~\ref{ch:Monday}. Thus $\Thick(A)$ consists of compact objects.

(3) $\implies$ (1) Since $A$ is compact in $\sfD(A)$ and generates it (see Example~\ref{ex:stmodA}), this implication follows from Theorem~\ref{thm:neeman-compacts}.

(1) $\implies$ (2) Using the Horseshoe Lemma~\cite[Lemma~2.2.8]{Weibel:1994}, it is a straightforward exercise to prove that the complexes as in (2) from a thick subcategory. Since $A$ is evidently contained in it, so is all of $\Thick(A)$.

(2) $\implies$ (1) Since $\Thick(A)$ is closed under finite sums, it contains $A^{n}$ for each $n\ge 1$. Since it is closed also under direct summands, it contains all finitely generated projectives, and also their shifts. 

Any complex of the form $P:= 0\to P^{i}\to\cdots \to P^{s}\to 0$ fits into an exact sequence of $A$-modules
\[
0\to \Si^{-s}P^{s}\to P\to P^{\les s-1}\to 0
\]
Thus, an induction on the number of non-zero components of $P$ yields that $P$, and hence anything isomorphic to it, is in $\Thick(A)$.
\end{proof}

\begin{remark}
\label{rem:auslander-buchsbaum-serre}
There is an inclusion $\Thick(A)\subseteq \sfD^{\sfb}(\mod A)$. In view of the implication (1) $\implies$ (2) of the preceding result, equality holds if and only if each $M$ in $\sfD^{\sfb}(\mod A)$ has a finite projective resolution. The latter condition is equivalent to the statement that the ring $A$ is \emph{regular}\index{regular ring}, by a theorem of Auslander, Buchsbaum, and Serre\index{Auslander, Buchsbaum, Serre!theorem}; see~\cite[Section~19]{Matsumura:1986} for the definition of a regular ring, and for a proof of that theorem. We note only that fields are regular rings, and the regularity property is inherited by localisations and polynomial extensions.
\end{remark}

We are heading towards a description of the thick subcategories of $\Thick(A)$. This involves a notion of support for complexes; what is described below is an extension of ``big support'' of modules defined in~\ref{Appendix}

\subsection{Support}

Recall that the \emph{Zariski spectrum}\index{Zariski spectrum} of the ring $A$ is the set
\[
\Spec A =\{\fp\subseteq A\mid\text{$\fp$ is a prime ideal} \}
\]
with topology where the closed sets are $\mcV(\fa)=\{\fp\in\Spec A\mid \fp\supseteq \fa\}$, where $\fa$ is an ideal in $A$. For each $M$ in $\sfD(A)$ we set
\[
\Supp_{A}M = \{\fp\in\Spec A\mid \HH*{M_{\fp}}\ne 0\}
\]
Note that passing to cohomology commutes with localisation: $\HH*{M_{\fp}} = {\HH *M}_{\fp}$. The following properties of are readily verified. 

\begin{enumerate}[\quad\rm(1)]
\item For each $M\in \sfD^{\sfb}(\mod A)$ one has
\[
\Supp_{A}M = \Supp_{A}\HH *M = \mcV(\ann_{A}\HH *M)\,.
\]
In particular, $\Supp_{A}M$ is a closed subset of $\Spec A$.
\item
$\Supp_{A}M=\varnothing$ if and only if $\HH *M=0$, that is to say, if $M=0$ in $\sfD(A)$.
\item
$\Supp_{A}(M\oplus N) = \Supp_{A}M\, \bigcup\, \Supp_{A}N$
\item
Given an exact triangle $L\to M\to N\to \Si L$ in $\sfD(\mod A)$ one has 
\[
\Supp_{A}M \subseteq \Supp_{A}L\,\cup\, \Supp_{A}N\,.
\]
Moreover, $\Supp_{A}(\Si M)=\Supp_{A}M$.
\item
$\Supp_{A}(M\lotimes_{A}N) = \Supp_{A}M\, \cap\, \Supp_{A}N$ for all $M,N$ in $\sfD^{\sfb}(\mod A)$.
\end{enumerate}
Compare with properties of support for group representations in Section~\ref{ssec:Rank varieties}.

A routine verification, using the properties of support above, yields:

\begin{lemma}
\label{lem:support-thicksubcat}
For any subset $\mcU$ of $\Spec A$, the full subcategory 
\[
\sfC_{\mcU} = \{M\in \sfD^{\sfb}(\mod A)\mid \Supp_{A}M\subseteq \mcU\}
\]
of $\sfD^{\sfb}(\mod A)$ is a thick subcategory. \qed
\end{lemma}

\subsection{Koszul complexes}
Let $a$ be an element in the ring $A$. The \emph{Koszul complex}\index{Koszul complex} on $a$ is the complex 
\[
\cdots \to 0\to A\xra{\ a\ }A\to 0\cdots \to
\]
of $A$-modules, where the non-zero components are in degree $-1$ and $0$; remember that the grading is cohomological. One can view it  as the \emph{mapping cone}\index{mapping cone} of the morphism of complexes $A\xra{a}A$ where $A$ is viewed as a complex, in the usual way. 

The Koszul complex on $a$ with coefficients in a complex $M$ is the mapping cone of the morphism $M\xra{a} M$, namely, the graded module
\[
\Si M \oplus M \quad\text{with differential} \quad
\begin{bmatrix}
-d^{M} & 0 \\
a       & d^{M}
\end{bmatrix}
\]
In anticipation of analogous constructions in triangulated categories we denote this complex $\kos Ma$. One then has an \emph{mapping cone exact sequence} of complexes
\[
0\to M\to \kos Ma \to \Si M\to 0\,.
\]
The Koszul complex on a sequence of elements $\bsa = a_{1},\dots,a_{n}$ in $A$ is defined by iterating this construction: $\kos M{\bsa}=M_{n}$, where $M_{0}=M$ and
\[
M_{i} = \kos {M_{i-1}}{a_{i}}\quad\text{for $i\geq 1$.}
\]
It is not hard to verify that there is an isomorphism of complexes
\[
\kos M{\bsa}\cong (\kos M{a_{1}})\otimes_{A}\cdots \otimes_{A}(\kos M{a_{n}})\,.
\]
Indeed, this is one way to construct the Koszul complex; see \cite[\S16]{Matsumura:1986} for details.

Finally, given an ideal $\fa$ of $A$, we let $\kos M{\fa}$ denote a Koszul complex on some generating set $a_{1},\dots,a_{n}$ of $\fa$. It does depend on $\bsa$; see,  however, Lemma~\ref{lem:kosloc}.

\begin{proposition}
\label{prop:kosca}
Let $\fa$ be an ideal in $A$ and $M$ a complex of $A$-modules. The following statements hold.
\begin{enumerate}[\quad\rm(1)]
\item 
$\kos M{\fa}$ is in $\Thick(M)$; in particular, if $M$ is perfect, so is $\kos M{\fa}$.
\item
$\Supp_{A}(\kos M{\fa})=\Supp_{A}M\cap \mcV(\fa)$ when $M$ is in $\sfD^{\sfb}(\mod A)$.
\item
$\Supp_{A}(\kos A{\fa}) = \mcV(\fa)$.
\end{enumerate}
\end{proposition}

\begin{proof}
It is clear from the mapping cone exact sequence that $M$ finitely builds $\kos Ma$, for any element $a\in A$.
An iteration yields (1). Moreover, (3) is a special case of (2).

As to (2), it suffices to consider the case of one element $a$. Then the mapping cone sequence yields in cohomology an exact sequence of graded $R$-modules:
\[
0\to H^{*}(M)/aH^{*}(M) \to H^{*}(\kos Ma)\to (0:_{H^{*}(M)}a)\to 0
\]
The module on the right is the submodule of $H^{*}(M)$ annihilated by $(a)$; in particular, it is supported on $\Supp_{R}M\cap\mcV(a)$. Moreover, since $H^{*}(M)$ is finitely generated, Nakayama's Lemma yields the first equality below:
\[
\Supp_{R}(H^{*}(M)/aH^{*}(M)) = \Supp_{R}H^{*}(M)\cap \mcV(a)=\Supp_{R}M\cap\mcV(a)\,.
\]
Thus the exact sequence above yields $\Supp_{R} H^{*}(\kos Ma)=\Supp_{R}M\cap \mcV(a)$.
\end{proof}

\subsection{The theorem of Hopkins}

We are now ready to discuss Hopkins' theorem on perfect complexes.

\begin{lemma}
\label{lem:con-hopkins}
Let $N$ be a complex in $\sfD^{\sfb}(\mod A)$. One then has $\Supp_{A}M\subseteq\Supp_{A}N$ for each $M$ in $\Thick(N)$.
\end{lemma}

\begin{proof}
Apply Lemma~\ref{lem:support-thicksubcat} with $\mcU=\Supp_{A}N$, noting that $N$ is in $\sfC_{\mcU}$. 
\end{proof}

The converse of the preceding result does not hold in general.

\begin{example}
Let $k$ be a field and set $A=k[x]/(x^{2})$. One then has 
\[
\Supp_{A}k = \Spec A= \Supp_{A}A\,;
\]
however, $k$ is not in $\Thick(A)$ because, for example, $k$ does not admit a finite projective resolution.
\end{example}

It is a remarkable result, proved by Hopkins~\cite{Hopkins:1987a}, see also Neeman~\cite{Neeman:1992a}, that for perfect complexes Lemma~\ref{lem:con-hopkins} has a perfect converse.

\begin{theorem}[Hopkins]
\label{thm:hopkins}
\index{Hopkins' theorem}
If $M$ and $N$ are perfect complexes and there is an inclusion $\Supp_{A}M\supseteq\Supp_{A}N$, then $M\in \Thick(N)$.\qed
\end{theorem}

A proof of this result, following an idea of Neeman, will be given in Section~\ref{sec:Friday1}. For other proofs see \cite{Hopkins:1987a}, \cite{Neeman:1992a}, and also \cite{Iyengar:2006}.

A caveat: In Hopkins' theorem it is crucial that both $M$ and $N$ are perfect.  Support over $A$ cannot detect if one complex is built out of another, even for complexes in $\sfD^{\sfb}(\mod A)$; see the exercises at the end of this chapter.

\begin{corollary}
When $M$ is a perfect complex, $\Thick(M)=\Thick(\kos A{\fa})$ for any ideal $\fa$ with $\mcV(\fa)=\Supp_{A}M$. \qed
\end{corollary}

We now explain how Hopkins' theorem gives a classification of the thick subcategories of perfect complexes. 
Recall that a subset $\mcV$ of $\Spec A$ is said to be \emph{specialisation closed}\index{specialisation!closed} if it has the property that if $\fq\supseteq \fp$ are prime ideals in $A$ and $\fp$ is in $\mcV$, then so is $\fq$. This is equivalent to the condition that $\mcV$ is a arbitrary union of closed sets. For any subcategory $\sfC$ of $\sfD^{\sfb}(A)$, we set
\[
\Supp_{A}\sfC=\bigcup_{M\in\sfC}\Supp_{A}M\,.
\]
This is a specialisation closed subset of $\Spec R$, as each $\Supp_{A}M$ is closed. 

\begin{theorem}
\label{thm:thickA}
The assignment $\sfC\mapsto\Supp_{A}{\sfC}$ gives a bijection
\begin{gather*}
\left\{
\begin{gathered}
\text{Thick subcategories} \\
\text{of perfect complexes}
\end{gathered}
\,\right\}
\longleftrightarrow
\left\{
\begin{gathered}
\text{Specialization closed} \\
\text{subsets of $\Spec A$}
\end{gathered}
\,\right\}
\end{gather*}
Its inverse assigns $\mcV$ to the subcategory $\sfC_{\mcV}=\{M\in\Thick(A)\mid \Supp_{A}M\subseteq\mcV\}$.
\end{theorem}
 
\begin{proof}

Let $\sfC$ be a thick subcategory  of $\Thick(A)$ and set $\mcV=\Supp_{A}\sfC$. It is clear that $\sfC \subseteq \sfC_{\mcV}$. For equality, it remains to prove that any perfect complex $M$ satisfying $\Supp_{A}M\subseteq \mcV$ is in $\sfC$. This is a consequence of Theorem~\ref{thm:hopkins}.

Indeed, for $M$ as above, $\Supp_{A}M$ is a closed set in $\Spec A$, hence one can find finitely many complexes in $\sfC$, say $N_{1},\dots,N_{k}$, such that there is an inclusion
\begin{align*}
\Supp_{A}M &\subseteq \bigcup_{i=1}^{k}\Supp_{A}(N_{i})\\
           & = \Supp_{A} N\,, \quad\text{where $N=\bigoplus_{i=1}^{k} N_{i}$}
\end{align*}
The equality holds by properties of support. Since $N$ is also perfect, one obtains $M\in\Thick(N)$, by Theorem~\ref{thm:hopkins}. It remains to note that $\Thick(N)\subseteq \sfC$, for the latter is a thick subcategory and contains $N$.

Conversely, for any specialisation closed subset $\mcV$ of $\Spec A$, there is evidently an inclusion $\Supp_{A}\sfC_{\mcV}\subseteq \mcV$, while equality holds because, for any $\fp\in\mcV$, the Koszul complex $\kos A{\fp}$ satisfies:
\[
\Supp_{A}(\kos A{\fp}) =\mcV(\fp)\subseteq\mcV\,,
\]
and hence $\kos A{\fp}$ is in $\sfC_{\mcV}$.
\end{proof}

\section{Brown representability and localisation}
\label{sec:Tuesday2}
In this section we introduce some technical tools that are important for the rest of this seminar.  
The principal one is the Brown representability which provides an abstract method for constructing objects in a triangulated category. For this it is important to work in a triangulated category that admits set-indexed coproducts.

Localisation is a method to invert morphisms in a category. If the category is triangulated, then each morphism $\sigma\colon X\to Y$ fits into an exact triangle $X\xra{\sigma} Y\to Z\to\Si X$, and an exact functor sends $\sigma$ to an invertible morphism if and only the cone $Z$ is annihilated. Thus we can think of a localisation functor either as a functor that inverts certain morphisms or as one that annihilates certain objects. The objects that are killed by a localisation functor form a localising subcategory.

The first systematic treatment of localisation can be found in
\cite{Gabriel/Zisman:1967a}. We follow closely the exposition in
\cite[\S3]{\bik:2008a} and refer to \cite{Krause:2007a} for further details.

\subsection{Brown representability}

The following result is known as \emph{Brown representability theorem}\index{Brown representability!theorem} and is due to Keller \cite{Keller:1994a} and Neeman \cite{Neeman:1996a}; it is a variation of a classical theorem of Brown \cite{Brown:1962a} from homotopy theory.

\begin{theorem}[Brown]
Let $\sfT$ be a compactly generated triangulated category. For a functor $H\colon \sfT^\op\to\Ab$ the following are equivalent.
\begin{enumerate}
\item The functor $H$ is cohomological and preserves set-indexed coproducts.
\item There exists an object $X$ in $\sfT$ such that $H\cong\Hom_\sfT(-,X)$.\qed
\end{enumerate}
\end{theorem}

The proof is in some sense constructive; it shows that any object in $\sfT$ arises as the homotopy colimit of a sequence of morphisms
\[
X_0\xra{\phi_0}X_1\xra{\phi_1}X_2\xra{\phi_2}\cdots
\] 
such that $X_0$ and the cone of each $\phi_i$ is a coproduct of objects of the form $\Si^n C$, with $n\in\bbZ$ and $C$ an object from the set of compact objects generating $\sfT$.

Here are some useful consequences of the Brown representability
theorem. The description of compact generation in
Proposition~\ref{pr:compact-generation} is another application.

\begin{corollary}
Each compactly generated triangulated category admits set-indexed products. 
\end{corollary}

\begin{proof}
Given a set of objects $\{X_{i}\}_{i\in I}$ in such a category $\sfT$, apply the Brown representability theorem to the functor $Y\mapsto \prod_{i\in I}\Hom_{\sfT}(Y,X_{i})$; check that it is cohomological and takes set-indexed coproducts in $\sfT$ to products in $\Ab$.
\end{proof}

\begin{corollary}
\label{cor:right-adjoint-exist}
Let $F\colon\sfT\to\sfU$ be an exact functor between triangulated
categories and suppose that $\sfT$ is compactly generated. Then $F$
preserves set-indexed coproducts if and only if $F$ admits a right
adjoint.
\end{corollary}

\begin{proof}
If $F$ preserves set-indexed coproducts, then for each object
$X\in\sfU$ the functor $\Hom_\sfU(F-,X)$ is representable by an object
$Y\in\sfT$. Sending $X$ to $Y$ yields a right adjoint of $F$. The
converse is clear, since each left adjoint preserves set-indexed coproducts.  
\end{proof}

\subsection{Localisation functors}

A functor $\bloc\col\sfC\to\sfC$ is called a \emph{localisation functor}\index{functor!localisation} if there exists a morphism $\eta\col\Id_\sfC\to \bloc$ such that the morphism $\bloc\eta\col \bloc\to \bloc^2$ is invertible and
$\bloc\eta=\eta \bloc$. Recall that a morphism $\mu\colon F\to G$ between functors is \emph{invertible}\index{functor!invertible} if and only if for each object $X$ the morphism $\mu{X}\colon FX\to GX$ is an isomorphism.  Note that we only require the existence of $\eta$; the actual morphism is not part of the definition of $L$ because it is determined by $L$ up to a unique
isomorphism $L\to L$.

A functor $\gam\col\sfC\to\sfC$ is called \emph{colocalisation functor}\index{functor!colocalisation} if its opposite functor $\gam^\op\col\sfC^\op\to\sfC^\op$ is a localisation functor. In this case there is a morphism
$\theta\colon \gam\to\Id_\sfC$ such that $\theta\gam\colon\gam\to\gam^2$ is invertible and $\theta\gam=\gam\theta$.

The following lemma provides an alternative description of a localisation functor. Given any class of morphisms
$S\subseteq\Mor\sfC$, the notation $\sfC[S^{-1}]$ refers to the universal category obtained from $\sfC$ by inverting all morphisms in $S$; see \cite{Gabriel/Zisman:1967a}.

\begin{lemma}\label{le:loc}
\pushQED{\qed} 
Let $\bloc\col\sfC\to\sfC$ be a functor and $\eta\col\Id_\sfC\to \bloc$  a morphism.
The following conditions are equivalent.  
\begin{enumerate}[{\quad\rm(1)}] 
\item The morphism $\bloc\eta\col \bloc\to \bloc^2$ is invertible and $\bloc\eta=\eta \bloc$.
\item There exists an adjoint pair of functors $F\col\sfC\to\sfD$ and $G\col\sfD\to\sfC$,
  with $F$ the left adjoint and $G$ the right adjoint, such that $G$ is fully faithful,
  $\bloc=G F$, and $\eta\col\Id_\sfC\to G F$ is the adjunction
  morphism.  
\end{enumerate}
In this case, $F\colon\sfC\to\sfD$ induces an equivalence $\sfC[S^{-1}]\xra{\sim}\sfD$, where 
\[
S=\{\si\in\Mor\sfC\mid F\si\text{ is invertible}\}
 =\{\si\in\Mor\sfC\mid L\si\text{ is invertible} \}.\qedhere 
\]
\end{lemma}

\begin{example}
Let $A$ be a commutative ring and denote by $S^{-1}A$ the localisation with respect to a multiplicatively closed subset $S\subseteq A$. Consider the functors $F\colon\Mod A\to \Mod S^{-1}A$ and $G\colon\Mod S^{-1}A\to
\Mod A$ defined by 
\[
FM=S^{-1}M=M\otimes_AS^{-1}A\quad{and}\quad GN= N\,.
\]
Then the composite $L=GF$ is a localisation functor.
\end{example}

\subsection{Acyclic objects and local objects}

Let $\sfT$ be a triangulated category which admits set-indexed products and coproducts. We write $\Si$ for the suspension functor on $\sfT$. Recall that a localising subcategory of $\sfT$ is a full triangulated subcategory that is closed under taking all coproducts. Analogously, a \emph{colocalising subcategory}\index{colocalising subcategory} of $\sfT$ is a full triangulated subcategory that is closed under taking all products.

The \emph{kernel}\index{functor!kernel} of an exact functor $F\col\sfT\to\sfT$ is the full subcategory
\[
\Ker F=\{X\in\sfT\mid FX=0\}\,,
\]
while the \emph{essential image} \index{functor!essential image} of $F$ is the full subcategory
\[
\Im F=\{X\in\sfT\mid X\cong FY\text{ for some $Y$ in $\sfT$}\}.
\]

Let $L\col\sfT\to\sfT$ be an exact localisation functor with morphism $\eta\colon \Id_\sfT\to L$ such that $L\eta=\eta L$ is invertible. Note that the kernel of $L$ is a localising subcategory of $\sfT$, because it coincides with the kernel of a functor that admits a right adjoint by Lemma \ref{le:loc}. Completing the natural morphism $\eta X\col X\to LX$ yields for each object $X$ in $\sfT$ a natural exact \emph{localisation triangle}\index{localisation triangle}
\begin{equation*}
\gam X\lto X\lto \bloc X\lto
\end{equation*}
This exact triangle gives rise to an exact functor $\gam\col\sfT\to\sfT$ with
\[
\Ker L = \Im\gam \quad\text{and}\quad \Ker\gam =\Im L.
\]
The objects in $\Ker L$ are called \emph{$\bloc$-acyclic}\index{acyclic} and the objects in $\Im \bloc$ are  \emph{$\bloc$-local}\index{local}.

The functor $\gam$ is a colocalisation functor, and each colocalisation functor on $\sfT$ arises in this way. This yields a natural bijection between localisation and colocalisation functors on $\sfT$. Note that $\Ker\gam$ is a colocalising subcategory of $\sfT$.

Given a subcategory $\sfC$ of $\sfT$, consider full subcategories
\begin{align*}
{^\perp}\sfC&=\{X\in\sfT\mid \text{$\Hom_\sfT(X,\Si^{n}Y)=0$ for all $Y\in \sfC$ and $n\in\bbZ$}\}, \\ 
\sfC^{\perp}&=\{X\in\sfT\mid \text{$\Hom_\sfT(\Si^{n}Y,X)=0$ for all $Y\in\sfC$ and $n\in\bbZ$}\}.
\end{align*}
Evidently, ${^\perp}\sfC$\index{cperpr@${^\perp}\sfC$} is a localising subcategory, and $\sfC^\perp$\index{cperpl@$\sfC^{\perp}$} is a colocalising subcategory.

The next lemma summarises  basic facts about localisation and colocalisation.

\begin{proposition}
\label{pr:loc-basic}
Let $\sfT$ be a triangulated category and $\sfS$ a triangulated
subcategory. Then the following are equivalent:
\begin{enumerate}
\item There exists a localisation functor $L\col\sfT\to\sfT$ such that
  $\Ker L=\sfS$.
\item There exists a colocalisation functor $\gam\col\sfT\to\sfT$ such
  that $\Im\gam=\sfS$.
\end{enumerate}
In that case both functors are related by a functorial exact triangle
\[
\gam X\lto X\lto \bloc X\lto.
\] 
Moreover, there are equalities \[\sfS^\perp=\Im
L=\Ker\gam\quad\text{and}\quad ^\perp(\sfS^\perp)=\sfS.\] The functor
$\gam$ induces a right adjoint for the inclusion $\sfS\to \sfT$, and
$L$ induces a left adjoint for the inclusion $\sfS^\perp\to\sfT$.\qed
\end{proposition}

Let us mention that $\bloc$ induces an equivalence of categories
$\sfT/{\Ker\bloc}\xra{\sim}\Im\bloc$ where $\sfT/{\Ker\bloc}$ denotes
the Verdier quotient of $\sfT$ with respect to $\Ker\bloc$. This can
be deduced from Lemma~\ref{le:loc}, but we do not need this fact.

\section{The stable module category of a finite group}
\label{sec:Tuesday3}
This lecture consists of an introduction to the cohomology of groups, leading to the definition of the cohomological variety of a $kG$-module.

\subsection{Algebra}
\label{ssec:group-cohomology}
We begin with the algebraists view of group cohomology. Let $k$ be a commutative ring and $G$ a group. Sooner
or later, we shall specialise to the case where $k$ is a field and $G$ is finite, but for the moment that is not necessary.

If $M$ and $N$ are a $kG$-modules, look at
\[ 
\Ext^*_{kG}(N,M)\,.
\]
We recall that this is defined as follows. Take a projective resolution\index{projective!resolution} 
of $N$ as a $kG$-module
\[ 
\dots \to P_n \to \dots \to P_1 \to P_0 \to 0
\]
with $\Coker(P_1\to P_0)\cong N$. Apply $\Hom_{kG}(-,M)$ to get
\[ 
0 \to \Hom_{kG}(P_0,M)\to \Hom_{kG}(P_1,M) \to \dots 
\]
and the cohomology of this complex in $n$th place is $\Ext^n_{kG}(N,M)$.

If $k$ is a commutative ring of coefficients and $M$ is a $kG$-module, we define
\[ 
H^*(G,M)=\Ext^*_{\bbZ G}(\bbZ,M)\cong \Ext^*_{kG}(k,M). 
\]
The isomorphism holds because, if $P$ is a projective resolution of $\bbZ$ over $\bbZ G$, then $k\otimes_{\bbZ}P$ is a projective resolution of $k$ over $kG$, and adjunction gives an isomorphism of complexes $\Hom_{\bbZ G}(P,M)\cong\Hom_{kG}(k\otimes_{\bbZ}P,M)$.

There are two ways to define products in group cohomology, namely \emph{cup products}\index{cup product} and 
\emph{Yoneda products}\index{Yoneda product}. When they are both defined, they agree.

The cup product defines maps
\[
H^i(G,M) \otimes_k H^j(G,N) \to H^{i+j}(G,M\otimes_k N). 
\]
whereas the Yoneda product defines maps
\[ 
\Ext^i_{kG}(M,L)\otimes_k \Ext^j_{kG}(N,M)\to \Ext^{i+j}_{kG}(N,L). 
\]
These agree when $L=M=N=k$, giving maps
\[ 
H^i(G,k) \otimes_k H^j(G,k) \to H^{i+j}(G,k). 
\]
This makes $H^*(G,k)$ into a  \emph{graded commutative ring},\index{graded commutative!ring}
meaning that the multiplication satisfies 
\[ 
xy = (-1)^{|x||y|}yx. 
\]
Here, $|x|$ denotes the degree of an element $x$. Furthermore,  cup and Yoneda product give the same module structure for $H^*(G,M)$ as a graded $H^*(G,k)$-module.

It is worth making a comment about graded commutativity versus commutativity. If $R$ is a graded commutative ring and $x\in R$ has odd degree then
\[ 
x^2 = - x^2 
\]
and so $2x^2=0$, which implies that $(2x)^2=0$. Furthermore, again using graded commutativity we see that
for all $y\in R$, $(2xy)^2= \pm (2x)^2y^2=0$. So $2x$ is in the \emph{nil radical}\index{nil radical} of $R$, and
it follows that modulo the nil radical, $x$ is congruent to $-x$.

The conclusion of this argument is as follows. If $R$ is graded commutative then $R$ modulo its nil radical is strictly commutative. So it makes sense to talk about the maximal ideal spectrum or the prime ideal spectrum of $R$.

\begin{theorem}[Evens \cite{Evens:1961a} (1961)]\index{Evens' Theorem}
\label{thm:Evens}
Let $G$ be a finite group and $k$ a commutative ring of coefficients. If $M$ is a noetherian $k$-module then $H^*(G,M)$ is a noetherian  $H^*(G,k)$-module. In particular, if $k$ is a noetherian ring then so is $H^*(G,k)$.\qed
\end{theorem}

We're interested in the case where $k$ is a field of characteristic $p$. In this case, Evens' theorem implies that 
$H^*(G,k)$ is a finitely generated graded commutative $k$-algebra. If $\Char(k)$ is zero or does not divide $|G|$ then there's nothing interesting here. You just get $k$ in degree zero. More generally, for any commutative ring of coefficients $k$,  $|G|$ annihilates positive degree elements.

The cohomology of elementary abelian groups was described in Section~\ref{ssec:gc}. Here is another interesting example; more are given later on in this section.

\begin{example} 
\index{Mathieu group $M_{11}$}
Let $G=M_{11}$, the Mathieu group, and suppose $\Char(k)=2$. Then $H^*(G,k)=k[x,y,z]/(x^2y+z^2)$ where
$|x|=3$, $|y|=4$ and $|z|=5$.
\end{example}

To get further with the commutative algebra of the cohomology ring, we next turn to Quillen's stratification theorem\index{Quillen!stratification theorem} \cite{Quillen:1971a,Quillen:1971b}. The first aspect is a characterisation of nilpotent elements, which should remind you of Chouinard's theorem.

\begin{theorem}[Quillen]
Let $G$ be a  finite group and $k$ a field of characteristic $p$. An element of $H^*(G,k)$ is nilpotent if and only if its restriction to every elementary abelian $p$-subgroup is nilpotent.\qed
\end{theorem}

The next aspect is a determination of the Krull dimension\index{Krull dimension!of $H^*(G,k)$} of cohomology.

\begin{theorem}[Quillen]
The Krull dimension of $H^*(G,k)$ is equal to the $p$-rank of $G$, namely the largest $r$ for which $(\bbZ/p)^r\le G$.\qed
\end{theorem}

More generally, Quillen described the prime ideal spectrum of $H^*(G,k)$ in terms of the set of elementary abelian subgroups of $G$, and their conjugations and inclusions. In his paper, he talks in terms of the inhomogeneous
maximal ideal spectrum, so we shall begin by doing the same. Since $H^*(G,k)$ modulo its nil radical is a finitely generated commutative $k$-algebra, Hilbert's Nullstellensatz\index{Hilbert's Nullstellensatz} tells us that the
prime ideal spectrum is determined by the maximal ideal spectrum and vice versa.

\begin{theorem}[Quillen]
\pushQED{\qed}
Let $V_G=\Max H^*(G,k)$, the spectrum of maximal ideals in $H^*(G,k)$. Then at the level of sets we have
\[ 
V_G=\varinjlim \Max H^*(E,k). \qedhere
\]
\end{theorem}

The limit is over the ``Quillen category''\index{Quillen!category} whose objects are the elementary abelian $p$-subgroups of $G$ and whose arrows are group homomorphisms induced by composing conjugations in $G$ and inclusions. 

It follows from Propositions~\ref{prop:gc-eab2} and \ref{prop:gc-eabp} that $\Max H^*(E,k)$ is an affine space of dimension equal to $\rank(E)$, and this describes $\Max H^*(G,k)$ as layers of affine spaces quotiented by  actions of normalisers and glued together.

In the last lecture we shall give a more precise statement of Quillen's stratification theorem, because it is an essential  ingredient in the classification of localising subcategories of the stable module category.

\begin{remark}
The homogeneous prime spectrum of a nonstandardly graded ring can be quite confusing. For example its spectrum can have singularities even though the ring itself is regular in the commutative algebra sense. An explicit example of this is the ring $k[x,y,z]$ with $|x| = 2, |y| = 4, |z| = 6$. The affine open patch corresponding to $y\ne  0$ has coordinate ring generated by $\alpha = x^{2}/y$, $\beta = xz/y^{2}$, and $\gamma = z^{2}/y^{3}$, with the single relation $\alpha\gamma=\beta^{2}$. This has a singularity at the origin. This example arises as $H^{*}(G, k)$ modulo its nil radical with $G = (Z/p)^{3} \rtimes \Sigma_{3}$ for $p\geq 5$, and $k$ a field of characteristic $p$.
\end{remark}

\subsection{Topology}

Next we examine group cohomology from the point of view of
an algebraic topologist.
Strictly speaking, this is not necessary for the development of
the subject of these lectures, but it would be a bad idea to know about 
group cohomology only from an algebraic standpoint, as many of
the ideas in the subject come from algebraic topology.

Let $EG$\index{EG@$EG$} be a contractible space with a free $G$-action. Such a space always exists, even for a topological group, by a theorem of Milnor. The quotient space by the action of $G$, denoted $BG=EG/G$ is\index{BG@$BG$}  independent of choice of $EG$, up to homotopy. This space $BG$ is called the \emph{classifying space}\index{classifying space} of the group $G$.

The topologist's definition of group cohomology is the cohomology of
the space $BG$: 
\[ 
H^*(G,k)=H^*(BG;k). 
\]
We reconcile this with the algebraic definition
\[ 
H^*(G,k)=\Ext^*_{\bbZ G}(\bbZ,k) 
\]
as follows: The singular chain complex $C_*(EG)$ has a $\bbZ$-basis consisting of the singular simplices; in degree $n$ these are the continuous maps $\Delta^n\to EG$ where $\Delta^n$ is a standard simplex. Since $EG$ is contractible, this is an acyclic complex except in degree zero, where it augments onto $\bbZ$. Since the action of $G$ on singular simplices is free, $C_*(EG)$ is a complex of free $\bbZ G$-modules. This means that it is a \emph{free resolution}\index{free resolution} of $\bbZ$ as a $\bbZ G$-module.
So we have
\begin{align*}
H^*(BG;k)&=H^*(\Hom_\bbZ(C_*(BG),k))\\
&=H^*(\Hom_{\bbZ G}(C_*(EG),k))\\
&\cong\Ext^*_{\bbZ G}(\bbZ,k).
\end{align*}

\subsection{Examples}

\begin{example}
\index{cohomology!$\bbZ$}
Our first example is an infinite group. Let $G=\bbZ$, the additive group of integers. We can take $EG=\bbR$
with $G$ acting by translation. Then $BG=\bbR/\bbZ=S^1$, the circle.
Thus
\[ 
H^0(\bbZ,k)\cong H^1(\bbZ,k)\cong k,\qquad H^i(\bbZ,k)=0 \text{ for } i\ge 2. 
\]
\end{example}

\begin{example}
Let $G=\bbZ/2$, the cyclic group of order two. Then $G$ acts freely on
an $n$-sphere $S^n$ by sending each point to the antipodal point. This is
not a contractible space, but if we embed $S^n$ in $S^{n+1}$ equatorially
then it contracts down to the south pole. So taking the union of all
the $S^n$, each embedded equatorially in the next, we get the infinite
sphere $S^\infty$ with the weak topology with respect to the subspaces $S^n$.
This is then a contractible space with a free $G$-action, so we take
$EG=S^\infty$. Then the quotient is $BG=\bbR P^\infty$, the infinite real
projective space, again with the weak topology.
If $k$ is a field of characteristic two then 
\[
H^*(G,k)=H^*(\bbR P^\infty;k)\cong k[x] 
\]
with $|x|=1$. See \ref{ssec:gc} for an algebraic perspective on this example and the next.
\end{example}

\begin{example}
\index{Klein four group!cohomology}
For use later in this lecture, we next look at the Klein four group\index{Klein four group} $G=\bbZ/2\times\bbZ/2$. In this case we can take $EG=S^\infty\times S^\infty$ and $BG=\bbR P^\infty \times \bbR P^\infty$.
If $\Char(k)=2$ then using the K\"unneth theorem we obtain
\[ 
H^*(G,k)=H^*(\bbR P^\infty\times \bbR P^\infty;k)\cong H^*(\bbR P^\infty;k)\otimes H^*(\bbR P^\infty;k)\cong k[x,y] 
\] 
with $|x|=|y|=1$. 
\end{example}

\begin{example}
\index{cohomology!quaternions}
Let $G=Q_8$, the quaternions, viewed as the vectors of length one in $SU(2)\cong S^3$. Then $G$ acts freely on $S^3$ by left multiplication. Taking cellular chains $C_*(S^3)$ and splicing in the homology gives us a
chain complex
\[ 
0 \to \bbZ \to C_3\to C_2 \to C_1 \to C_0 \to \bbZ \to 0 
\]
where $C_0,\dots,C_3$ are free $\bbZ G$-modules.

Form an infinite splice of copies of this complex:
\[ 
\xymatrix@=5mm{\dots\to C_1\ar[r]&C_0\ar[rr]\ar[dr]&&C_3\ar[r]&C_2\to C_1\to 
C_0\to \bbZ\to 0\\&&\bbZ\ar[ur]\ar[dr]\\&0\ar[ur]&&0} 
\]
to obtain a free resolution of $\bbZ$ as a $\bbZ G$-module.

From the existence of this periodic resolution, we can deduce that for any coefficients $k$, the cohomology
$H^*(Q_8,k)$ is periodic with period $4$. The interesting case for us is when $k$ is a field of characteristic $2$.
In this case the periodicity is given by multiplication by an element $z\in H^4(Q_8,k)$, and we have 
\[
H^*(Q_8,k)/(z)\cong H^*(S^3/Q_8;k)\,.
\]
\end{example}

The same argument proves the following general theorem.

\begin{theorem}
If $G$ acts freely on $S^{n-1}$ then $H^*(G,k)$ is periodic with period dividing $n$\index{free action on sphere}.\qed
\end{theorem}

\begin{example}
If $G=\bbZ/2\times\bbZ/2$ then $G$ cannot act freely on any sphere of any dimension, because we computed $H^*(G,k)=k[x,y]$ and this is not periodic.
\end{example}

\subsection{Cohomological varieties}
\index{cohomological variety}
\label{ssec:Cohomological Varieties}

\begin{definition}
Let $G$ be a finite group and $k$ a field of characteristic $p$.
If $M$ is a finitely generated $kG$-module, we have a map
\[ 
H^*(G,k)=\Ext^*_{kG}(k,k)\xrightarrow{- \otimes_k M} \Ext^*_{kG}(M,M). 
\]
Let $I_M$ be the kernel of this map, an ideal of $H^*(G,k)$. Then
we define $V_G(M)$ to be the subvariety of $V_G=\Max\ H^*(G,k)$ 
determined by this ideal.
\end{definition}

The following properties should be compared with the properties of the rank variety $V^r_E(M)$ in Section~\ref{ssec:Rank varieties}.

\begin{theorem}
\label{th:VGM}
Let $M,N$ be finitely generated $kG$-modules.
\begin{enumerate}
\item $V_G(M)$ is a closed homogeneous subvariety of $V_G$.
\item $V_G(M)=\{0\}$ if and only if $M$ is projective.
\item $V_G(M\oplus N)=V_G(M)\cup V_G(N)$.
\item $V_G(M\otimes_k N)=V_G(M) \cap V_G(N)$.
\item $\dim V_G(M)$ measures the rate of growth of the minimal 
projective resolution of $M$.
\item If $G=E$ is elementary abelian then there is an isomorphism
$V_E\cong \bbA^r(k)$ taking $V_E(M)$ to $V^r_E(M)$ for every $M$.

Note: for $p$ odd, this isomorphism involves a 
Frobenius twist.\index{Frobenius!twist} \qed
\end{enumerate}
\end{theorem}

The next step is a notion of support for infinitely generated modules.

\subsection{Rickard idempotents}
\label{ssec:Rickard idempotents}
Let $G$ be a finite group and $k$ a field whose characteristic divides $|G|$. Recall that $\StMod(kG)$ has the same objects as $\Mod(kG)$ but the arrows are module homomorphisms modulo those that factor through a projective module. The following result is the entry point for the application of homotopy theoretic techniques to the modular representation theory of finite groups.

\begin{theorem}
\label{thm:stmod-generation}
The stable module category $\StMod(kG)$ is triangulated, with suspension defined by $\Omega^{-1}$. Moreover, it is compactly generated, with compact objects the finite dimensional modules:
\[
\stmod(kG)\simeq \StMod(kG)^{\sfc}\,.
\]
Thus, the simple modules are a compact generating set for $\StMod(kG)$; in particular, if $G$ is a $p$-group, then $k$ is a compact generator for $\StMod(kG)$.
\end{theorem}

\begin{proof}
Since $kG$ is a finite dimensional and self-injective, this result is a special case of Example~\ref{ex:stmodA}.
\end{proof}

The next result is due to Rickard~\cite{Rickard:1997a}.

\begin{theorem}
Given a thick subcategory $\sfC$ of $\stmod(kG)$, there exists a functorial triangle in $\StMod(kG)$, unique up to isomorphism,
\[ 
\gam_\sfC(M) \to M \to L_\sfC(M) \to
\]
such that $\gam_\sfC(M)$ is in the localising subcategory of $\StMod(kG)$ generated by $\sfC$
and there are no maps in $\StMod(kG)$ from $\sfC$, or equivalently, from $\Loc(\sfC)$ to $L_\sfC(M)$.
\end{theorem}

\begin{proof}
The localizing subcategory $\Loc(\sfC)$ generated by $\sfC$ viewed as a subcategory of $\StMod(kG)$ is obviously compactly generated, and the inclusion 
\[
\Loc(\sfC)\subseteq \StMod(kG)
\]
preserves coproducts. Brown representability thus yields a right adjoint, which we denote $\gam_{\sfC}$; see Corollary~\ref{cor:right-adjoint-exist}. The existence of the functor $L$, and the remaining assertions, now follow from Proposition~\ref{pr:loc-basic}.
\end{proof}

By construction, see Proposition~\ref{pr:loc-basic}, the functors $\gam_\sfC$ and $L_\sfC$ are  orthogonal idempotents\index{Rickard idempotent} in the sense that
\begin{gather*}
\gam_\sfC\gam_\sfC\cong \gam_\sfC\qquad L_\sfC L_\sfC\cong L_\sfC\qquad
\gam_\sfC L_\sfC\cong L_\sfC\gam_\sfC\cong 0.
\end{gather*}

\begin{definition}
\index{gammav@$\gam_\mcV$}
\index{locv@$L_\mcV$}
Recall that a collection $\mcV$ of closed homogeneous irreducible subvarieties of $V_G$ is \emph{specialisation closed}\index{specialisation!closed} if $V\in \mcV,\ W\subseteq V \Rightarrow W\in\mcV$.

We write $\sfC_\mcV$ for the thick subcategory of $\stmod(kG)$ generated by finitely generated modules $M$ such that every irreducible component of $V_G(M)$ is contained in $\mcV$, and set
\[
\gam_\mcV=\gam_{\sfC_\mcV}\quad\text{and} \quad L_\mcV=L_{\sfC_\mcV}\,.
\]
\end{definition}

In Section~\ref{ssec:bcr-classification} we prove that the $\sfC_\mcV$ are the only thick subcategories of $\stmod(kG)$, provided $G$ is a $p$-group.

\subsection{Varieties for $\StMod(kG)$}
\index{cohomological variety}
We write $\mcV_G(k)$ for the set of non-zero closed homogeneous irreducible subvarieties of $V_G=\max H^*(G,k)$, or equivalently the set of non-maximal homogeneous prime ideals in $H^*(G,k)$.

\begin{definition}
Let $V\in \mcV_G(k)$. Set $\mcV=\{W\subseteq V\}$ and $\mcW=\mcV\setminus\{V\}$, and define 
\[ 
\gam_V=\gam_\mcV L_\mcW=L_\mcW\gam_\mcV\,. 
\]
\end{definition}

\begin{definition}
\label{def:mcVGM}
If $M$ is in $\StMod(kG)$, we define 
\[ 
\mcV_G(M)=\{V \in \mcV_G(k)\mid \gam_V(M)\ne 0\}. 
\]
\end{definition}

The following list of properties should be compared with the properties of $\mcV^r_E(M)$ from Section~\ref{ssec:Rank varieties}.

\begin{theorem}
\label{th:mcVGM}
Let $M$ be a $kG$-module. 
\begin{enumerate}
\item $\mcV_G(M)=\varnothing$ if and only if $M$ is projective.
\item $\mcV_G(M\oplus N)=\mcV_G(M)\cup\mcV_G(N)$, and more generally
\[ 
\mcV_G(\bigoplus_\alpha M_\alpha)=\bigcup_\alpha\mcV_G(M_\alpha). 
\]
\item $\mcV_G(M\otimes_k N)=\mcV_G(M)\cap\mcV_G(N)$.
\item If $M$ is finitely generated then $\mcV_G(M)=\{C\in\mcV_G(k)\mid V\subseteq V_G(M)\}$.
\item For any subset $\mcV\subseteq\mcV_G(k)$ there exists a $kG$-module $M$ such that $\mcV_G(M)=\mcV$. 
\item If $G=E$ is elementary abelian then the isomorphism $V_E\cong\bbA^r(k)$ of Theorem~\ref{th:VGM} 
takes $\mcV_E(M)$ to $\mcV^r_E(M)$ for every $M$. \qed
\end{enumerate}
\end{theorem}

\begin{remark} Observe that when dealing with the stable category $\stmod(kG)$ of finitely generated modules, we used $V_{G}$, the spectrum of maximal ideals in $H^{*}(G, k)$, whereas for the stable category $\StMod(kG)$ of all modules we used $\mcV_{G}$, the spectrum of homogeneous prime ideals in $H^{*}(G,k)$. In the case of a finitely generated module, $V_{G}(M)$ is determined by $\mcV_{G}(M)$, by the theorem above. It would have been possible to use $\mcV_{G}$ throughout, but we chose not to, partly for historical reasons. The origin of the use of $V_{G}$ is Quillen's work \cite{Quillen:1971b}, and much of the literature on finitely generated modules has been written in this context.
\end{remark}

\subsection{Classification of thick subcategories}
\label{ssec:bcr-classification}
\index{stmodkg@$\stmod(kG)$!thick subcategory of}

We begin with a definition.

\begin{definition}
Write $\Thick(M)$ for the thick subcategory of $\stmod(kG)$ generated by a finitely generated $kG$-module $M$.
\end{definition}

The following result due to Benson, Carlson, and Rickard~\cite{Benson/Carlson/Rickard:1997a} is the analogue of Hopkins' theorem~\ref{thm:hopkins} for the stable module category. For the sake of simplicity we stick to the case of $p$-groups. For a more general finite group there are some extra technicalities.

\begin{theorem}
\label{thm:bcr-classification}
Let $G$ be a $p$-group, let $M$ be a finitely generated $kG$-module, and set $\mcV=\{V\in\mcV_G(k)\mid V\subseteq V_G(M)\}$. Then $\sfC_{\mcV}=\Thick(M)$.
\end{theorem}

\begin{remark}
The hypothesis that $G$ is a $p$-group comes in as follows: For any $kG$-module $M$, if $\sHom_{kG}(k,M)=0$, then
$\sHom_{kG}(M,M)=0$, since $k$ is a compact generator for $\StMod(kG)$, by Theorem~\ref{thm:stmod-generation}, and hence  $M$ is projective. 
\end{remark}

\begin{proof}
Let $\sfC'=\Thick(M)$. 
It is clear that $M$ is in $\sfC$, so $\sfC'\subseteq\sfC$. Consider Rickard triangles associated with these subcategories:
\[ 
\gam_\sfC\to \Id \to L_\sfC\to ,\qquad \gam_{\sfC'}\to \Id \to L_{\sfC'}\to.
\] 
Since $\sfC'\subseteq \sfC$, there is an equality $\gam_{\sfC'}\gam_{\sfC}=\gam_{\sfC'}$. Therefore if $N$ is in $\StMod(kG)$, there is an exact triangle 
\[ 
\gam_{\sfC'}(N) \to\gam_\sfC(N) \to L_{\sfC'}\gam_\sfC(N)\to\,. 
\] 
So $L_{\sfC'}\gam_\sfC(N)$ is in $\Loc(\sfC)$. This implies that $\mcV_G(L_{\sfC'}\gam_\sfC(N))\subseteq\mcV$. 
But there are no homomorphisms from $M$ to $L_{\sfC'}\gam_\sfC(N)$, so with $M^{*}=\Hom_{k}(M,k)$, one gets
\begin{align*}
0=\sHom_{kG}(M,L_{\sfC'}\gam_\sfC(N)) 
&\cong \sHom_{kG}(k,\Hom_{k}(M,L_{\sfC'}\gam_\sfC(N)))  \\
&\cong \sHom_{kG}(k,M^*\otimes_k L_{\sfC'}\gam_\sfC(N)). 
\end{align*}
where the first isomorphism is by adjunction while the second one holds because $M$ is finite dimensional over $k$.
Since $G$ is a $p$-group, this implies that $M^*\otimes_k L_{\sfC'}\gam_\sfC(N)$ is $0$ in the stable module category. Noting that $\mcV_{G}(M^{*})=\mcV_{G}(M)$ (see the exercises to this lecture) the tensor product theorem, \ref{th:mcVGM}(3), then yields
\[ 
\mcV_G(M)\cap \mcV_G(L_{\sfC'}\gam_\sfC(N))=\varnothing. 
\] 
Hence $\mcV_G(L_{\sfC'}\gam_\sfC(N))=\varnothing$, and so $L_{\sfC'}\gam_\sfC(N)=0$.
So $\gam_{\sfC'}(N)\to \gam_{\sfC}(N)$ is an isomorphism
for all $N$.
Finally, if $N$ is in $\sfC$ then $\gam_{\sfC}(N)\to N$ is an
isomorphism, so
$\gam_{\sfC'}(N)\to N$ is an isomorphism. 
This implies that $N$ is in $\sfC'$.
\end{proof}

The result below can be deduced from Theorem~\ref{thm:bcr-classification} along the same lines that Theorem~\ref{thm:thickA} is deduced from Theorem~\ref{thm:hopkins}.

\begin{corollary}
Let $G$ be a $p$-group. The thick subcategories of $\stmod(kG)$ are in bijection with subsets of $\mcV_G(k)$ that are closed under specialisation. Under this bijection, a specialisation closed subset $\mcV\subseteq\mcV_G(k)$ corresponds to the full subcategory consisting of those finitely generated $kG$-modules $M$ with the property that every irreducible component of $V_G(M)$ is an element of $\mcV$.\qed
\end{corollary}

\begin{remark}
The extra technicality in the case where the finite group $G$ is not a $p$-group is the following. There is more than one simple $kG$-module, and we need to impose the extra condition on our thick subcategories that they are closed under tensor product with each simple module, or equivalently under tensor product with all finitely generated $kG$-modules. This condition is automatically satisfied in the case of a $p$-group since every finitely generated module has a finite filtration whose filtered quotients are direct sums of copies of $k$. 

For a general finite group $G$, we say that a thick subcategory of $\stmod(kG)$ is \emph{tensor ideal}\index{tensor!ideal} if this extra condition holds. The proof above goes through for an arbitrary finite group to yield a  classification of the tensor ideal thick subcategories of $\stmod(kG)$.

The classification of all thick subcategories of $\stmod(kG)$ when $G$ is not a $p$-group is problematic. There is a subvariety of $V_G$ called the \emph{nucleus}\index{nucleus} which captures the extra complication in this situation. For  further details see \cite{Benson:1995a,Benson:2002a,Benson/Carlson/Rickard:1997a}.
\end{remark}

\section{Exercises}
\label{exer:Tuesday}
In the following  exercises $A$ denotes a commutative noetherian ring, $G$ a finite group, and, depending on the context, $M$ is either a complex of $A$-modules, or a $kG$-module defined over a field $k$, usually with $\Char k$ dividing $|G|$.

\begin{enumerate}[\quad\rm(1)]

\item Find $\supp_{\bbZ}M$ for a finitely generated abelian group $M$. What is $\supp_{\bbZ}(\bbQ/\bbZ)$? 

\item Compute $\Supp_{A} k(\fp)$ and $\supp_{A} k(\fp)$ for each prime ideal $\fp$.

\item Verify that the following subcategories of $\sfD=\sfD^{b}(\mod A)$ are thick:
\begin{gather*}
\{M\in\sfD\mid \text{$M$ is perfect}\}\\
\{M\in\sfD\mid \text{$M$ has finite injective dimension}\}\\
\{M\in\sfD\mid \text{$\length_{A}\hh M$ is finite}\}
\end{gather*}
Think of other interesting properties of complexes, and consider whether they define thick subcategories of $\sfD$.

\item Prove that when $M$ is perfect so is $\kos M\fa$, for any ideal $\fa$ in $A$.

\item Recall Hopkins' theorem: if $M$ and $N$ are perfect complexes over $A$, then $\supp_{A}M\subseteq\supp_{A}N$ implies $M\in \Thick(N)$.

Even when the $A$-modules $H^{*}(M)$ and $H^{*}(N)$ are finitely generated, the hypothesis that $M,N$ are perfect is needed; find relevant examples.

\item Prove (without recourse to the result of Hopkins) that if $M$ and $N$ are finitely generated $\bbZ$-modules (or, even bounded complexes of finitely generated $\bbZ$-modules) with $\Supp_{\bbZ}M\subseteq \Supp_{\bbZ}N$, then $M$ is in $\Thick(N)$.

\item When $G$ is a $p$-group and $M$ is a finitely generated $kG$-module, $V_{G}(M)$ coincides with the subvariety of $V_{G}$ defined by the annihilator of $H^{*}(G,M)$ as a module over $H^{*}(G,k)$. Prove this.

\item 
Assume $M$ is a finite dimensional $kG$-module. Prove that the following conditions are equivalent:
\begin{enumerate}[\quad\rm(a)]
\item $V_{G}(M)=\{0\}$;
\item $M$ has finite projective dimension as a $kG$-module;
\item $M$ is a projective $kG$-module.
\end{enumerate}

\item Prove that when $M$ is a finite generated $kG$-module, there is an equality
\[
V_{G}(\Hom_{k}(M,k))=V_{G}(M)\,.
\] 

\item Let $G=\bbZ/2$ and assume $\Char k=2$. Prove that  $H^{*}(G,k)\cong k[x]$, a polynomial algebra over $k$ on a indeterminate $x$ of degree $1$.

\item Let $G=\bbZ/p$ with $p\geq 3$ and $\Char k=p$. Prove that
  $H^{*}(G,k)$ is the tensor product of an exterior algebra and a
  polynomial algebra, that is, one has an isomorphism of $k$-algebras:
  $H^{*}(G,k)\cong \Lambda (x)\otimes_k k[y]$, with $|x|=1$ and
  $|y|=2$.

\item 
The K\"unneth isomorphism yields an isomorphism of graded $k$-algebras
\[
H^{*}(G\times H,k)\cong H^{*}(G,k)\otimes_{k}H^{*}(H,k)
\]
where the tensor product on the right is the graded tensor product: 
\[
(a\otimes b)\cdot (c\otimes d)=(-1)^{|b||c|}ac\otimes bd\,.
\]
Using this isomorphism, compute the cohomology algebra of an elementary abelian $p$-group of rank $r$, for any $p\geq 2$.

\item Prove that $V_{G}\!\setminus\!\{0\}$ is connected for any finite group $G$.

\noindent Hint: Use Quillen's theorem, noting that each non-trivial $p$-group has a central element of order $p$.

\item Describe the conjugacy classes of elementary abelian 2-subgroups of $Q_{8}$. Then, using Quillen's theorem, compute the variety of $Q_{8}$, when $\Char k=2$.

\item Let $G=D_{8}$, the dihedral group of order $8$ and $\Char k=2$.
The following series of exercises leads to a description of $V_{G}$:
\begin{enumerate}[\rm(i)]
\item Prove that there are precisely two conjugacy classes of maximal elementary $2$-subgroups of $G$ (call them $E_{1}$ and $E_{2}$) both of rank $2$, and that they contain the centre, $\bbZ/2$, of $G$.
\item Show that $E_{i}$ is normal in $G$, so that $G/E_{i}\cong \bbZ/2$.
\item Recall that $H^{*}(E_{i},k)\cong k[x,y]$, with $|x|=1=|y|$. Prove that the $\bbZ/2$ action on $k[x,y]$, from (ii), is given by $x\mapsto x$ and $y\mapsto x+y$, and that 
\[
H^{*}(E_{i},k)^{G/E_{i}}\cong k[x,(x+y)y]\,.
\]
\item Use Quillen's theorem to deduce that $V_{G}$ consists of two affine planes glued along a line:
\[
V_{G}=\bbA^{2}(k)\cup_{\bbA^{1}(k)}\bbA^{2}(k)\,.
\]
\end{enumerate}

\item 
This exercise describes how to rotate a triangle in $\StMod(kG)$ arising from an exact sequence of $kG$-modules:
\[ 
0\to A \to B \to C \to 0\,.
\] 
Choose a surjective map $P\to C$ with $P$ a projective module. Write $\Omega C$ for the kernel of the composite surjection $P\to C$. Prove that there is an exact sequence of $kG$-modules
\[ 
0 \to \Omega C \to P \oplus A \to B \to 0\,. 
\]
One gets a triangle $\Omega C\to A \to B\to $ in $\StMod(kG)$ as $A\cong P\oplus A$ there.

Similarly, if we embed $B$ into an injective module $I$ and form the dual construction, we obtain an exact sequence
\[ 
0 \to B \to I \oplus C \to \Omega^{-1}A \to 0\,. 
\]
This gives a triangle $B\to C\to \Omega^{-1}A\to $ in $\StMod(kG)$.

\item Let $\sfA$ be an abelian category with enough injectives; for example, the category of modules over some ring. 
If $X,Y$ are complexes over $\sfA$ such that $Y^{n}$ is injective for all $n$ and $Y^{n}=0$ for $n\ll 0$, then the canonical map:
\[
\Hom_{\sfK(\sfA)}(X,Y)\to \Hom_{\sfD(\sfA)}(X,Y) 
\]
is an isomorphism. Using this, prove that for any $M,N$ in $\sfA$ one has
\[
\Hom_{\sfD(\sfA)}(M,\Si^{i}N)\cong \Ext^{i}_{\sfA}(M,N) \quad\text{for all $i\ge 0$.} 
\]
Here $M,N$ are viewed as complexes concentrated in degree zero.

\item Let $\sfT$ be a compactly generated triangulated category. Show that each object $X\ne 0$ admits a non-zero morphism $C\to X$ from a compact object $C$.
\end{enumerate}

\chapter{Wednesday}
\thispagestyle{empty}

The main goal of Wednesday's lectures is to introduce and develop a theory of support for triangulated categories endowed with an action of a commutative ring. The reason for doing this is that the proofs of the main results, to be presented in Friday's lectures, involve various triangulated categories (of modules over group algebras, and differential graded modules over differential graded rings) and one needs a notion of support applicable to each one of these contexts. In fact, what is required is a more fundamental construction, which underlies support, namely, local cohomology functors. This unifies and extends Grothendieck's local cohomology for commutative rings~\cite{Hartshorne:1967} and the technology of Rickard functors for the stable module category, presented in Section~\ref{ssec:Rickard idempotents}.

\section{Local cohomology and support}
\label{sec:Wednesday1}
In this lecture we explain the construction and the basic properties of local cohomology functors with respect to the action of a graded-commutative noetherian ring. These functors are then used to define the support for objects in a triangulated category. We follow closely the exposition in \cite[\S\S4-5]{Benson/Iyengar/Krause:2008a}. For an  introduction to local cohomology for commutative rings, on which much of our development is based, see the original source~\cite{Hartshorne:1967}; see \cite{Iyengar:2007} for a more recent exposition, highlighting its connections to various branches of mathematics.

\subsection{Central ring actions}
\label{ssec:ringactions}
Let $\sfT$ be a triangulated category admitting set-indexed
coproducts.  Recall that we write $\Si$ for the suspension on
$\sfT$. For objects $X$ and $Y$ in $\sfT$, let
\[ 
\Hom^*_\sfT(X,Y)=\bigoplus_{i\in\bbZ}\Hom_\sfT(X,\Si^i Y)
\]
be the graded abelian group of morphisms. Set
$\End^{*}_{\sfT}(X)=\Hom^{*}_{\sfT}(X,X)$; this is a graded ring, and
$\Hom^*_\sfT(X,Y)$ is a right $\End^{*}_{\sfT}(X)$ and
left $\End^{*}_{\sfT}(Y)$-bimodule.

Let $R$ be a graded-commutative ring; thus $R$ is $\bbZ$-graded and  $rs=(-1)^{|r||s|}sr$ for each pair of homogeneous elements $r,s$ in $R$.  We say that $\sfT$ is \emph{$R$-linear}\index{Rlinear@$R$-linear}, or that $R$ \emph{acts}\index{Raction@$R$-action} on $\sfT$, if there is a homomorphism $\phi\col R\to Z^*(\sfT)$ of graded rings, where $Z^*(\sfT)$ is the \emph{graded centre}\index{triangulated category!graded centre} of $\sfT$.  Note that $Z^*(\sfT)$ is a graded-commutative ring, where for each $n\in\bbZ$ the component in degree $n$ is
\[
Z^n(\sfT)=\{\eta\col\Id_\sfT\to\Si^n\mid\eta\Si=(-1)^n\Si\eta\}.
\] 
This yields for each object $X$ a homomorphism $\phi_X\col R\to\End^*_\sfT(X)$ of graded rings such that for all objects $X,Y\in\sfT$ the $R$-module structures on $\Hom^*_\sfT(X,Y)$ induced by $\phi_{X}$ and $\phi_{Y}$ agree, up to the usual sign rule.

For the rest of this lecture, $\sfT$ will be a compactly generated triangulated category with set-indexed coproducts, and $R$ a graded-commutative noetherian ring acting on $\sfT$.

\subsection{Local cohomology functors} 
We write $\Spec R$ for the set of homogeneous prime ideals of $R$. Fix $\fp\in\Spec R$ and let $M$ be a graded $R$-module. The homogeneous localisation of $M$ at $\fp$ is denoted by $M_{\fp}$ and $M$
is called \emph{$\fp$-local}\index{plocal@$\fp$-local} when the natural map $M\to M_{\fp}$ is bijective.

Given a homogeneous ideal $\fa$ in $R$, we set
\[
\mcV(\fa) = \{\fp\in\Spec R\mid \fp\supseteq \fa\}\,.
\] 
A graded $R$-module $M$ is \emph{$\fa$-torsion}\index{atorsion@$\fa$-torsion} if each element of $M$ is annihilated by a power of $\fa$; equivalently, if $M_\fp=0$ for all $\fp\in\Spec R \setminus\mcV(\fa)$.

The \emph{specialisation closure}\index{specialisation!closure} of a subset $\mcU$ of $\Spec R$ is the set
\[
\cl\mcU=\{\fp\in\Spec R\mid\text{there exists $\fq\in\mcU$ with $\fq\subseteq \fp$}\}.
\] 
The subset $\mcU$ is \emph{specialisation closed}\index{specialisation!closed} if $\cl\mcU=\mcU$; equivalently, if $\mcU$ is a union of Zariski closed subsets of $\Spec R$. For each specialisation closed subset $\mcV$ of $\Spec R$, we define the full subcategory of $\sfT$ of \emph{$\mcV$-torsion objects}\index{vtorsion@$\mcV$-torsion objects}\index{teev@$\sfT_{\mcV}$} as follows:
\[
\sfT_\mcV= \{X\in\sfT\mid\Hom^*_\sfT(C,X)_\fp= 0\text{ for all }
C\in\sfT^c,\, \fp\in \Spec R\setminus \mcV\}.
\]
This is a localising subcategory and there exists a localisation functor $L_\mcV\col\sfT\to\sfT$ such that $\Ker L_\mcV=\sfT_\mcV$; see \cite[Lemma 4.3, Proposition 4.5]{Benson/Iyengar/Krause:2008a}.

By Proposition~\ref{pr:loc-basic}, the localisation functor $L_{\mcV}$ induces a colocalisation functor
on $\sfT$, which we denote $\gam_\mcV$, and call the \emph{local cohomology functor}\index{local cohomology!functor} with respect to $\mcV$. For each object $X$ in $\sfT$ there is then an exact localisation triangle
\begin{equation}
\label{eq:locseq}
\gam_{\mcV}X\lto X\lto \bloc_{\mcV}X\lto\,.
\end{equation}

For each $\fp$ in $\Spec R$ and each object $X$ in $\sfT$ set
\[
X_\fp=L_{\mcZ(\fp)}X\,, \quad\text{where $\mcZ(\fp) = \{\fq\in\Spec R\mid \fq\not\subseteq \fp\}$.}
\]
The notation is justified by \cite[Theorem 4.7]{Benson/Iyengar/Krause:2008a}: the adjunction morphism $X\to X_{\fp}$ induces for any compact object $C$ an isomorphism of $R$-modules
\[
\Hom_\sfT^*(C,X)_{\fp}\xra{\cong}\Hom_\sfT^*(C,X_{\fp})\,.
\] 
We say $X$ is \emph{$\fp$-local}\index{plocal@$\fp$-local} if the adjunction morphism $X\to X_\fp$ is an isomorphism; this is equivalent to the condition that there exists \emph{some} isomorphism $X\cong X_{\fp}$ in $\sfT$.

Consider the exact functor $\gam_{\fp}\col\sfT\to\sfT$ obtained by
setting
\[
\gam_{\fp}X= \gam_{\mcV(\fp)}(X_{\fp}) \quad\text{for each object $X$
  in $\sfT$},
\]
and let $\gam_\fp\sfT$ denote its essential image.  One has a natural isomorphism $\gam_{\fp}^{2}\cong \gam_{\fp}$, and an object $X$ from $\sfT$ is in $\gam_\fp\sfT$ if and only if the $R$-module $\Hom^{*}_{\sfT}(C,X)$ is $\fp$-local and $\fp$-torsion for every compact object $C$; see \cite[Corollary~4.10]{Benson/Iyengar/Krause:2008a}.

\subsection{Support}
The \emph{support}\index{support} of an object $X$ in $\sfT$ is by definition the set
\[
\supp_{R} X=\{\fp\in\Spec R\mid\gam_{\fp}X\ne 0\}.
\] 
The result below, see \cite[Theorem~5.15]{Benson/Iyengar/Krause:2008a}, collects  basic properties of support; it is in fact, an axiomatic characterization. It is convenient to set for each  $X$ in $\sfT$
\[
\supp_R H^*(X)=\bigcup_{C\in\sfT^c} \supp_R \Hom^*_\sfT(C,X)\,.
\] 
Note that the support on the right hand side refers to that defined in Section~\ref{se:appendix-supp}.

\begin{theorem}
\label{thm:axioms}
The assignment that sends each object $X$ in $\sfT$ to the subset
$\supp_{R} X$ of $\Spec R$ has  the following properties:
\begin{enumerate}
\item \emph{Cohomology:} For each object $X$ in $\sfT$ one has
\[
\cl(\supp_{R} X) = \cl(\supp_R H^*(X))\,.
\] 
\item \emph{Orthogonality:} For objects $X$ and $Y$ in $\sfT$, one has that
\[
\cl(\supp_{R} X)\cap\supp_{R} Y=\varnothing \quad\text{implies}\quad \Hom_\sfT(X,Y)=0\,.
\]
  \item \emph{Exactness:} For every exact triangle $X\to Y\to Z\to$ in $\sfT$, one has
\[
\supp_{R} Y\subseteq \supp_{R} X \cup \supp_{R} Z\,.
\] 
\item \emph{Separation:} For any specialisation closed subset $\mcV$ of $\Spec R$ and object $X$ in $\sfT$, there exists an exact triangle $X'\to X\to X''\to$ in $\sfT$ such that
\[
\supp_{R} X'\subseteq\mcV \quad\text{and}
\quad\supp_{R}X'' \subseteq\Spec R\setminus \mcV\,.
\]
  \end{enumerate}
Any other assignment having these properties coincides with $\supp_R(-)$.\qed
\end{theorem} 

\section{Koszul objects and support}
\label{sec:Wednesday2}
In the last lecture we defined a notion for support for triangulated category, based on local cohomology functors. In this one we describe some methods for computing it, in terms of the support, in the sense of \ref{Appendix}, of the cohomology. 

Let $R$ be a graded commutative noetherian ring and $\sfT$ an $R$-linear triangulated category; see Section~\ref{ssec:ringactions}. To each specialisation closed subset $\mcV$ of $\Spec R$ we associated an exact localisation triangle
\[
\gam_{\mcV}X\to X\to L_{\mcV}X\to \,.
\]
One should think of $\gam_{\mcV}X$ as the part of $X$ supported on $\mcV$ and $L_{\mcV}X$ as the part supported on $\Spec R\setminus \mcV$; this will be clarified in this lecture. This point of view should at least make the following statement plausible.

\begin{lemma}
\label{lem:loc-commute}
If $\mcV\subseteq \mcW$ are specialisation closed sets, then for each $X$ in $\sfT$ there are natural isomorphisms
\[
\begin{gathered}
\gam_{\mcV}\gam_{\mcW}X\xra{\ \cong \ } \gam_{\mcV}X\xla{\ \cong \ }\gam_{\mcW}\gam_{\mcV}X \\
\gam_{\mcV}L_{\mcW}X =0 = L_{\mcW}\gam_{\mcV}X 
\end{gathered}
\]
\end{lemma}

\begin{proof}
The hypothesis implies that $\mcV$-torsion objects are $\mcW$-torsion: $\sfT_{\mcV}\subseteq\sfT_{\mcW}$. Thus, essentially by definition, the natural map $\gam_{\mcW}\gam_{\mcV}X\to \gam_{\mcV}X$ is a isomorphism; equivalently, $L_{\mcW}\gam_{\mcV}X=0$, as follows by considering the localisation triangle~\eqref{eq:locseq} defined by $\mcW$ and the object $\gam_{\mcV}X$. This settles half the claims.

As to the rest: It is easy to verify that the composition $\gam_{\mcV}\gam_{\mcW}$ is a right adjoint to the inclusion $\sfT_{\mcV}\subseteq\sfT$, and hence isomorphic to $\gam_{\mcV}$, by the uniqueness of adjoints. 
Applying $\gam_{\mcV}$ to the localisation triangle~\eqref{eq:locseq} defined by $\mcW$, it then follows that $\gam_{\mcV}L_{\mcW}X =0$.
\end{proof}

Even when $\mcV$ is not contained in $\mcW$, the local cohomology functors interact in a predictable, and useful, way; see \cite[Proposition~6.1]{Benson/Iyengar/Krause:2008a}. One consequence of these is worth recording:

\begin{theorem}
\label{thm:vw}
\pushQED{\qed}
Fix a point $\fp$ in $\Spec R$. For any specialisation closed subsets $\mcV$ and $\mcW$ of $\Spec R$ such that $\mcV \setminus \mcW = \{\fp \}$ there are natural isomorphisms 
\[
L_{\mcW}\gam_{\mcV}\cong \gam_{\fp} \cong \gam_{\mcV} L_{\mcW}\,.\qedhere
\]
\end{theorem}
This may be viewed as an expression of the local nature of the functor $\gam_{\fp}$.

\subsection{Computing support}

Let $X$ be an object in $\sfT$. For any compact object $C$ in $\sfT$, we write $H^{*}_{C}(X)$ for $\Hom^{*}_{\sfT}(C,X)$, and think of it as the \emph{cohomology of $X$ with respect to $C$}. 

For any subset $\mcU$ of $\Spec R$, we write $\min\,\mcU$ for the prime ideals which are minimal, with respect to inclusion, in the set $\mcU$, and for any $R$-module $M$, set
\[
\min_{R} M = \min (\Supp_{R}M)\,.
\]
This set coincides with $\min (\supp_{R}M)$, for $\Supp_{R}M$ is the specialisation closure of $\supp_{R}M$; see Lemma~\ref{le:supp-ann}. The result below, which is \cite[Theorem~5.2]{Benson/Iyengar/Krause:2008a}, is often useful for computing supports.
 
\begin{theorem}
\label{thm:cohom-supp}
For each $X$ in $\sfT$ there is an equality
\[
\supp_{R}X = \bigcup_{C\in\sfT^{\sfc}}\min_{R}H^{*}_{C}(X)
\]
In particular, $\supp_{R}X=\varnothing$ if and only if $X=0$. \qed
\end{theorem}

We omit the proof, and will focus on its consequences, one of which is:

\begin{theorem}
\label{thm:loc-supp}
For each specialisation closed subset $\mcV$ of $\Spec R$ there are equalities
\begin{align*}
\supp_{R}\gam_{\mcV}X & = \supp_{R}X\cap \mcV \\
\supp_{R}L_{\mcV}X &= \supp_{R}X\cap (\Spec R\setminus \mcV)
\end{align*}
In particular, $\supp_{R}X=\supp_{R}\gam_{\mcV}X\, \coprod\, \supp_{R}L_{\mcV}X$.
\end{theorem}

\begin{proof} To begin with, we verify the

\smallskip

\emph{Claim}: $\supp_{R}\gam_{\mcV}X\subseteq \mcV$ and $\supp_{R}L_{\mcV}X\subseteq \Spec R\setminus \mcV$.

\smallskip

Indeed, for each $\fp$ in $\mcV$, the construction of $\gam_{\fp}$ gives the first equality below: 
\[
\gam_{\fp}(L_{\mcV}X) = L_{\mcZ(\fp)}\gam_{\mcV(\fp)}L_{\mcV}X = 0\,,
\]
while the second equality holds by Lemma~\ref{lem:loc-commute} since $\mcV(\fp)\subseteq\mcV$. Thus $\supp_{R}L_{\mcV}X$ is contained in $\Spec R\setminus \mcV$. On the other hand, since the image of $\gam_{\mcV}$ is in $\sfT_{\mcV}$, for each compact object $C$ one gets the second inclusion below:
\[
\min_{R}H^{*}_{C}(\gam_{\mcV}X)\subseteq \supp_{R}H^{*}_{C}(\gam_{\mcV}X)\subseteq \mcV\,.
\]
Thus Theorem~\ref{thm:cohom-supp} implies $\supp_{R}\gam_{\mcV}X\subseteq \mcV$. This justifies the claim.

Now, since $\gam_{\fp}$ is an exact functor, for any $\fp\in\Spec R$, from the localisation triangle~\eqref{eq:locseq}  one gets inclusions
\begin{gather*}
\supp_{R}\gam_{\mcV}X  \subseteq \supp_{R}X\cup\supp_{R}L_{\mcV}X\\
\supp_{R}X  \subseteq \supp_{R}\gam_{\mcV}X \cup \supp_{R}L_{\mcV}X
\end{gather*}
Combining the first inclusion with those in the claim yields:
\begin{align*}
\supp_{R}\gam_{\mcV}X 
&\subseteq (\supp_{R}X\cup \supp_{R}L_{\mcV}X)\cap \mcV \\
&= (\supp_{R}X\cap \mcV) \cup (\supp_{R}L_{\mcV}X\cap \mcV) \\
&=\supp_{R}X\cap \mcV
\end{align*}
In the same vein, an inclusion $\supp_{R}L_{\mcV}X\subseteq \supp_{R}X\cap (\Spec R\setminus \mcV)$ holds.
The desired equalities follow, as $\supp_{R}X\subseteq \supp_{R}\gam_{\mcV}X \cup \supp_{R}L_{\mcV}X$.
\end{proof}

\begin{corollary}
\label{cor:vlocal-tests}
Let $\mcV$ be a specialisation closed subset of $\Spec R$.
\begin{enumerate}[\quad\rm(1)]
\item $\supp_{R}X\subseteq \mcV\iff X\in\sfT_{\mcV}\iff\gam_{\mcV}X\xra{\cong}X\iff L_{\mcV}X=0$.
\item $\mcV\cap \supp_{R}X=\varnothing \iff X\xra{\cong}L_{\mcV}X\iff \gam_{\mcV}X=0$.
\end{enumerate}
\end{corollary}

\begin{proof}
It follows from the preceding result that $\supp_{R}X\subseteq \mcV$ holds if and only if $\supp_{R}L_{\mcV}X=0$, that is to say, if and only if $L_{\mcV}X=0$; see Theorem~\ref{thm:cohom-supp}. Now (1) follows from the localisation triangle \eqref{eq:locseq}. A similar argument works for (2).
\end{proof}

\begin{corollary}
If $\cl(\supp_{R}X)\cap\supp_{R}Y=\varnothing$, then $\Hom^{*}_{\sfT}(X,Y)=0$.
\end{corollary}

\begin{proof}
For $\mcV=\cl(\supp_{R}X)$ the previous corollary yields isomorphisms
\[
\gam_{\mcV}X\xra{\cong} X\quad\text{and}\quad Y\xra{\cong}L_{\mcV}Y\,.
\]
Hence $\Hom^{*}_{\sfT}(X,Y)\cong \Hom^{*}_{\sfT}(\gam_{\mcV}X,L_{\mcV}Y)=0$; the equality holds because there are no non-zero morphisms from $\mcV$-torsion to $\mcV$-local objects, by Proposition~\ref{pr:loc-basic}.
\end{proof}

\begin{corollary}
\label{cor:plocptor}
For object $X\ne $ in $\sfT$ and $\fp\in\Spec R$ the following are equivalent.
\begin{enumerate}[\quad\rm(1)]
\item
$\gam_{\fp}X\cong X$;
\item
$\supp_{R}X=\{\fp\}$;
\item
The $R$-module $H^{*}_{C}(X)$ is $\fp$-local and $\fp$-torsion for each $C\in\sfT^{\sfc}$.
\end{enumerate}
\end{corollary}

\begin{proof}

(1) $\implies$ (3) For each compact object $C$ the $R$-module $H^{*}_{C}(\gam_{\fp}X)$ is $\fp$-local because there are isomorphisms of $R$-modules
\[
H^{*}_{C}(\gam_{\fp}X) = H^{*}_{C}(L_{\mcZ(\fp)}\gam_{\mcV(\fp)}X) \cong H^{*}_{C}(\gam_{\mcV(\fp)}X)_{\fp}
\]
It is also $\fp$-torsion, for it is localisation of the $\fp$-torsion module $H^{*}_{C}(\gam_{\mcV(\fp)}X)$.

(3) $\implies$ (2) follows from Theorem~\ref{thm:cohom-supp}, and Lemma~\ref{le:torsion-local}.

(2) $\implies$ (1) When $\supp_{R}X=\{\fp\}$ holds, so does $\supp_{R}\gam_{\mcV(\fp)}X=\{\fp\}$, by Theorem~\ref{thm:loc-supp}. Corollary~\ref{cor:vlocal-tests} then yields isomorphisms
\[
X\xla{\cong} \gam_{\mcV(\fp)}X\xra{\cong} L_{\mcZ(\fp)}\gam_{\mcV(\fp)}X=\gam_{\fp}X\,.
\]
For the second isomorphism, note that $\mcZ(\fp)\cap \{\fp\}=\varnothing$.
\end{proof}

\subsection{Koszul objects}

Let $C$ be an object in $\sfT$ and $r\in R^{d}$, a homogeneous element of degree $d$. Since $\sfT$ is $R$-linear, $r$ induces a morphism $C\xra{r} \Si^{d}C$; completing it to an exact triangle yields an object denoted $\kos Cr$:
\[
C\xra{r}\Si^{d}C \to \kos Cr \to \Si C \to 
\]
Note that $\kos Cr$ is well-defined, but only up to non-unique isomorphism. We call it the \emph{Koszul object}\index{Koszul object} on $r$. For each $X$ in $\sfT$, the triangle above induces an exact sequence of graded $R$-modules:
\begin{gather*}
\to H^{*}_{C}(X)[-d-1]\xra{\ r\ } H^{*}_{C}(X)[-1]\to H^{*}_{\kos Cr}(X)
\to H^{*}_{C}(X)[-d]\xra{\ r\ } H^{*}_{C}(X)\to
\end{gather*}
This translates to an exact sequence of graded $R$-modules:
\begin{gather}
\label{eq:kosles}
0\lto \frac{H^{*}_{C}(X)}{rH^{*}_{C}(X)}[-1]\lto H^{*}_{\kos Cr}(X)\lto (0\,:_{H^{*}_{C}(X)}\,r)\lto 0\,.
\end{gather}
Compare this with the exact sequence appearing in the proof of Proposition~\ref{prop:kosca}.

Given a sequence $\bsr=r_{1},\dots,r_{n}$ we denote $\kos C{\bsr}$ the object obtained by iterated Koszul construction. To be precise, $\kos C{\bsr}=C_{n}$ where 
\[
C_{0}=C\quad\text{and}\quad C_{i}=\kos {C_{i-1}}{r_{i}}\quad \text{for $i\ge 1$}.
\]
Finally, given an ideal $\fa$ in $R$, we write $\kos C{\fa}$\index{kosa@$\kos C{\fa}$} for any Koszul object $\kos C{\bsr}$, where $\bsr$ is a finite generating set for the ideal $\fa$.

The following result is a consequence of \cite[Proposition~2.11(2)]{Benson/Iyengar/Krause:bik2}. It shows the (in)dependence of the Koszul object on a generating set for the ideal defining it. For tensor triangulated categories, one has a more precise result; see \cite[Lemma~6.0.9]{Hovey/Palmieri/Strickland:1997a}.

\begin{lemma}
\label{lem:kosloc}
\pushQED{\qed}
If $\fa$ and $\fb$ are ideals in $R$ such that $\mcV(\fa)\subseteq \mcV(\fb)$, then 
\[
\Loc_{\sfT}(\kos C{\fa})\subseteq\Loc_{\sfT}(\kos C{\fb})\,.\qedhere
\]
\end{lemma}

The next result records the basic calculations concerning cohomology of Koszul objects. We give only a sketch of the argument, referring the reader to \cite[Lemma~5.11]{Benson/Iyengar/Krause:2008a} for details.

\begin{proposition}
\label{prop:kosprop}
Fix objects $C$ and $X$ in $\sfT$ and a point $\fp\in \Spec R$. 
\begin{enumerate}[\quad\rm(1)]
\item
There exists an integer $s\ge 0$, independent of $C$ and $X$, such that
\[
\fp^{s}\Hom^{*}_{\sfT}(\kos C{\fp},X)=0=\fp^{s}\Hom^{*}_{\sfT}(X,\kos C{\fp})\,.
\]
In particular, $\Hom^{*}_{\sfT}(\kos C{\fp},X)$ and $\Hom^{*}_{\sfT}(X,\kos C{\fp})$ are $\fp$-torsion.
\item
$\kos C{\fp}$ is in $\sfT_{\mcV(\fp)}$.
\item
$\Hom^{*}_{\sfT}(C,X)=0$ implies $\Hom^{*}_{\sfT}(\kos C{\fp},X)=0$. The converse holds if the $R_{\fp}$-module $\Hom^{*}_{\sfT}(C,X)$ is either $\fp$-torsion, or $\fp$-local and finitely generated.
\item
When $C$ is compact, there is an isomorphism of $R$-modules:
\[
\Hom^{*}_{\sfT}(\kos C{\fp},\gam_{\fp}X)\cong \Hom^{*}_{\sfT}(\kos C{\fp},X)_{\fp}\,.
\]
\end{enumerate}
\end{proposition}

\begin{proof}[Sketch of a proof]
Statements (1) and (3) follow from \eqref{eq:kosles} and an induction on the number of generators for $\fp$.
Then (2) is a consequence of (1), while (4) is justified by the isomorphisms
\[
\Hom^{*}_{\sfT}(\kos C{\fp},\gam_{\mcV(\fp)}X_{\fp})
\cong \Hom^{*}_{\sfT}(\kos C{\fp},X_{\fp})
\cong \Hom^{*}_{\sfT}(\kos C{\fp},X_{\fp})
\]
where the first one follows from (2) and the second holds as $\kos C{\fp}$ is compact.
\end{proof}

Using Koszul objects one can establish a more precise (and economical) version of Theorem~\ref{thm:cohom-supp}.
 
\begin{theorem} 
\label{thm:cohom-supp2}
Let $\sfG\subset\sfT$ be a compact generating set and $X$ an object in $\sfT$. Then
\[
\fp\in\supp_{R}X\iff H^{*}_{\kos C{\fp}}(X)\ne 0\ \text{for some $C\in \sfG$}\,.
\]
In particular, there is an equality
\[
\supp_{R}X = \bigcup_{\begin{stackrel}{\fp\in\Spec R} {C\in\sfG}\end{stackrel}}
\min_{R}H^{*}_{\kos C{\fp}}(X)
\]
\end{theorem}

\begin{proof}
Fix $\fp$ in $\Spec R$. The following equivalences hold:
\begin{alignat*}{3}
\gam_{\fp}X\ne 0 
    & \iff H^{*}_{C}(\gam_{\fp}X)\ne 0 &&\text{for some $C\in\sfG$}& \\
	& \iff H^{*}_{\kos C{\fp}}(\gam_{\fp}X) \ne 0 && \text{by Proposition~\ref{prop:kosprop}(3)}&\\
	& \iff H^{*}_{\kos C{\fp}}(X)_{\fp} \ne 0 && \text{by Proposition~\ref{prop:kosprop}(4)}& \\
	& \iff \fp \in \min_{R}H^{*}_{\kos C{\fp}}(X)\quad && \text{by Proposition~\ref{prop:kosprop}(1)}&
\end{alignat*}
In the second step we have used the fact that $\Hom_{\sfT}(C,\gam_{\fp}X)$ is $\fp$-local and $\fp$-torsion, by Corollary~\ref{cor:plocptor}.
\end{proof}

\section{The homotopy category of injectives}
\label{sec:Wednesday3}
In this lecture, we explain how to enlarge the stable module category $\StMod(kG)$ slightly to a category $\KInj{kG}$. The effect of this on the variety theory is that it puts back the ``missing origin'' in the collection of subvarieties of $V_G$. Most of the material in this lecture is taken from the paper of Benson and Krause \cite{Benson/Krause:2008a}. 

\subsection{The stable module category and Tate cohomology}%
\index{stable module category}

Let $G$ be a finite group and $k$ be a field of characteristic $p$. Recall from Section~\ref{sec:Monday1} that the module category $\Mod(kG)$ has as its objects the $kG$-modules, and arrows the module homomorphisms. The stable module category $\StMod(kG)$ has the same objects, but its arrows are given by
\[ 
\sHom_{kG}(M,N) = \Hom_{kG}(M,N)/P\Hom_{kG}(M,N) 
\]
where $P\Hom_{kG}(M,N)$ is the subspace of homomorphisms 
that factor through some projective module. The categories
$\mod(kG)$ and $\stmod(kG)$ are the
full subcategories of finitely generated modules. 
The categories $\Mod(kG)$, $\mod(kG)$ are abelian categories
while $\StMod(kG)$, $\stmod(kG)$ are triangulated categories
with shift $\Omega^{-1}$.

The endomorphisms of $k$ in $\StMod(kG)$ form a graded commutative
ring called the \emph{Tate cohomology ring}\index{Tate!cohomology}
whose degree $n$ part is
\[ 
\hat H^n(G,k) = \sHom_{kG}(\Omega^n k, k) \cong 
\sHom_{kG}(\Omega^{n+m}k,\Omega^m k) 
\]
The multiplication is ``shift and compose'': if $x\in\hat H^m(G,k)$
and $y\in \hat H^n(G,k)$ then we choose corresponding homomorphisms
$\Omega^mk\to k$ and $\Omega^{n+m}k\to\Omega^m k$ and compose them
to obtain a representative for the product $xy$.

If $n\ge 0$ then $\hat H^n(G,k)\cong H^n(G,k)$, while for $n<0$ \emph{Tate duality}\index{Tate!duality} implies that
$\hat H^n(G,k)$ is the vector space dual of $H^{-n-1}(G,k)$.

The result below identifies groups for which the Tate cohomology ring
is noetherian; see \cite[Lemma~10.1]{\bik:2008a} for a proof.

\begin{theorem}
\label{thm:tate-nonnoetherian}
If the Sylow $p$-subgroups of $G$ are cyclic then $\hat H^*(G,k)$ is 
a localisation of $H^*(G,k)$ given by inverting any 
positive degree non-nilpotent element. 
In this case $\hat H^*(G,k)$ is noetherian.

If the Sylow $p$-subgroups of $G$ are not cyclic then every element of
negative degree in $\hat H^*(G,k)$ is nilpotent. 
In this case $\hat H^*(G,k)$ is not noetherian. \qed
\end{theorem}

Next we describe another view of Tate cohomology. A \emph{Tate resolution}\index{Tate!resolution}  of a module $M$ is obtained by splicing together a projective and an injective resolution:
\[ 
\xymatrix@=5mm{\dots\to\hat P_2\to\hat P_1\ar[r]&\hat P_0\ar[rr]\ar[dr] && 
\hat P_{-1}\ar[r] &\hat P_{-2} \to \cdots \\ 
&&M\ar[ur]\ar[dr] \\  
&0\ar[ur]&&0} 
\]
If $N$ is another module then the Tate Ext\index{Tate!Ext} group $\tExt^n_{kG}(M,N)$ is the cohomology of the cochain complex $\Hom_{kG}(\hat P_*,N)$. In particular the Tate cohomology is given by $\hat H^n(G,k)=\tExt^n_{kG}(k,k)$.

\subsection{The derived category of $kG$-modules}
\index{derived category}

One of the problems with the stable module category $\StMod(kG)$ is that the graded endomorphism ring of the trivial module is the Tate cohomology ring $\hat H^*(G,k)$, which is
usually not noetherian. So for example it is a hassle to
deal with injective resolutions over this ring. We can try to cure
this by using the ordinary cohomology ring, but then the maximal
ideal of positive degree elements needs special treatment, and
keeping track of this is again a hassle.

We can try to solve this problem by moving to the derived category,
but this creates new problems. Let us briefly recall the construction
of the derived category. The category of cochain complexes $\sfC(\Mod
kG)$ has as objects the complexes of $kG$-modules and as arrows the
degree preserving maps of complexes.  The homotopy category of cochain
complexes $\sfK(\Mod kG)$ has the same objects, but the arrows are the
homotopy classes of maps of complexes.  Finally, the derived category
$\sfD(\Mod kG)$ has the same objects, but the arrows are obtained by
adjoining inverses to the quasi-isomorphisms in $\sfK(\Mod kG)$;
recall that a map of cochain complexes is a
\emph{quasi-isomorphism}\index{quasi-isomorphism} if the induced map
between the cohomologies of the complexes is an isomorphism.

Note that $\sfC(\Mod kG)$ is an abelian category while the categories $\sfK(\Mod kG)$ and $\sfD(\Mod kG)$ are triangulated.

If $M$ and $N$ are modules, made into complexes whose only non-zero terms
are in degree zero, then the space of degree $n$
homomorphisms in the derived category from $M$ to $N$ 
is isomorphic to $\Ext^n_{kG}(M,N)$. In particular, the graded endomorphism
ring of $k$ is $H^*(G,k)$.

\subsection{Problems with the derived category}

The first problem we have with $\sfD(\Mod kG)$ is that a finitely generated non-projective module $M$ regarded as a complex concentrated in a single degree is not compact. In other words, $\Hom_{kG}(M,-)$ does not distribute over direct sums. So for example 
\[
\bigoplus_i H^*(G,X_i) \to H^*(G,\bigoplus_i X_i) 
\]
is not an isomorphism in general. In fact, the compact objects are the \emph{perfect complexes},\index{perfect complex} namely the complexes isomorphic to  bounded complexes of finitely generated projective modules.

Another problem is that there are not a lot of localising subcategories, so it is unlikely to help us directly to classify the localising subcategories of $\StMod(kG)$.

\begin{definition}
We write $\Loc_\sfT(\sfC)$\index{Loc@$\Loc_\sfT(\sfC)$} for the smallest localising category of a triangulated category $\sfT$ containing a subcategory (or collection of objects) $\sfC$.
\end{definition}

\begin{theorem}
If $G$ is a $p$-group, the only localising subcategories of
$\sfD(\Mod kG)$ are zero and the entire category.
\end{theorem}
\begin{proof}
If $X$ is a non-zero object, we claim that $\Loc(X)$ is the whole
category. Since $X$ is non-zero, it has some non-vanishing cohomology, say
$H^i(X)$. Since $kG$ has a filtration where the filtered quotients are 
isomorphic to $k$, $X\otimes_k kG$ is in $\Loc(X)$. 
Then $H^i(X\otimes_k kG)=H^i(X)\otimes_k kG$ is a free module. 
So it splits off the complex (Exercise!) and $kG$ concentrated in
degree $i$ is in $\Loc(X)$. But this generates $\sfD(\Mod kG)$.
\end{proof}

\subsection{The category $\KInj{kG}$}
\index{KInj@$\KInj{kG}$}

Better than either the stable module category or the derived category is the homotopy category of injective modules $\KInj{kG}$. The following result is an analogue of Theorem~\ref{thm:stmod-generation} for the homotopy category: 

\begin{theorem}
\label{thm:kinj-generation}
The homotopy category $\KInj{kG}$ is triangulated, with suspension defined by $\Sigma$. Moreover, it is compactly generated, and the natural localisation functor $\KInj{kG}\to \sfD(\Mod kG)$ induces an equivalence of categories:
\[
\KInj{kG}^{\sfc}\simeq \sfD^{\sfb}(\mod kG)\,.
\]
The quasi-inverse associates to each complex in $\sfD^{\sfb}(\mod kG)$ its injective resolution. In particular, the injective resolutions of the simple modules form a compact generating set for $\KInj{kG}$.
\end{theorem}

\begin{proof}
The first part, namely that $\KInj{kG}$ is a triangulated category, is special case of Example~\ref{ex:frobstable}(2). For the statement about compact generation and the identification of compact objects, see \cite[Proposition~2.3]{Krause:2005a}.
\end{proof}

In view of this theorem, $\KInj{kG}$ should be regarded as the correct ``big'' category for $\sfD^{\sfb}(\mod kG)$, whereas $\sfD(\Mod kG)$ is not.

Let us give names for some particular objects in $\KInj{kG}$ that we shall need. We write:
\begin{itemize}
\item
$ik$\index{ik@$ik$} for an injective resolution of $k$,
\item
$pk$\index{pk@$pk$} for a projective resolution of $k$, and
\item
$tk$\index{tk@$tk$} for a Tate resolution of $k$.
\end{itemize}

In $\KInj{kG}$ there is then an exact triangle 
\[ 
pk \to ik \to tk \to\,.
\]

For the next step, we note that if $X$ and $Y$ are objects in $\KInj{kG}$, then $X\otimes_k Y$ \index{tensor!product} can be made into an object in $\KInj{kG}$ by taking the total complex of the tensor product, with diagonal action:
\[ 
(X\otimes_k Y)_i=\bigoplus_{j+k=i}X_j\otimes Y_k 
\]
with differential 
\[ 
d(x\otimes y)=d(x)\otimes y + (-1)^{|x|}x\otimes d(y). 
\]
The object $ik$ acts as a \emph{tensor identity}\index{tensor!identity}:

\begin{proposition}
\label{prop:tensorid}
The map $k \to ik$ induces an isomorphism $X \to X \otimes_k ik$  for any object $X$ in $\KInj{kG}$. \qed
\end{proposition}

\subsection{Recollement}
\label{ssec:recollement}
The relationship between $\KInj{kG}$, the derived category  $\sfD(\Mod kG)$ and the stable module category 
$\StMod(kG)$ is given by a \emph{recollement},\index{recollement}  as follows.

We write $\KacInj{kG}$\index{KacInj@$\KacInj{kG}$} for the full subcategory of $\KInj{kG}$ given by the acyclic complexes\index{acyclic!complex} of injective $kG$-modules. Every acyclic complex is a Tate resolution\index{Tate!resolution} of a module, namely the image of the middle map in the complex. Homotopy classes of maps between acyclic complexes correspond to homomorphisms in the stable module category. Thus Tate resolutions give an equivalence of categories
\[ 
\StMod(kG) \simeq \KacInj{kG}\,. 
\]
The next result is from \cite{Benson/Krause:2008a}.

\begin{theorem}
\label{th:recollement}
\pushQED{\qed}
There is a recollement
\begin{equation*}
\KacInj{kG} \ \begin{smallmatrix} \Hom_k(tk,-) \\
\hbox to 50pt{\leftarrowfill} \\ \hbox to 50pt{\rightarrowfill} \\ 
\hbox to 50pt{\leftarrowfill} \\ - \otimes_k tk
\end{smallmatrix} \ \KInj{kG} \ \begin{smallmatrix} \Hom_k(pk,-) \\
\hbox to 50pt{\leftarrowfill} \\ \hbox to 50pt{\rightarrowfill} \\ 
\hbox to 50pt{\leftarrowfill} \\ - \otimes_k pk
\end{smallmatrix} \ \sfD(\Mod kG).  \qedhere
\end{equation*}
\end{theorem}

Thus the category $\KInj{kG}$ can be thought of as being glued 
together from the categories $\StMod(kG)$ and $\sfD(\Mod kG)$.

The compact objects in these categories are only preserved by the left adjoints, giving us the following sequence:
\[
 \stmod(kG) \xleftarrow[\ - \otimes_k tk\ ]{}
\sfD^{\sfb}(\mod kG) \xleftarrow[\ - \otimes_k pk\ ]{} \sfD^{\sfb}(\proj kG). 
\]
This expresses $\stmod(kG)$ as the quotient of $\sfD^{\sfb}(\mod kG)$ by the perfect complexes. This was first proved by Buchweitz~\cite{Buchweitz:1986}; see also Rickard~\cite{Rickard:1989a}.

\subsection{Varieties for objects in $\KInj{kG}$}
We now introduce a notion of support for objects in $\KInj{kG}$. To begin with note that the graded endomorphism ring of $ik$ is $H^*(G,k)$, which is graded commutative and noetherian. For each object $X$ in $\KInj{kG}$ there is a homomorphism
\[ 
H^*(G,k)=\Ext^*_{kG}(ik,ik) \xrightarrow{-\otimes_k X} \Ext^*_{kG}(X\otimes_{k}ik,X\otimes_{k}ik) \cong \Ext^*_{kG}(X,X)\,,
\]
where the isomorphism is from Proposition~\ref{prop:tensorid}. This allows us to apply the technology developed in Sections~\ref{sec:Wednesday1} and \ref{sec:Wednesday2} in this context.

Let $\mcV_G$ be the set of homogeneous prime ideals in $H^*(G,k)$ (including the maximal one). Then for each $\fp\in\mcV_G$ we have a local cohomology functor 
\[
\gam_\fp\col \KInj{kG}\to \KInj{kG}\,.
\]

\begin{definition}
For an object $X$ in $\KInj{kG}$, we define
\[ 
\mcV_G(X)= \{\fp \in \mcV_G \mid \gam_\fp X \ne 0\}. 
\]
\end{definition}

The assertions in the result below are all obtained as special cases of results concerning local cohomology and support for triangulated categories, discussed in the previous lectures.

\begin{proposition}
The assignment $X\mapsto \mcV_{G}(X)$ has the following properties:
\begin{enumerate}
\item $\mcV_G(X)=\varnothing$ if and only if $X = 0$.
\item $\mcV_G(X \oplus Y) = \mcV_G(X) \cup \mcV_G(Y)$, and more generally
\[ 
\mcV_G(\bigoplus_\alpha X_\alpha)=\bigcup_\alpha\mcV_G(X_\alpha). 
\]
\item $\mcV_G(\gam_\fp ik) = \{\fp\}$.
\item For any $\mcV\subseteq\mcV_G$, there exists an object $X$ in $\KInj{kG}$ with $\mcV_G(X)=\mcV$. \qed
\end{enumerate}
\end{proposition}

Properties of support that are specific to $\KInj{kG}$, including a tensor product formula for supports:
\[
\mcV_G(X \otimes_k Y) = \mcV_G(X) \cap \mcV_G(Y)
\]
will be deduced as a consequence of results presented in later lectures.

\subsection{Comparison with $\StMod(kG)$}

For objects in the subcategory 
\[ 
\KacInj{kG}\simeq\StMod(kG) 
\]
of $\KInj{kG}$ the definition of $\mcV_G(X)$ given in the previous section agrees with that of Definition \ref{def:mcVGM}. Since the object $k$ in $\StMod(kG)$ corresponds to $tk$ in $\KacInj{kG}$ we write $\mcV_G(tk)$ for what was denoted $\mcV_G(k)$ in the discussion of varieties for $\StMod(kG)$. Thus we have
\[ 
\mcV_G=\mcV_G(tk) \cup \{H^{\ges 1}(G,k)\} 
\]
where $H^{\ges 1}(G,k)$ is the maximal ideal of positive degree elements.

Referring back to the recollement (Theorem~\ref{th:recollement}), 
$X$ is isomorphic to $X \otimes_k tk$ if and only if $X$ is in
$\KacInj{kG}$. If $X$ is not acyclic then
\[
\mcV_G(X)=\mcV_G(X\otimes_k tk)\cup\{H^{\ges 1}(G,k)\} \,.
\]

\begin{remark}
It is not necessary to understand the rest of this lecture for the goals of this seminar. Our purpose is to place $\KInj{kG}$ in a wider context.
\end{remark}

\subsection{$\KInj{B}$ as a derived invariant}

Recall that the \emph{blocks}\index{block} of a group algebra $kG$ are the indecomposable two sided ideal direct factors.  So the block decomposition of $kG$ is of the form
\[ 
B_0 \times \dots \times B_s\,. 
\]
This decomposition is unique, and every indecomposable $kG$-module is a module for $B_i$ for a unique value of $i$. 

\begin{theorem}
Let $B$ and $B'$ be blocks of group algebras. The following conditions are equivalent:
\begin{enumerate}
\item There is a tilting complex\index{tilting complex} 
over $B$ whose endomorphism ring in $\sfD^{\sfb}(\mod B)$ is isomorphic to $B'$
\item $\sfD^{\sfb}(\mod B)$ and $\sfD^{\sfb}(\mod B')$ are triangle equivalent
\item $\sfD(\Mod B)$ and $\sfD(\Mod B')$ are triangle equivalent
\item $\KInj B$ and $\KInj{B'}$ are triangle equivalent. \qed
\end{enumerate}
\end{theorem}

The equivalence of the first three of these is Rickard's theorem.

\subsection{$\KInj{kG}$ is a derived category}

Given complexes $X$ and $Y$ of $kG$-modules, we form a complex $\fHom_{kG}(X,Y)$ whose $n$th component is
\[ 
\prod_{m\in\bbZ}\fHom_{kG}(X_m,Y_{n+m}) 
\]
with differential given by
\[ 
(d(f))(x)=d(f(x)) - (-1)^{|f|}f(d(x))\,. 
\]
Composition of maps makes $\fEnd_{kG}(X)=\fHom_{kG}(X,X)$ into a \emph{differential graded algebra}\index{differential graded!algebra} over which $\fHom_{kG}(X,Y)$ is a \emph{differential graded module}\index{differential graded!module}.

\begin{theorem}
\index{Keller's Theorem}
If $C$ is a compact generator for $\KInj{kG}$ then
\[ 
\fHom_{kG}(C,-) \colon \KInj{kG} \to \sfD(\fEnd_{kG}(C))
\]
is an equivalence of categories.

If $C$ is just a compact object, which does not necessarily generate, then
\[ 
\fHom_{kG}(C,-) \colon \Loc(C) \to \sfD(\fEnd_{kG}(C))
\]
is an equivalence of categories.
\end{theorem}

\begin{proof}
This is a direct application of  Monday's Exercise 23.
\end{proof}

In this theorem, we have used $\sfD(A)$ to denote the \emph{derived category}\index{derived category} of a differential graded algebra $A$. To form this, we first form the category whose objects are the differential graded modules and whose arrows are the homotopy classes of degree preserving maps. Then we invert the quasi-isomorphisms\index{quasi-isomorphism} just as we did in the usual derived category of a ring.

If the differential graded algebra is just a ring concentrated in degree zero, with zero differential, then a differential graded module is the same as a complex of modules and we recover the usual definition of the derived category of  a ring.

\begin{observation} 
If $G$ is a $p$-group, then $k$ is the only simple $kG$-module, so that $ik$ is a compact generator for $\KInj{kG}$; see Theorem~\ref{thm:kinj-generation}. Therefore
\[ 
\KInj{kG} \simeq \sfD(\fEnd_{kG}(ik))\,.
\]
For a non-$p$-group, we just get an equivalence between the localising subcategory of $\KInj{kG}$ generated by $ik$ and $\sfD(\fEnd_{kG}(ik))$.
\end{observation}

The next result, which goes by the name \emph{Rothenberg--Steenrod construction}, provides a link to algebraic topology:

\begin{theorem}
\index{Rothenberg--Steenrod!construction}
For any path-connected space $X$ we have a quasi-isomorphism of differential graded algebras
\[ 
\fEnd_{C_*(\Omega X;k)}(k) \simeq C^*(X;k) 
\]
where $\Omega X$ is the loop space of $X$.\qed
\end{theorem}

In case $X=BG$, the connected components of $\Omega X$ are contractible, and we have $\Omega X \simeq G$. So the differential graded algebra  $C_*(\Omega X;k)$ is quasi-isomorphic to the group algebra $kG$. So the Rothenberg--Steenrod construction gives
\[ 
\fEnd_{C_*(\Omega X;k)}(k) \simeq \fEnd_{kG}(ik) \simeq C^*(BG;k). 
\]

\begin{theorem}
If $G$ is a finite $p$-group, there are equivalences of categories
\[ 
\KInj{kG} \simeq \sfD(\fEnd_{kG}(ik)) \simeq \sfD(C^*(BG;k)) \,.
\]
\end{theorem}

For a more general finite group, the same argument shows that the right side is equivalent to the subcategory $\Loc(ik)$ of $\KInj{kG}$.

Under this equivalence the tensor product $-\otimes_k -$ with diagonal $G$-action on the left hand side corresponds to the derived $E_\infty$ tensor product on the right hand side. This is remarkable, for the latter
takes a great deal of topological machinery to develop, see for example Elmendorf, Kriz, Mandell, and May~\cite{Elmendorf/Kriz/Mandell/May:1996a}.

We have the following dictionary relating operations in $\KInj{kG}$ with operations in $\sfD(C^*(BG;k))$. Note that this dictionary turns things upside down, so that induction corresponds to restriction and
restriction corresponds to coinduction.\medskip

\begin{center}
\renewcommand{\arraystretch}{1.6}
\begin{tabular}{|c|c|}
\hline
$\KInj{kG}$ & $\sfD(C^*(BG;k))$ \\ \hline\hline
$- \otimes_k -$ & $- \lotimes_{C^*(BG;k)} -$ \\
diagonal $G$-action & $E_\infty$ tensor product \\ \hline
induction, $-_H{\uparrow^G}$ & restriction \\
$kG\otimes_{kH}-$ & via $C^*(BG;k)\to C^*(BH;k)$ \\ \hline
restriction, $-_G{\downarrow_H}$ & coinduction \\
via $kH\to kG$ & $\Hom_{C^*(BG;k)}(C^*(BH;k),-)$ \\ \hline
$ik$ & $C^*(BG;k)$ \\ \hline
$kG$ & $k$ \\ \hline
\end{tabular}
\end{center}

\section{Exercises}
\label{exer:Wednesday}
Take a hike and eat a cake.

\chapter{Thursday}
\thispagestyle{empty} In the last chapter we introduced a notion of
support for triangulated categories. The starting point in this one is
a notion that we call `stratification' for an $R$-linear triangulated
category $\sfT$. It identifies conditions under which support can be
used to parameterise localising subcategories of $\sfT$. The crux of
the stratification condition is local in nature, in that it involves
only the subcategories $\gam_{\fp}\sfT$, so can be, and usually is,
verified one prime at a time. We illustrate this technique by
outlining the proof of the main results of this seminar; all this is
part of Section~\ref{sec:Thursday1}. At first glance, the
stratification condition is rather technical and of limited scope. To
counter this, in Section~\ref{sec:Thursday2} we discuss a number of
interesting consequences that follow from this property. The last
section has a different flavour: it makes concrete some of the ideas
and constructions we have been discussing by describing them in the
case of the Klein four group.

\section{Stratifying triangulated categories}
\label{sec:Thursday1}
In this lecture we identify under what conditions the notion of support classifies localising subcategories of
a triangulated category. Most of the material in this lecture is based on \cite{Benson/Iyengar/Krause:bik2}.

\subsection{Classifying localising subcategories}\index{localising subcategory}
Recall the setting from Wednesday's lectures: $\sfT$ is an $R$-linear triangulated category, meaning that $\sfT$ is a compactly generated triangulated category with small coproducts and $R$ is a noetherian graded commutative ring with a given homomorphism $R \to Z^{*}(\sfT)$ to the graded centre of $\sfT$. This amounts to giving for each object  $X$ in $\sfT$ a homomorphism of graded rings 
\[ 
R\to\End^*_\sfT(X), 
\] 
and these homomorphisms satisfy the following compatibility condition. Given a morphism $\phi\colon X\to Y$ in $\sfT$, the diagram
\[ 
\xymatrixcolsep{4pc}
\xymatrix{R\ar[r]\ar[d]& \End^*_\sfT(X)\ar[d]^{\Hom^{*}_{\sfT}(X,\phi)} \\ 
\End^*_\sfT(Y) \ar[r]^{\Hom^{*}_{\sfT}(\phi,Y)} & \Hom^{*}_{\sfT}(X,Y) } 
\]
commutes, up to the expected sign. We set
\[ 
\Spec R =\{\text{homogeneous prime ideals of $R$}\}. 
\]
For each $\fp\in\Spec R$ there is a \emph{local cohomology functor}\index{local cohomology!functor} which is an exact functor $\gam_\fp \colon\sfT\to\sfT$. For each object $X$ in $\sfT$ there is a \emph{support}\index{support} 
\[ 
\supp_RX = \{\fp \in\Spec R\mid \gam_\fp X \ne 0\}. 
\]
The support of $\sfT$ is the set
\[
\supp_R\sfT = \{\fp \in\Spec R \mid \gam_\fp \sfT\ne 0\}\,.
\]
Our goal in this lecture is to examine under what conditions this notion of support classifies localising subcategories. We have a map
\[ 
\{\text{localising subcategories of $\sfT$}\} 
\xrightarrow{\sigma}
\{\text{subsets of $\supp_R T$}\}  
\] 
given by
\[ 
\sfS \mapsto \bigcup_{X\in\sfS}\supp_R X. 
\]
When this map is a bijection we say that support \emph{classifies} localising subcategories.

\begin{remark}
If $\sigma$ is a bijection then the inverse map
\[ \{\text{localising subcategories of $\sfT$}\} \xleftarrow{\tau}
\{\text{subsets of $\supp_R\sfT$}\}  \] 
sends a subset $\mcV$ of 
$\supp_R\sfT$ to the localising subcategory consisting of the
modules whose support is contained in $\mcV$.
\end{remark}

\begin{proposition}
\label{pr:strat}
If $\sigma$ is a bijection, there are two consequences:
\begin{enumerate}
\item The \emph{local-global principle:}\index{local-global principle} for each object $X$, the localising
subcategory generated by $X$ is the same as that generated by
\[ 
\{\gam_\fp X \mid \fp \in \Spec R\}. 
\]
\item The \emph{minimality condition:}\index{minimality condition} for $\fp \in\supp_R\sfT$, the subcategory $\gam_\fp \sfT$ of objects supported at $\fp$ is a \emph{minimal} localising subcategory of $\sfT$.
\end{enumerate}
\end{proposition}

\begin{definition}
We say that $\sfT$ is \emph{stratified}\index{stratified} by the action of $R$ if the two conditions listed in
Proposition \ref{pr:strat} hold.
\end{definition}

We should note right away that the minimality is the critical condition; we know of no examples where the local-global principle fails; see Section~\ref{ssec:lgprinciple}.

One of the crucial observations of \cite{Benson/Iyengar/Krause:bik2} is the converse to Proposition \ref{pr:strat}: if these two conditions hold then the map is a bijection. The map $\sigma\tau$ is clearly the identity. To see that $\tau\sigma$ is the identity, we use both the local-global principle and the minimality condition. We'll prove a more general statement in Theorem \ref{th:lg-classifies}, later in this lecture.

\subsection{Minimality}

It is useful to have a criterion for checking the minimality condition.
The following test is Lemma 4.1 of \cite{Benson/Iyengar/Krause:bik2}.

\begin{lemma}
\label{le:minimality}
Assume that $\sfT$ is compactly generated.  A non-zero localising subcategory
$\sfS$ of $\sfT$ is minimal if and only if for every pair of non-zero objects
$X$ and $Y$ in $\sfS$ we have $\Hom_\sfT^*(X,Y)\ne 0$.
\end{lemma}

\begin{proof}
If $S$ is minimal then $Y\in\Loc_\sfT(X)$ so if $\Hom^*_\sfT(X,Y)=0$
it would follow that $\Hom^*_\sfT(Y,Y)=0$ and so $Y=0$.

Conversely if $\sfS$ has a proper non-zero localising subcategory $\sfS'$,
we may assume $\sfS'=\Loc_\sfT(X)$ for some $X\ne 0$. 
Since $\sfT$ is
compactly generated, there is a localisation functor for $\sfS'$ 
(Lemma 2.1 of \cite{Benson/Iyengar/Krause:bik2}). 
So if $W\in\sfS\setminus\sfS'$, we have a triangle $W'\to W \to W''\to$
with $W'\in \sfS'$ (so $W''\ne 0$) and $\Hom_\sfT(X,W'')=0$. 
\end{proof}

What happens if just the local-global principle holds, without 
assuming that the minimality condition holds?
In this case, we have to take into account the localising subcategories
of $\gam_\fp \sfT$.

\begin{theorem}
\label{th:lg-classifies}
Set $\mcV=\supp_R\sfT$. When the local-global principle holds, there are inclusion preserving bijections
\[
\left\{
\begin{gathered}
\text{Localizing}\\ \text{subcategories of $\sfT$}
\end{gathered}\; \right\} 
\xymatrix@C=3pc {\ar@<1ex>[r]^-{{\sigma}} & \ar@<1ex>[l]^-{{\tau}}}
\left\{
\begin{gathered}
  \text{Families $(\sfS(\fp))_{\fp\in\mcV}$ with $\sfS(\fp)$ a}\\
  \text{localizing subcategory of $\gam_{\fp}\sfT$}
\end{gathered}\;
\right\}
\] 
defined by ${\sigma}(\sfS)=(\sfS\cap\gam_{\fp}\sfT)_{\fp\in\mcV}$ and
${\tau}(\sfS(\fp))_{\fp\in\mcV}=\Loc_{\sfT}\big(\sfS(\fp)\mid \fp\in \mcV\big)$.
\end{theorem}

\begin{proof}
We first claim that $\sigma\tau$ is the identity. Set $\sfS=\Loc(\sfS(\fp)\mid\fp \in\mcV)$. Since  $\gam_\fp \gam_\fq=0$ for $\fq\ne\fp$, we have $\gam_\fp \sfS=\sfS(\fp)$. Since $\gam_\fp \gam_\fp=\gam_\fp$ we therefore have 
$\sfS\cap\gam_\fp \sfT\subseteq\gam_\fp \sfS=\sfS(\fp)$. The reverse inclusion is obvious.

Next we claim that $\tau\sigma$ is the identity. Let $\sfS$ be a localising subcategory of $\sfT$. We have
$\tau\sigma(\sfS)=\Loc_{\sfT}(\sfS\cap\gam_\fp \sfT\mid\fp \in\mcV)\subseteq \sfS$. So letting $X$ be in $\sfS$, we must prove that $X$ is in $\tau\sigma(S)$. Using the local-global principle, we have 
\[
\gam_\fp X\in \sfS\cap\gam_\fp \sfT\subseteq \tau\sigma(S)\,,
\]
and so $X\in\Loc_{\sfT}(\gam_\fp X\mid\fp\in \supp_{R}X)\subseteq\tau\sigma(S)$.
\end{proof}

\subsection{When does the local-global principle hold?}
\label{ssec:lgprinciple}

Turning now to the local-global principle, one has the following:

\begin{theorem}
Let $\sfT$ be compactly generated and $X$ an object in $\sfT$ satisfying $\dim\supp_R X < \infty$. Then $\Loc_\sfT(X)=\Loc_\sfT\{\gam_\fp X\mid \fp \in\supp_R X\}$. \qed
\end{theorem}

As an immediate consequence, one gets:

\begin{corollary}
If $\dim\Spec R$ is finite, then the local-global principle holds. \qed
\end{corollary}

The local-global principle also holds automatically in the context of tensor triangulated categories. 
We discuss these next.

\subsection{Tensor Triangulated Categories}
\label{ssec:ttcats}
Let $(\sfT,\otimes,\one)$ be a \emph{tensor triangulated category}.\index{tensor!triangulated category}
For us, this means:
\begin{itemize}
\item $\sfT$ is a triangulated category,
\item $\otimes\colon\sfT\times\sfT\to\sfT$ is a \emph{symmetric monoidal}
tensor product\index{symmetric monoidal} (i.e., commutative and associative
up to coherent natural isomorphisms),
\item $\otimes$ is exact in each variable and preserves 
small coproducts,
\item $\one$ is a unit for the tensor product, and is compact.
\end{itemize}

The \emph{Brown representability theorem}\index{Brown representability!theorem} implies that there exist \emph{function objects}.\index{function object} Given objects $X$ and $Y$ in $\sfT$ there is an object $\fHom(X,Y)$ in $\sfT$, 
contravariant in $X$ and covariant in $Y$, together with an adjunction
\[ 
\Hom_\sfT(X\otimes Y,Z)\cong \Hom_\sfT(X,\fHom(Y,Z)). 
\]
By construction, $\fHom(X,Y)$ will be functorial in $X$, and we assume also functoriality in $Y$; all examples encountered in these notes have this property.
 
Since $\sfT$ is tensor triangulated category, the ring $\End^*_\sfT(\one)$ is \emph{graded commutative}.\index{graded commutative!ring} It acts on $\sfT$ via
\[ 
\End^*_\sfT(\one) \xrightarrow{-\otimes X}\End^*_\sfT(X) 
\]
If $R$ is noetherian and graded commutative then any homomorphism $R\to\End^*_\sfT(\one)$ of rings induces an action $R\to Z(\sfT)$. Such an action is said to be \emph{canonical}.\index{canonical action}

If the action of $R$ is canonical, the adjunction defining function objects is $R$-linear. Using this, we get
\[ 
\gam_\mcV X \cong X \otimes \gam_\mcV\one,\qquad L_\mcV X \cong X \otimes L_\mcV\one, \qquad
\gam_\fp X \cong X \otimes \gam_\fp \one 
\]
Here, $\mcV$ is a specialisation closed subset and $\fp$ is a point in $\Spec R$.

In a tensor triangulated category, we talk of \emph{tensor ideal}\index{tensor!ideal} localising subcategories:\index{localising subcategory} $\sfS$ is tensor ideal if $X\in\sfS$, $Y\in\sfT$ implies $X \otimes Y \in \sfS$.

\begin{definition}
We write $\Loc^\otimes_\sfT(\sfS)$\index{Loc1@$\Loc^\otimes_\sfT$} for the smallest tensor ideal localising subcategory of $\sfT$ containing a subcategory (or a collection of objects) $\sfS$.
\end{definition}

The following result is \cite[Theorem~7.2]{Benson/Iyengar/Krause:bik2}.

\begin{theorem}
\label{thm:localglobal-tt}
Let $\sfT$ be a tensor triangulated category with canonical $R$-action. Then for each object $X$ in $\sfT$ we have
\[ 
\Loc^\otimes_\sfT(X)=\Loc^\otimes_\sfT(\gam_\fp X \mid \fp \in\Spec R). 
\]
In particular, if $\one$ generates $\sfT$ then the local global principle holds for $\sfT$.\qed
\end{theorem}

In a tensor triangulated category $\sfT$ with a canonical action of a graded commutative noetherian ring $R$, the localising subcategories defined by the functors $\gam_\fp$ are all tensor ideal. It therefore makes sense to restrict our attention to the classification of tensor ideal localising subcategories. In this context, we have
the following analogue of Theorem~\ref{th:lg-classifies}.

\begin{theorem}
\label{th:lg-tt}
Let $\sfT$ be a tensor triangulated category with canonical $R$-action. Set $\mcV=\supp_R\sfT$. Then there is bijection 
\[
\left\{
\begin{gathered}
\text{Tensor ideal localizing}\\ \text{subcategories of $\sfT$}
\end{gathered}\; \right\} 
\xymatrix@C=3pc {\ar@<1ex>[r]^-{{\sigma}} & \ar@<1ex>[l]^-{{\tau}}}
\left\{
\begin{gathered}
  \text{Families $(\sfS(\fp))_{\fp\in\mcV}$ with $\sfS(\fp)$ a tensor}\\
  \text{ ideal localizing subcategory of $\gam_{\fp}\sfT$}
\end{gathered}\;
\right\}
\] 
\end{theorem}

The map $\sigma$ and its inverse $\tau$ are defined in the same way as in Theorem \ref{th:lg-classifies}, and the proof is essentially the same as the proof of that theorem.

Since the local-global principle automatically holds for tensor ideal localising subcategories, we say that a tensor triangulated category $\sfT$ is \emph{stratified} by $R$ if every $\gam_\fp \sfT$ is either minimal among tensor ideal localising subcategories or is zero. Under this condition, the maps $\sigma$ and $\tau$ establish a
bijection between tensor ideal localising subcategories of $\sfT$ and subsets of $\supp_R\sfT$.

The analogue of the minimality test of Lemma \ref{le:minimality} in the tensor triangulated situation is as follows.

\begin{lemma}
\label{le:minimality-tensor}
Assume that $\sfT$ is a compactly generated tensor triangulated category with small coproducts. A non-zero tensor ideal localising subcategory $\sfS$ of $\sfT$ is minimal if and only if for every pair of non-zero objects
$X$ and $Y$ in $\sfS$ there exists an object $Z$ such that we have $\Hom_\sfT^*(X\otimes Z,Y)\ne 0$. \qed
\end{lemma}

In fact, the object $Z$ can be chosen independently of $X$ and $Y$. For example, a compact generator for $\sfT$ always suffices.

\subsection{The stable module category} 

The category $\StMod(kG)$ is a tensor triangulated category,  so the local-global principal holds. 
In \cite{Benson/Iyengar/Krause:bik3} we prove:

\begin{theorem}
\label{th:strat-StMod}
The tensor triangulated category $\StMod(kG)$ is stratified by the canonical action of $H^*(G,k)$. \qed
\end{theorem}

What we have seen in this lecture is that in order to prove this, we only need to see that each $\gam_\fp \StMod(kG)$ is minimal or zero. In fact the only case in which we get zero is if $\fp$ is the maximal ideal of positive degree elements. So the support of $\StMod(kG)$ is the set of non-maximal homogeneous primes in $\mcV_{G}=\Spec H^{*}(G,k)$. As a direct consequence one gets:

\begin{corollary}
Support defines a one to one correspondence between tensor ideal localising subcategories of $\StMod(kG)$ and subsets of the set of non-closed points in $\mcV_G$. \qed
\end{corollary}

We prove Theorem~\ref{th:strat-StMod} by first establishing the corresponding result for the homotopy category, $\KInj{kG}$, discussed next.

\subsection{The category $\KInj{kG}$}
\label{ssec:strategy}
Again this is a tensor triangulated category, so the local-global principal holds. The following is the main result of \cite{Benson/Iyengar/Krause:bik3}:

\begin{theorem}
\label{th:strat-KInj}
The tensor triangulated category $\KInj{kG}$ is stratified by the canonical action of $H^*(G,k)$. \qed
\end{theorem}

This time, every point in $\mcV_{G}$, the homogenous spectrum of $H^{*}(G,k)$, including the maximal ideal, appears in the support of $\KInj{kG}$. 

\begin{corollary}
Support defines a one to one correspondence between tensor ideal localising subcategories of $\KInj{kG}$ and subsets of the set of all subsets of $\mcV_G$.\qed
\end{corollary}

\subsection{Overview of classification for $\Mod(kG)$}
\label{Overview of classification}
We finish this lecture with an outline of the strategy for
proving the classification theorem for subcategories of
$\Mod(kG)$ satisfying the two conditions of Definition~\ref{def:Mod-loc}
together with the tensor ideal condition. The main theorem states
that these are in one to one correspondence with subsets of the
set of non-maximal homogeneous prime ideals in $H^*(G,k)$.\bigskip

\noindent
{\bf Step 1.} Go down from $\Mod(kG)$ to $\StMod(kG)$. Every non-zero 
subcategory of
$\Mod(kG)$ satisfying the given conditions contains the projective modules,
and passes down to a tensor ideal localising subcategory of $\StMod(kG)$. This
gives a bijection with the set of
tensor ideal localising subcategories of $\StMod(kG)$,
so we are reduced to proving Theorem \ref{th:strat-StMod}.\bigskip

\noindent
{\bf Step 2.} Go up to $\KInj{kG}$. According to the discussion in
Lecture 3, each tensor ideal localising subcategory of $\StMod(kG)$ corresponds
to two such for $\KInj{kG}$. One is the image under the inclusion
\[ 
\StMod(kG)\simeq\KacInj{kG} \to \KInj{kG} 
\]
and the other is generated by this  together with the extra object $pk$. So we are reduced to proving Theorem \ref{th:strat-KInj}.\bigskip

\noindent
{\bf Step 3.} 
Reduce from $\KInj{kG}$ to $\KInj{kE}$ for $E$ an elementary abelian $p$-group using the Quillen Stratification Theorem and a suitably strengthened version of Chouinard's Theorem. This is the subject of the last lecture on Friday.\bigskip

\noindent
{\bf Step 4.} 
If $p=2$, we can go  from $\KInj{kE}$ to differential graded modules over a graded polynomial ring, viewed as a differential graded algebra, using a version of the Bernstein--Gelfand--Gelfand correspondence. This is explained in Section~\ref{sec:Friday2}.

The proof is rather more complicated when $p$ is odd. We still get down to a graded polynomial ring, but the reduction involves a number of steps. The first one is to pass to the Koszul complex of group algebra $kE$. The Koszul complex can be viewed as a differential graded algebra, and then it is quasi-isomorphic to a graded exterior algebra, since $kE$ is a complete intersection ring. We then again apply a version of the  Bernstein--Gelfand--Gelfand correspondence to get to a graded polynomial ring. The technical tools needed to execute this proof require considerable preparation, beyond what is already presented in these lectures, so we can do no more than refer the interested readers to the source~\cite{Benson/Iyengar/Krause:bik3}.

\bigskip

\noindent
{\bf Step 5.} 
Deal directly with a graded polynomial ring using the minimality condition discussed in this lecture.
This is addressed in Section~\ref{sec:Friday1}.

\section{Consequences of stratification}
\label{sec:Thursday2}
In the last lecture we learnt about the stratification condition and its connection to the problem of classifying localizing subcategories of triangulated categories. In this lecture we present some other, not immediately obvious, consequences of stratification that serve to illustrate how strong a condition it is, and how useful it is when one can establish that it holds in some context. The basic reference for the material presented below is again \cite{Benson/Iyengar/Krause:bik2}.

For most of the lecture $\sfT$ will be an $R$-linear triangulated category as in Section~\ref{sec:Wednesday1}; we tackle the tensor triangulated case at the very end. Recall that $\sfT$ is said to be stratified\index{stratified} by $R$ if: 
\begin{enumerate}[\quad\rm(S1)]
\item
The local global principle holds; see \ref{pr:strat}. An equivalent condition is that 
\[
\gam_{\mcV}X \in \Loc_{\sfT}(\gam_{\fp}X\mid \fp\in\supp_{R}X)
\]
for any specialisation closed subset $\mcV$ of $\Spec R$; see \cite[Theorem~3.1]{Benson/Iyengar/Krause:bik2}.
\item
$\gam_{\fp}\sfT$ is minimal for each $\fp\in\supp_{R}\sfT$.
\end{enumerate}
Recall that an object $X$ of $\sfT$ is in $\gam_{\fp}\sfT$ if and only if $H^{*}_{C}(X)$ is $\fp$-local and $\fp$-torsion for all $C\in\sfT^{\sfc}$; see Corollary~\ref{cor:plocptor}. Condition (S2) is the statement that for any such $X\ne 0$, one has $\Loc_{\sfT}(X)=\gam_{\fp}\sfT$; that is to say, $X$ builds any other $\fp$-local and $\fp$-torsion object.

\begin{remark}
The local global principle (S1) holds when the Krull dimension of $R$ is finite, or if $\sfT$ is tensor triangulated with $\Loc_{\sfT}(\one)=\sfT$ and the ring $R$ acts on $\sfT$ via a homomorphism $R\to \End^{*}_{\sfT}(\one)$.
\end{remark}

\begin{example}
When $A$ is a commutative noetherian ring, $\sfD(A)$ is stratified by the canonical $A$-action. Indeed, (S1) holds because $\sfD(A)$ is suitably tensor triangulated, while (S2) is by a theorem of Neeman~\cite{Neeman:1992a}. We will present a proof of this result in Section~\ref{sec:Friday1}.
\end{example}

\begin{example}
When $G$ is a finite group and $k$ a field with $\Char k$ dividing $|G|$, both $\StMod(kG)$ and $\KInj{kG}$ are stratified by canonical actions of the cohomology ring $\HH *{G;k}$. This is proved in \cite{Benson/Iyengar/Krause:bik3}, and is the focal point of this seminar.
\end{example}

\subsection{Classification theorems}

For the remainder of this lecture $\sfT$ will be a triangulated category stratified by an action of a graded commutative ring $R$. The first consequence is that localising subcategories are parameterised by subsets of $\supp_{R}\sfT$, which is something that was discussed already in the previous lecture.

\begin{theorem}
\label{thm:strat-class}
The maps assigning a subcategory $\sfS$ to its support, $\supp_{R}\sfS$, induces a bijection
\[
\left\{
\text{Localizing subcategories of $\sfT$}
\right\} 
\xra{\ \supp_{R}(-)\ }
\left\{
  \text{Subsets of $\supp_{R}\sfT$}
\right\}\,.
\] 
Its inverse sends a subset $\mcU$ of $\supp_{R}\sfT$ to $\{X\in\sfT\mid \supp_{R}X\subseteq \mcU\}$. \qed
\end{theorem}

The following statement is an immediate consequence of the theorem; on the other hand, it is not hard to prove that the corollary implies stratification.

\begin{corollary}
\label{cor:strat-building}
If $\supp_{R}X\subseteq \supp_{R}Y$, then $X$ is in $\Loc_{\sfT}(Y)$.\qed
\end{corollary}

Next we explain how, under further, though mild, hypotheses on $\sfT$ stratification implies a classification of the thick subcategories of compact objects.

\begin{definition}
We say that an $R$-linear triangulated category $\sfT$ is \emph{noetherian}\index{noetherian category} if for all compact objects $C$ in $\sfT$ the $R$-module $\End_{\sfT}^{*}(C)$ is finitely generated; equivalently, if the $R$-module $\Hom_{\sfT}^{*}(C,D)$ is finitely generated for all $D$ compact.
\end{definition}

The derived category of a commutative noetherian ring $R$, viewed as an $R$-linear category, is noetherian: This is easy to verify, once you accept that the compact objects are the perfect complexes; see Theorem~\ref{thm:compacts-DA}. Another example of a noetherian category is $\KInj{kG}$, the homotopy category of complexes of injectives
of a finite group $G$, viewed as $H^{*}(G,k)$-linear categories. This is a restatement of Even's theorem~\ref{thm:Evens}, given the identification of compact objects in the categories,  Theorem~\ref{thm:kinj-generation}.
Note that the stable module category $\StMod(kG)$ is not usually noetherian, see Theorem~\ref{thm:tate-nonnoetherian}; however, it embeds in the noetherian category $\KInj{kG}$, by Theorem~\ref{th:recollement}, and this usually suffices for applications.

One consequence of the noetherian condition is that supports of compact objects are closed. A more precise statement is proved in \cite[Theorem~5.5]{Benson/Iyengar/Krause:2008a}.

\begin{theorem}
\label{thm:fg-support}
If $C$ is a compact object in $\sfT$ and the $R$-module $\End_{\sfT}^{*}(C)$ is noetherian, then 
\[
\supp_{R}C = \mcV(\fa)\quad\text{where $\fa=\Ker(R\to \End^{*}_{\sfT}(C))$.}
\]
In particular, $\supp_{R}C$ is a closed subset of $\Spec R$. \qed
\end{theorem}

From the stratification condition on $\sfT$ one gets:

\begin{theorem}
\label{thm:strat-thick}
When the $R$-linear triangulated  category $\sfT$ is stratified and noetherian, there are inclusion preserving inverse-bijections
\[
\left\{
\begin{gathered}
\text{Thick subcategories}\\ \text{of compact objects in $\sfT$}
\end{gathered}\; \right\} 
\xymatrix@C=3pc {\ar@<1ex>[r]^-{{\supp_{R}(-)}} & \ar@<1ex>[l]^-{{\supp_{R}^{-1}(-)}}}
\left\{
\begin{gathered}
  \text{Specialisation closed }\\
  \text{subsets of $\supp_{R}\sfT$}
\end{gathered}\;
\right\}
\] 
where $\supp_{R}^{-1}(\mcV)=\{C\in\sfT^{\sfc}\mid \supp_{R}C\subseteq \mcV\}$, for $\mcV$ specialisation closed.
\end{theorem}

\begin{proof}
There are two crucial issues: One is that for compact objects $C,D$, if $\supp_{R}C\subseteq\supp_{R}D$, then $C$ is in $\Thick_{\sfT}(D)$; second, any closed subset of $\supp_{R}\sfT$ is the support of some compact object.
We leave the latter to the exercises, and sketch an argument for the former.

When $\supp_{R}C\subseteq \supp_{R}D$ holds, it follows from Corollary~\ref{cor:strat-building} that $C$ is in the localising subcategory generated by $D$. Since $C,D$ are compact, Theorem~\ref{thm:neeman-compacts} implies that $C$ is in fact in the thick subcategory generated by $D$, as desired.

A different proof for this part of the proof is presented in \cite[Section~6]{Benson/Iyengar/Krause:bik2}.
\end{proof}

\subsection{Orthogonality}

A central problem in any additive category is to understand when there are non-zero morphisms between a given pair of objects. Lemma~\ref{le:minimality} makes it clear that this is at the heart of the stratification property for the triangulated categories we have been considering. Conversely, the stratification condition allows us to give fairly precise answers to this problem. The one below is from \cite[Section~5]{Benson/Iyengar/Krause:bik2}. For a proof in the case of tensor triangulated categories, see Theorem~\ref{thm:tensor-product} and the remarks following it.

\begin{theorem}
\label{thm:orthogonality}
When the $R$-linear triangulated category $\sfT$ is noetherian and stratified, there is an equality 
\[
\supp_{R} \Hom^{*}_{\sfT}(C, D) = \supp_{R}C \cap \supp_{R}D\,,
\]
for each pair of compact objects $C, D$ in $\sfT$. \qed
\end{theorem}
The support on the left hand side is the usual one from commutative algebra, as the $R$-module  $\Hom^{*}_{\sfT}(C, D)$ is finitely generated; see Lemma~\ref{le:supp-ann}.

As a corollary one gets the following ``symmetry of Ext vanishing'' type result that was proved for local complete intersection rings by Avramov and Buchweitz~\cite{Avramov/Buchweitz:2000a}. The corresponding result is also true for modules over group algebras, and is much easier to prove; see Exercise~{10} at the end of this chapter.

\begin{corollary}
When in addition $R^{i} =0$ holds for $i<0$, one has $\Hom^{n}_{\sfT}(C,D)=0$ for $n\gg 0$ if and only if $\Hom^{n}_{\sfT}(D,C)=0$ for $n\gg 0$.
\end{corollary}

\begin{proof}
For $R$ as in the statement, a finitely generated $R$-module $M$ satisfies $M^{i}=0$ for $i\gg 0$ if and only if $\supp_{R}M\subseteq \mcV(R^{\ges 1})$; we leave this as an exercise.

Assume $\Hom^{n}_{\sfT}(C,D)=0$ for $n\ge s$. Then $R^{\ges s}$ annihilates $\Hom^{*}_{\sfT}(C,D)$. Noting that $R^{\ges s}$ is an ideal of $R$, as $R^{i} =0$ for $i<0$, one gets the inclusion below:
\[
\supp_{R}\Hom^{*}_{\sfT}(D,C) = \supp_{R}\Hom^{*}_{\sfT}(C,D) \subseteq \mcV(R^{\ges s})=\mcV(R^{\ges 1})\,.
\]
The equality on the left is by Theorem~\ref{thm:orthogonality}, while the one of the right holds because $R^{\ges 1}$ and $R^{\ges s}$ have the same radical. The desired vanishing is thus a consequence of the exercise from the previous paragraph.
\end{proof}

In \cite[Section~5]{Benson/Iyengar/Krause:bik2}  we present variations of Theorem~\ref{thm:orthogonality} where, for example, $D$ need not be compact. Using these one can give another proof of the classification of thick subcategories of compact objects, Theorem~\ref{thm:strat-thick}.

\subsection{Tensor triangulated categories}

Let $\sfT$ be a tensor triangulated category with a canonical $R$-action. We assume that $\sfT$ is stratified\index{tensor!triangulated category!stratified} as a tensor triangulated category, so that for each $\fp\in\supp_{R}\sfT$, the subcategory $\gam_{\fp}\sfT$, which is tensor ideal, contains no non-zero tensor ideal localising subcategories.
There is then a \emph{tensor product theorem}\index{tensor!product!theorem} for support:

\begin{theorem}
\label{thm:tensor-product}
For any objects $X, Y$ in $\sfT$ there is an equality
\[
\supp_{R}(X\otimes Y) = \supp_{R} X \cap \supp_{R}Y\,.
\]
\end{theorem}

Compare this identity with the one in Theorem~\ref{thm:orthogonality}. For compact objects, these results can be deduced from each other, using $\fHom(C,D)\cong \fHom(C,\one)\otimes D$. Thus, Theorem~\ref{thm:orthogonality} may be viewed as an analogue of the tensor product theorem for categories that admit no tensor product. 

\begin{proof}
For each $\fp\in\Spec R$ there are isomorphisms
\[
\gam_{\fp}(X\otimes Y ) \cong \gam_{\fp}X \otimes \gam_{\fp}Y \cong  \gam_{\fp}X \otimes Y\,.
\]
These follow from the fact that $\gam_{\fp}(-)\cong \gam_{\fp}\one \otimes (-)$; see Section~\ref{ssec:ttcats}.
They yield an inclusion 
\[
\supp_{R}(X\otimes Y) \subseteq \supp_{R} X \cap \supp_{R}Y\,.
\]
As to the reverse inclusion: When $\gam_{\fp}X\ne 0$, the stratification condition implies that $\gam_{\fp}\one$ is in $\Loc^{\otimes}_{\sfT}(\gam_{\fp}X)$, so that $\gam_{\fp}Y$ is in $\Loc^{\otimes}(\gam_{\fp}X\otimes Y)$. Thus if $\gam_{\fp}Y\ne 0$ also holds, then $\gam_{\fp}(X\otimes Y)\ne 0$. This completes the proof.
\end{proof}

\section{The Klein four group}
\label{sec:Thursday3}
\index{Klein four group}
In this lecture we make some explicit calculations. We show how to compute the local cohomology functors using homotopy colimits. Then we illustrate this method by looking at a specific example. We use the
Klein four group, because its group algebra is of tame representation
type. Thus one has a complete classification of all finite dimensional
representations.

\subsection{Homotopy colimits}
Let $\sfT$ be an $R$-linear triangulated category, and  let $X_{1}\xra{f_{1}} X_2 \xra{f_{2}} X_3 \xra{f_{3}}\cdots$ be a sequence of morphisms in $\sfT$.  Its \emph{homotopy colimit}\index{homotopy!colimit}, denoted $\hocolim X_n$, is defined by an exact triangle
\[
\bigoplus_{n\ges1} X_{n}\stackrel{\theta}\lto \bigoplus_{n\ges1}
X_{n}\lto \hocolim X_{n}\lto
\]
where $\theta$ is the map $(\id-f_{n})$; see \cite{Bokstedt/Neeman:1993a}. 

Now fix a homogeneous element $r\in R$ of degree $d$. For each $X$ in $\sfT$ and each integer $n$ set $X_n=\Si^{nd}X$ and consider the commuting diagram
\[
\xymatrix{
X\ar@{=}[r]\ar[d]^{r^{}} &X\ar@{=}[r]\ar[d]^{r^{2}} &X\ar@{=}[r]\ar[d]^{r^{3}} &\cdots\\
X_1\ar[d]\ar[r]^r        &X_2\ar[d]\ar[r]^r         &X_3\ar[d]\ar[r]^r         &\cdots\\ \kos
Xr\ar[r]                 &\kos Xr^{2}\ar[r]         &\kos Xr^{3}\ar[r]         &\cdots }
\] 
where each vertical sequence is given by the exact triangle defining $\kos Xr^{n}$, and the morphisms in the last row are the (non-canonical) ones induced by the commutativity of the upper squares. The gist of the next result is that the homotopy colimits of the horizontal sequences in the diagram compute $\bloc_{\mcV(r)}X$ and $\gam_{\mcV(r)}X$. In \cite[Proposition~2.9]{Benson/Iyengar/Krause:bik2} we prove:

\begin{proposition}
\label{prop:hocolim-local}
\pushQED{\qed}
Let $r\in R$ be a homogeneous element of degree $d$. For each $X$ in $\sfT$ the adjunction morphisms $X\to L_{\mcV(r)} X$ and $\gam_{\mcV(r)}X\to X$ induce isomorphisms
\[
\hocolim X_n\stackrel{\sim}\lto L_{\mcV(r)}X\quad\text{and}\quad
\hocolim \Si^{-1}(\kos Xr^{n})\stackrel{\sim}\lto \gam_{\mcV(r)}X\, .\qedhere
\]
\end{proposition}

From this, a standard argument based an induction on the number of generators of $\fa$ yields the result below; see \cite[Proposition~2.11]{Benson/Iyengar/Krause:bik2}.

\begin{proposition}
\label{prop:locandkos}
\pushQED{\qed}
For each ideal $\fa$ of $\Spec R$ and object $X\in \sfT$ one has
\[
\Loc_{\sfT}(\gam_{\mcV(\fa)}X)=\Loc_{\sfT}(\kos X\fa)\,.\qedhere
\]
\end{proposition}
Note that the left hand side depends only on the radical of $\fa$; cf.~Lemma~\ref{lem:kosloc}.

\subsection{Representations of the Klein four group}

We describe the representation theory of the Klein four group\index{Klein four group!representations}, and
illustrate the abstract notions of the seminar using this example.

Let $G=\langle g_1,g_2\rangle\cong\bbZ/2\times\bbZ/2$ and let $k$ be an algebraically closed field of characteristic two. Let $kG$ be the group algebra of $G$ over $k$, and let $x_1=g_1-1$, $x_2=g_2-1$ as elements of $kG$. Then $x_1^2=x_2^2=0$, and we have
\[
kG = k[x_1,x_2]/(x_1^2,x_2^2). 
\]
We describe $kG$-modules by diagrams in which the vertices represent
basis elements as a $k$-vector space, and an edge
\[ 
\xymatrix{\overset{a}{\bullet} \ar@{-}[d]^{x_i} \\ 
\underset{b}{\bullet}} 
\]
indicates that $x_ia=b$. If there is no edge labelled $x_i$ in the downwards direction from a vertex then $x_i$ sends the corresponding basis vector to zero. For example, the group algebra $kG$ has the following diagram:
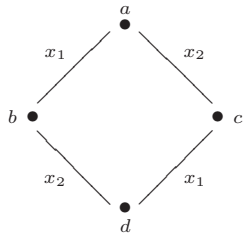
\[ 
\xymatrix{&\overset{a}{\bullet}\ar@{-}[dl]_{x_1}\ar@{-}[dr]^{x_2}\\
{\scriptstyle b\ }\bullet\ar@{-}[dr]_{x_2}&&\bullet{\scriptstyle\ c}\ar@{-}[dl]^{x_1}\\
&\underset{d}{\bullet}} 
\]
As a vector space, $kG=ka\oplus kb\oplus kc\oplus kd$. We have $\rad^2kG=\soc\,kG=kd$, $\rad\,kG=\soc^2kG=kb\oplus kc \oplus kd$.

Here are the diagrams for the syzygies of the trivial module:
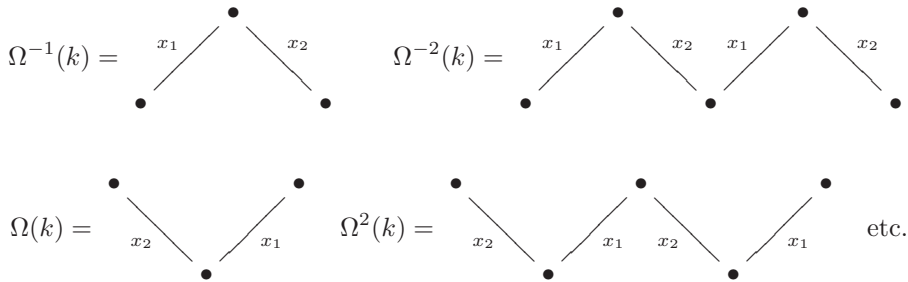
\begin{gather*} 
\Omega^{-1}(k)=\vcenter{\xymatrix{&\bullet\ar@{-}[dl]_{x_1}\ar@{-}[dr]^{x_2}\\
\bullet&&\bullet}}\qquad
\Omega^{-2}(k)=\vcenter{\xymatrix{&\bullet\ar@{-}[dl]_{x_1}\ar@{-}[dr]^{x_2}
&&\bullet\ar@{-}[dl]_{x_1}\ar@{-}[dr]^{x_2}\\
\bullet&&\bullet&&\bullet}} \\ \\
\Omega(k)=\vcenter{\xymatrix{\bullet\ar@{-}[dr]_{x_2}&&\bullet\ar@{-}[dl]^{x_1}\\
&\bullet}}\quad
\Omega^2(k)=\vcenter{\xymatrix{\bullet\ar@{-}[dr]_{x_2}&&\bullet\ar@{-}[dl]^{x_1}
\ar@{-}[dr]_{x_2}&&\bullet\ar@{-}[dl]^{x_1}\\
&\bullet&&\bullet}}\quad\text{\rm etc.}
\end{gather*}
For each non-negative integer we have
\[ 
\Ext^n_{kG}(k,k)\cong \sHom_{kG}(k,\Omega^{-n}(k))  
\]
and so $\dim_k\Ext^n_{kG}(k,k)=n+1$. In fact, the full cohomology algebra is
\[ 
H^*(G,k)=\Ext^*_{kG}(k,k)=k[\zeta_1,\zeta_2] 
\]
with $\deg(\zeta_1)=\deg(\zeta_2)=1$; see also Proposition~\ref{prop:gc-eab2}.

Pick a non-zero element $r=r_1\zeta_1+r_2\zeta_2$ in $H^1(G,k)$ (with $r_i\in k$). Then each power $r^j$ gives us an injective map from $k$ to $\Omega^{-j}(k)$ whose cokernel we denote by $L_{r^j}$. Thus we get a commutative diagram with exact columns
\[ 
\xymatrix{\ k_{\phantom{|}}\ \ar@{=}[r]\ar@{>->}^r[d] & 
\ k_{\phantom{|}}\ \ar@{=}[r]\ar@{>->}[d]^{r^2}& 
\ k_{\phantom{|}}\ \ar@{=}[r]\ar@{>->}[d]^{r^3}& \cdots\\
\Omega^{-1}(k) \ar@{->>}[d] \ar@{>->}[r]^r& 
\Omega^{-2}(k)\ar@{->>}[d]\ar@{>->}[r]^r&
\Omega^{-3}(k)\ar@{->>}[d]\ar@{>->}[r]^(0.55){r}& \cdots \\ 
L_r\ar@{>->}[r]&L_{r^2}\ar@{>->}[r]&L_{r^3}\ar@{>->}[r]&\cdots} 
\]
Thus, in $\StMod{kG}$ the module $L_{r^{j}}$ is the Koszul object $\kos k{r^{j}}$, for each $j\geq 1$. In the case where $r=\zeta_1$, for example the diagrams are as follows:
\begin{equation}\label{eq:V4} 
\vcenter{\xymatrix{\ k_{\phantom{|}}\ \ar@{=}[r]\ar@{>->}^r[d] & 
\ k_{\phantom{|}}\ \ar@{=}[r]\ar@{>->}[d]^{r^2}& 
\ k_{\phantom{|}}\ \ar@{=}[r]\ar@{>->}[d]^{r^3}& \cdots\\
\xy 
(-4,-4)*{\bullet}="A"; (0,0)*{\bullet}="B"; (4,-4)*{\bullet}="C"; 
"B";"A";**\dir{-}; "B";"C";**\dir{-}; 
\endxy
\ar@{->>}[d] \ar@{>->}[r]^(0.4){r}& 
\xy 
(-4,-4)*{\bullet}="A"; (0,0)*{\bullet}="B"; (4,-4)*{\bullet}="C"; (8,0)*{\bullet}="D";
(12,-4)*{\bullet}="E";
"B";"A";**\dir{-}; "B";"C";**\dir{-}; "D";"C";**\dir{-}; "D";"E";**\dir{-}; 
\endxy
\ar@{->>}[d]\ar@{>->}[r]^(0.44){r}&
\xy 
(-4,-4)*{\bullet}="A"; (0,0)*{\bullet}="B"; (4,-4)*{\bullet}="C"; (8,0)*{\bullet}="D";
(12,-4)*{\bullet}="E"; (16,0)*{\bullet}="F"; (20,-4)*{\bullet}="G";
"B";"A";**\dir{-}; "B";"C";**\dir{-}; "D";"C";**\dir{-}; "D";"E";**\dir{-}; 
"F";"E";**\dir{-}; "F";"G";**\dir{-}; 
\endxy
\ar@{->>}[d]\ar@{>->}[r]^(0.7){r}& \cdots \\ 
\xy 
(0,0)*{\bullet}="B"; (4,-4)*{\bullet}="C"; 
"B";"C";**\dir{-}; 
\endxy
\ar@{>->}[r]&
\xy 
(0,0)*{\bullet}="B"; (4,-4)*{\bullet}="C"; (8,0)*{\bullet}="D";
(12,-4)*{\bullet}="E";
"B";"C";**\dir{-}; "D";"C";**\dir{-}; "D";"E";**\dir{-}; 
\endxy
\ar@{>->}[r]&
\xy 
(0,0)*{\bullet}="B"; (4,-4)*{\bullet}="C"; (8,0)*{\bullet}="D";
(12,-4)*{\bullet}="E"; (16,0)*{\bullet}="F"; (20,-4)*{\bullet}="G";
"B";"C";**\dir{-}; "D";"C";**\dir{-}; "D";"E";**\dir{-}; 
"F";"E";**\dir{-}; "F";"G";**\dir{-}; 
\endxy
\ar@{>->}[r]&\cdots} }
\end{equation}
In this diagram, the $k$ in the top row injects as the leftmost vertex in the modules in the middle row. The case of a general element $r$ can be reduced to this one by applying automorphisms of the group algebra.

\begin{remark}
We can now describe the classification of the finite dimensional indecomposable $kG$-modules; for details see \cite[\S4.3]{Benson:1991a}. They come in three types:
\begin{enumerate}
\item The group algebra $kG$ itself.
\item For each $n\in \bbZ$, the module $\Omega^n(k)$.
\item For each one dimensional subspace of $H^1(G,k)$, choose an
element $r$; then for each integer $n\ge 1$ there is a module $L_{r^n}$.
Thus we get a family of modules indexed by $\bbP^1(k) \times \bbZ_{\ge 1}$.
\end{enumerate}
The modules in the third family for a particular choice of $r$, namely those appearing in the bottom row of diagram \eqref{eq:V4}, form a tube in the Auslander--Reiten quiver of $kG$. The maps going up
the tube are the maps in the bottom row of the diagram.

For larger elementary abelian groups the representation type is wild, and we cannot write down such a classification. Even for the Klein four group, the infinite dimensional modules cannot be classified.
\end{remark}

Next we compute the local cohomology functor $\gam_\fp\col\StMod kG\to\StMod kG$ for each homogeneous prime ideal $\fp$ of $H^*(G,k)$. Observe that $\gam_\fm=0$ for the unique maximal ideal $\fm=H^+(G,k)$.  Now the homogeneous non-maximal prime ideals of $H^*(G,k)$ are the zero ideal $\fn=(0)$, and the principal ideals $(r)$, one for each one dimensional subspace of $H^1(G,k)$. Observe that
\[
\gam_{(r)}=\gam_{\mcV(r)}\quad\text{and}\quad \gam_{\fn}=L_{\mcZ(\fn)},
\]
since each ideal $(r)$ is maximal among all non-maximal ideals, and since $\fn$ is the unique minimal ideal; see Exercise 8 for Thursday.

The finite dimensional indecomposables in
\[ 
\gam_{\mcV(r)}\StMod(kG) \subseteq \StMod(kG) 
\] 
are precisely the modules $L_{r^n}$, $n\ge 1$. For $r=\zeta_1$, the module $L_{\mcV(r)}(k)$ is the colimit of the modules in the middle row of diagram \eqref{eq:V4}, while the module $\gam_{\mcV(r)}(k)$ is the colimit of
the modules in the bottom row of the diagram. We can draw diagrams of
these infinite dimensional modules as follows:
\begin{gather*} 
L_{\mcV(r)}(k) = \vcenter{\xy
(-4,-4)*{\bullet}="A"; (0,0)*{\bullet}="B"; (4,-4)*{\bullet}="C"; (8,0)*{\bullet}="D";
(12,-4)*{\bullet}="E"; (16,0)*{\bullet}="F"; (20,-4)*{\bullet}="G"; (23,-1)*{}="H";
"B";"A";**\dir{-}; "B";"C";**\dir{-}; "D";"C";**\dir{-}; "D";"E";**\dir{-}; 
"F";"E";**\dir{-}; "F";"G";**\dir{-}; "G";"H";**\dir{-};
\endxy} \cdots \\ \\
\gam_{\mcV(r)}(k) = \vcenter{\xy
(0,0)*{\bullet}="B"; (4,-4)*{\bullet}="C"; (8,0)*{\bullet}="D";
(12,-4)*{\bullet}="E"; (16,0)*{\bullet}="F"; (20,-4)*{\bullet}="G"; (23,-1)*{}="H";
"B";"C";**\dir{-}; "D";"C";**\dir{-}; "D";"E";**\dir{-}; 
"F";"E";**\dir{-}; "F";"G";**\dir{-}; "G";"H";**\dir{-};
\endxy} \cdots 
\end{gather*}
The colimit of the vertical exact sequences in diagram \eqref{eq:V4}
is the exact sequence
\[ 
0 \to k \to L_{\mcV(r)}(k) \to \gam_{\mcV(r)}(k) \to 0 
\]
given by the inclusion of the left hand vertex in $L_{\mcV(r)}(k)$.

The remaining prime  to deal with is $\fn$, the zero ideal. The module $L_{\mcZ(\fn)}(k)$ can be described as
follows. Let $k(t)$ be the field of rational functions in one variable, regarded as an infinite dimensional vector space over $k$. Then $L_{\mcZ(\fn)}(k)$ is the module whose underlying vector space is a direct sum of two copies of $k(t)$, with $G$-action given by
\[ 
g_1 \mapsto \begin{pmatrix}I&0\\I&I\end{pmatrix}\qquad
g_2\mapsto \begin{pmatrix}I&0 \\ t.I & I \end{pmatrix} 
\]
where $I$ is the identity endomorphism of $k(t)$ and 
$t.I$ is the endomorphism of $k(t)$ given by multiplication by $t$.
There is a map from $k$ to $L_{\mcZ(\fn)}(k)$ sending the identity to
the vector $\bigl(\begin{smallmatrix}0\\1\end{smallmatrix}\bigr)$,
and the cokernel is $\gam_{\mcZ(\fn)}(k)$. This gives the  exact sequence
\[ 
0\to k \to L_{\mcZ(\fn)}(k) \to \gam_{\mcZ(\fn)}(k) \to 0. 
\]

Note that the modules $\gam_\mcV(k)$ are all periodic of period one in this
example, so that a short exact sequence
\[ 
0 \to k \to L_\mcV(k) \to \gam_\mcV(k) \to 0 
\]
gives rise to a triangle
\[ 
\gam_\mcV(k) \to k \to L_\mcV(k)\to 
\]
in the stable module category, which is the localisation triangle for $\mcV$.

We can now draw a diagrammatic representation of the  set of thick subcategories
of $\sfD^b(\mod(kG))$.
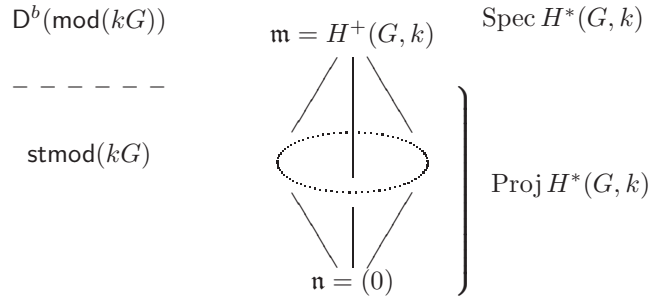
\medskip
\[ 
\xy (0,34)*{\sfD^b(\mod(kG))};(0,16)*{\stmod(kG)};
(-10,25)*{};(10,25)*{} **\dir{--};
(35,15)*\xycircle(10,4){.};
(27,11)*{};(33,1)*{} **\dir{-};(35,9)*{};(35,1)*{} **\dir{-};
(43,11)*{};(37,1)*{} **\dir{-};(27,19)*{};(33,29)*{} **\dir{-};
(35,13)*{};(35,29)*{} **\dir{-};(43,19)*{};(37,29)*{} **\dir{-};
(35,-1)*{\fn=(0)};(35,32)*{\fm=H^+(G,k)};(50,-2)*{};(50,25)*{} **\frm{)};
(64,12)*{\Proj H^*(G,k)};(63,34)*{\Spec H^*(G,k)};
 \endxy\medskip 
 \]
In this diagram, the dotted circle in the middle represents the set of closed points of the projective line over $k$, indexing the choices of $(r)$. The part below the top vertex represents $\stmod(kG)$ while the whole
diagram represents $\sfD^b(\mod(kG))$.

\section{Exercises}
\label{exer:Thursday}
\begin{enumerate}[\quad\rm(1)]

\item Let $\sfT$ be a compactly generated triangulated category. Given
  any class $\sfC$ of compact objects, prove that there exists a
  localisation functor $L\colon\sfT\to\sfT$ such that $\Ker L=\Loc
  (\sfC)$.\smallskip

  \noindent Hint: Use Brown representability to show that the
  inclusion $\Loc (\sfC)\to\sfT$ admits a right adjoint.

\item Let $A=\left[\begin{smallmatrix}k&k\\ 0&k\end{smallmatrix}\right]$ be the algebra of $2\times 2$ upper triangular matrices over a field $k$ and let $\sfT$ denote the derived category of all $A$-modules.  Up to isomorphism, there are precisely two indecomposable projective $A$-modules:
\[
P_1=
\begin{bmatrix} k&0\\    
0&0\end{bmatrix} \qquad\text{and}\qquad P_2= \begin{bmatrix}0&k\\    
0&k\end{bmatrix}
\]
satisfying $\Hom_A(P_1,P_2)\neq 0$ and $\Hom_A(P_2,P_1)= 0$.  For $i=1,2$ let $\bloc_i$ denote the localisation functor such that the $\bloc_i$-acyclic objects form the smallest localising subcategory containing $P_i$, viewed as a complex concentrated in degree zero. Show that $\bloc_1\bloc_2\neq\bloc_2\bloc_1$.

\item Let $\sfT$ be a an $R$-linear triangulated category. Each specialisation closed subset of $\Spec R$ gives rise to a recollement of $\sfT$ in the sense of Section~\ref{sec:Wednesday3}. Describe all 6 functors explicitly, and compare it with the prototype of a recollement given in \cite[\S1.4]{BBD:1983}.

\item Let $k$ be a field, $A=k[x]/x(x-1)$, and $\sfT$ its derived category $\sfD(A)$. Then $\sfT$ is $A$-linear, hence also $k$-linear, via restriction along the homomorphism $k\to A$. Prove that $\sfT$ has four localising subcategories. Hence $\sfT$ cannot be stratified by the $k$-action. It is however stratified by the $A$-action; this is a special case of Neeman's theorem, but can be verified directly.

\item Prove that $\Ext^{*}_{\bbZ}(\bbQ,\bbZ)=0$. Thus the `orthogonality' relation:
\[
\Supp_{\bbZ}\Hom^{*}_{\sfD(\bbZ)}(C,Y)=\Supp_{\bbZ}C\cap\Supp_{\bbZ}Y
\]
may fail when $C$ is not compact; i.e. when $H^{*}(C)$ is not finitely generated.

Another example to bear in mind is a complete local ring $A$, with maximal ideal $\fm$. With $E$ denoting the injective hull of $A/\fm$ one has
\[
\Supp_{A}\Hom^{*}_{\sfD(A)}(E,E)=\Supp_{A}\Hom_{A}(E,E) = \Supp_{A}A = \Spec A.
\]
On the other hand, $\Supp_{A}E=\{\fm\}$.

\item Prove that $\supp_{R}\sfT=\supp_{R}\sfT^{\sfc}$.

\item Let $\sfT$ be an $R$-linear triangulated category. Prove that there are equalities
\[
\{\fp\in\Spec R\mid X_{\fp}\ne 0\} =\bigcup_{C\in\sfT^{\sfc}}\Supp_{R}\Hom^{*}_{\sfT}(C,X)=\cl(\supp_{R}X)
\]
for each $X$ in $\sfT$. The set on the left is what we denoted $\Supp_{R}X$.

\item Let $\fp$ be a homogeneous prime ideal in $R$. Prove that $\gam_\fp\cong \gam_{\mcV(\fp)}$ when $\fp$ is maximal, with respect to inclusion, in $\supp_{R}\sfT$, and that $\gam_{\fp}\cong L_{\mcZ(\fp)}$ when $\fp$ is minimal in $\supp_{R}\sfT$.

\item For any object $X$ in $\sfT$ and homogeneous ideal $\fa$ in $R$, prove that
\[
\supp_{R}(\kos X{\fa})= \mcV(\fa)\cap \supp_{R}X\,;
\]
cf.~Proposition~\ref{prop:kosca}. Using this, prove that any closed subset of $\supp_{R}\sfT$ is realizable as the support of a compact object in $\sfT$.

\item Let $M$ and $N$ be $kG$-modules. Recall (or prove) that there is an isomorphism
\[
\Ext^{*}_{kG}(M,N)\cong H^{*}(G,\Hom_{k}(M,N))\,,
\]
compatible with the $H^{*}(G,k)$-actions. Assume now that $M$ and $N$ are finitely generated and prove that there is an isomorphism of $kG$-modules:
\[
\Hom_{k}(\Hom_{k}(M,N),k) \cong \Hom_{k}(N,M)
\]
Conclude that the support of $\Ext^{*}_{kG}(M,N)$ and $\Ext^{*}_{kG}(N,M)$, as modules over $H^{*}(G,k)$, coincide. Compare with Theorem~\ref{thm:orthogonality}.
\end{enumerate}

\chapter{Friday}
\thispagestyle{empty}

In this last chapter we put together the various ideas we have been developing in the week's lectures. The goal, as has been stated often enough, is a classification of the localising subcategories of the stable module category of a finite group, over a field of characteristic $p$. The strategy of the proof was described in Section~\ref{ssec:strategy}, and we begin this chapter at the last step, which is also where the whole story begins, namely, Neeman's classification of the localising subcategories of the derived category of a commutative noetherian ring. Using a (version of) this result, and a variation of the Bernstein-Gelfand-Gelfand correspondence, we explain how to tackle the case of the homotopy category of complexes of injective modules over an elementary abelian two group. This is the content of Section~\ref{sec:Friday2}. Finally, in the last section, we use Quillen's results to describe how to pass from arbitrary groups to elementary abelian ones. If the dust settles down, you should be able to see a fairly complete proof of our main results for the case $p=2$.

\section{Localising subcategories of $\sfD(A)$}
\label{sec:Friday1}
In this lecture $A$ is a commutative noetherian ring and $\sfD$ the derived category of the category of $A$-modules.
Recall that $\sfD$ is a triangulated category admitting set-indexed coproducts. Its compact objects are:
\[
\sfD^{\sfc} = \Thick(A)\,,
\]
the set of perfect complexes of $A$-modules; see Theorem~\ref{thm:compacts-DA}. Moreover, one has
\[
\Loc_{A}(A)  = \sfD\,,
\]
so $\sfD$ is a compactly generated triangulated category, with $A$ a compact generator. The category $\sfD$ is also $A$-linear, with structure homomorphisms 
\[
A\to \Hom^*_{\sfD}(M,M)
\]
given by left multiplication. As $A$ is a compact generator and $\Hom_{\sfD}^{*}(A,A)\cong A$, it follows that the $A$-linear category $\sfD$ is noetherian. Note that $\supp_{A}\sfD=\Spec A$.

This lecture is devoted to proving the following result:

\begin{theorem}
\label{thm:stratify-da}
The triangulated category $\sfD$ is stratified by $A$.
\end{theorem}

A proof is given later in this lecture. For now, we record some direct corollaries. The one below is by Theorem~\ref{thm:strat-class}, and is one of the main results in~\cite{Neeman:1992a}.

\begin{corollary}
\label{cor:neeman}
There is bijection between localising subcategories of $\sfD$ and subsets of $\Spec A$.\qed
\end{corollary}

The next result is by Theorem~\ref{thm:strat-thick}, since $\sfD$ is a noetherian category.

\begin{corollary}
\label{cor:hopkins}
There is bijection between the thick subcategories of $\Thick(A)$ and specialisation closed subsets of $\Spec A$. \qed
\end{corollary}

\subsection{Local cohomology and support}
Next we relate the functors $\gam_{\mcV}$ and $L_{\mcV}$ to familiar functors from commutative algebra. What follows is a summary of results from \cite[Section~9]{Benson/Iyengar/Krause:2008a}. 

To begin with,  note that or each $M$ in $\sfD$ there is an isomorphism of $A$-modules
\[
\Hom_\sfD^{*}(A,M)\cong H^{*}(M)
\]

For any $\fp\in\Spec A$, there are natural isomorphisms 
\[
L_{\mcZ(\fp)}(M) \cong M_{\fp}\,,
\]
where $M_{\fp}$ denotes the usual localisation of the complex $M$ at $\fp$. In other words, the functor $M\mapsto M_{\fp}$ is localisation  with respect to $\mcZ(\fp)$.

For each ideal $\fa$ in $A$, the functor $\gam_{\mcV(\fa)}$ is the derived functor of the \emph{$\fa$-torsion functor}\index{atorsionfunctor@$\fa$-torsion functor}, that associates to each $A$-module $M$, the submodule
\[
\{m\in M\mid \fa^{n} m =0 \text{ for some $n\geq 0$}\}.
\]
It follows that, for any $M$ in $\sfD$, the cohomology of $\gam_{\mcV(\fa)}(M)$ is \emph{local cohomology}\index{local cohomology} of $M$ with support in $\fa$, in the sense of Grothendieck:
\[
H^{*}(\gam_{\mcV(\fa)}M)\cong H^{*}_{\fa}(M) \quad\text{for each $M\in\sfD$.}
\]
In particular, for each $\fp\in\Spec A$ there is an isomorphism $H^{*}(\gam_{\fp}M)\cong H^{*}_{\fp A_{\fp}}(M_{\fp})$, so that the support of $M$ is computed as:
\[
\supp_{A}M = \{\fp\in\Spec A\mid H^{*}_{\fp A_{\fp}}(M_{\fp})\ne 0\}\,.
\]
Compare this with other methods for computing support given in Lemma~\ref{le:supp-tests}.

\subsection{Proof of Theorem~\ref{thm:stratify-da}}
The derived tensor product $-\lotimes_{A}-$ on $\sfD$ endows it with a structure of a tensor triangulated category, with unit $A$. Since $A$ is a compact generator for $\sfD$, it follows from Theorem~\ref{thm:localglobal-tt} that the local global principle holds for $\sfD$. It thus remains to verify stratification condition (S2) from Section~\ref{sec:Thursday2}.

Fix $\fp\in\Spec A$ and set $k(\fp)=A_{\fp}/\fp A_{\fp}$, the residue field of $A$ at $\fp$. Note that $k(\fp)$ is $\fp$-local and $\fp$-torsion and hence is in $\gam_{\fp}\sfD$, by Corollary~\ref{cor:plocptor}. In particular, the latter category is non-empty, and to prove that $\gam_{\fp}\sfD$ is minimal, it thus suffices to prove:
\[
\Loc_{\sfD}(M)=\Loc_{\sfD}(k(\fp)) \quad\text{for any $M\in\gam_{\fp}\sfD$ with $H^{*}(M)\ne 0$.}
\] 
We reduce to the case where $A$ is a local ring, with maximal ideal $\fp$, as follows: The homomorphism of rings $A\to A_{\fp}$ gives a restriction functor $\iota\col \sfD(A_{\fp})\to \sfD(A)$. Note that there is an isomorphism $M\cong \iota(M_{\fp})$ in $\sfD$, because $M$ is $\fp$-local. Since $\iota$ is compatible with coproducts, and maps $k(\fp)$ (viewed as an $A_{\fp}$-module) to $k(\fp)$, it suffices to prove
\[
\Loc_{\sfD(A_{\fp})}(M_{\fp})=\Loc_{\sfD(A_{\fp})}(k(\fp))\,.
\]
Observe that $M_{\fp}$ and $k(\fp)$ are local and torsion with respect to the ideal $\fp A_{\fp}\subset A_{\fp}$. Hence replacing $A$ and $M$ by $A_{\fp}$ and $M_{\fp}$, we may assume $A$ is local, with maximal ideal $\fp$, and set $k=A/\fp$. Since $\gam_{\fp} = \gam_{\mcV(\fp)}$, by Exercise~8 from Thursday's lecture, we then need to prove that
\[
\Loc_{\sfD}(M)=\Loc_{\sfD}(k) \quad\text{for each $M\in\gam_{\mcV(\fp)}\sfD$.}
\]

First we tackle the case $M=\gam_{\mcV(\fp)}A$. We claim $k$ is in $\Loc_{\sfD}(\gam_{\mcV(\fp)}A)$. Indeed, since localising subcategories of $\sfD$ are tensor closed, so one has:
\[
(\gam_{\mcV(\fp)}A\lotimes_{A}k)\, \in\, \Loc_{\sfD}(\gam_{\mcV(\fp)}A)\,.
\]
Now note that there are isomorphisms $k\cong \gam_{\mcV(\fp)}k\cong (\gam_{\mcV(\fp)}A\lotimes_{A}k)$.
Therefore
\[
\Loc_{\sfD}(k)\subseteq \Loc_{\sfD}(\gam_{\mcV(\fp)}A)\,.
\]
We verify the reverse inclusion by verifying the follows claims:
\[
\gam_{\mcV(\fp)}A\,\in\, \Loc_{\sfD}(\kos A\fp)\quad\text{and}\quad
\kos A\fp\,\in\, \Thick_{\sfD}(k)\,.
\]
The first inclusion is from Proposition~\ref{prop:locandkos}. As to the second: Since $A$ is finitely generated as an $A$-module and $\kos A\fp$ is in the thick subcategory generated by $A$, one obtains that the $A$-module $H^{*}(\kos A\fp)$ is finitely generated. It is also $\fp$-torsion, by Proposition~\ref{prop:kosprop}, and hence of finite length. Since $A$ is local, with residue field $k$, it follows that $\kos A\fp$ is in the thick subcategory generated by $k$; see Exercise~6.  

To sum up: $\Loc_{\sfD}(\gam_{\mcV(\fp)}A)=\Loc_{\sfD}(k)$. Since localising subcategories of $\sfD$ are tensor closed, and the exact functor $-\lotimes_{A}M$ preserves coproducts, the preceding equality yields the second equality below:
\[
\Loc_{\sfD}(M)= \Loc_{\sfD}(\gam_{\mcV(\fp)}A\lotimes_{A}M)=\Loc_{\sfD}(k\lotimes_{A}M)
\]
The first equality holds because $M$ is in $\gam_{\mcV(\fp)}\sfD$; see Section~\ref{ssec:ttcats}. Since $H^{*}(M)$ is non-zero, one gets in particular that $H^{*}(k\lotimes_{A}M)$ is non-zero as well. Since $k$ is a field, in $\sfD$ there is an isomorphism
\[
k\lotimes_{A}M \cong H^{*}(k\lotimes_{A}M)\,.
\]
Since $H^{*}(k\lotimes_{A}M)$  is a non-zero $k$-vector space, it can build $k$ and vice-versa, in $\sfD(k)$ and hence also in $\sfD$.

This completes the proof of the theorem.

\section{Elementary abelian 2-groups}
\label{sec:Friday2}
\index{elementary abelian 2-group}
Fix a prime number $p$, and let $E=\langle g_{1},\dots,g_{r}\rangle$ be an elementary abelian $p$-group of rank $r$. We consider $\KInj{kE}$, the homotopy category of complexes of injective (which in this is the same case as free) $kE$-modules. Recall that the tensor product over $k$, with diagonal $kE$-action, induces on $\KInj{kE}$ a structure of a tensor triangulated category, and that there is a canonical action of $H^{*}(E,k)$ on it.

As explained in Section~\ref{ssec:strategy}, one of the main steps in \cite{Benson/Iyengar/Krause:2008a} is a proof of the statement that $\KInj{kE}$ is stratified by $H^{*}(E,k)$. In this lecture, we prove this result in the special case where $p=2$, namely:

\begin{theorem}
\label{thm:2groups}
Let $k$ be a field of characteristic $2$ and $E$ an elementary abelian $2$-group. Then $\KInj{kE}$ is stratified by the canonical action of $H^{*}(E,k)$.
\end{theorem}

The proof of this theorem is given in Section~\ref{ssec:proof2groups}. It uses a variation, and enhancement, of the classical Bernstein-Gelfand-Gelfand correspondence~\cite{Bernstein/Gelfand/Gelfand:1978}; see Remark~\ref{rem:bgg} for a detailed comparison. The statement, and the proof, of this latter result requires a foray into some homological algebra of differential graded modules over differential graded algebras.

\subsection{Differential graded algebras}
\label{Differential graded algebras}
\index{differential graded!algebra}
Let $A$ be  differential graded algebra. This means that $A$ is a graded algebra with a differential satisfying the Leibniz rule:
\[
d(ab) = d(a)\,b+(-1)^{|a|}a\,d(b)
\]
for all homogeneous elements $a,b$ in $A$.  In the same vein, a  differential graded  module $M$ over such a  differential graded  algebra $A$ is a  graded $A$-module with a differential such that the multiplication satisfies the appropriate Leibniz rule. 

Any graded algebra $A$ can be viewed as a  differential graded  algebra with zero differential; most of the  differential graded  algebras we encounter in the discussion below will be of this nature. In this case, any graded $A$-module is a  differential graded  $A$-module with zero differential. However, the category of  differential graded  $A$-modules is typically much larger, as should be clear soon. A ring is a  differential graded  algebra concentrated in degree $0$, and then a  differential graded  module is nothing more, and nothing less, than a complex of modules.

Let $A$ be a  differential graded  algebra. A morphism $f\col M\to N$ of  differential graded  $A$-modules is a homomorphism of graded $A$-modules that is at the same time a morphism of complexes. Differential graded $A$-modules and their morphisms form an abelian category. The usual notion of homotopy (for complexes of modules over a ring) carries over to differential graded modules, and one can form the associated homotopy category. Inverting the quasi-isomorphisms there gives the \emph{derived category}\index{derived category} $\sfD(A)$ of  differential graded  $A$-modules. It is triangulated and the  differential graded  module $A$ is a compact object and a generator; see \cite{Keller:1994a}.

\subsection{A BGG correspondence}
\label{ssec:BGG correspondence}
\index{BGG correspondence}
Let $k$ be a field of characteristic $2$. Let $E$ be an elementary abelian $2$-group and $kE$ its group algebra.
Thus $E\cong (\bbZ/2)^{r}$ and, setting $z_{i}=g_{i}-1$, there is an isomorphism of $k$-algebras
\[
kE\cong k[z_{1},\dots,z_{r}]/(z_{1}^2,\dots,z_{r}^{2})\,.
\]
Let $S=k[x_{1},\dots,x_{r}]$ be a graded polynomial algebra over $k$ where $|x_{i}|=1$ for each $i$. We view it as a  differential graded algebra with zero differential. Our goal is to construct an explicit equivalence between the triangulated categories $\KInj{kE}$ and $\sfD(S)$, the derived category of  differential graded  modules over $S$.
To this end, we mimic the approach in \cite[Section 7]{Avramov/Buchweitz/Iyengar/Miller:2010a}.

The graded $k$-algebra $kE \otimes_{k} S$, with component in degree $i$ being $kE\otimes_{k}S^{i}$ and product  defined by 
\[
(a\otimes s)\cdot (b\otimes t) = ab\otimes st\,,
\]
is commutative. Again we view it as a  differential graded  algebra with zero differential, and consider in it the element
of degree one:
\[ 
\delta = \sum_{i=1}^{r}z_{i}\otimes_{k} x_{i}\,.
\]
It is easy to verify that $\delta^{2}=0$; keep in mind that the characteristic of $k$ is $2$. In what follows $J$ denotes the  differential graded  module over $kE\otimes_{k}S$ with underlying graded module and differential given by
\[
J = kE \otimes_{k} S\quad\text{and}\quad d(e)=\delta e.
\]
Observe that since $J$ is a  differential graded  module over $kE\otimes_{k}S$, for each  differential graded  module $M$ over $kE$ there is an induced structure of a  differential graded  $S$-module on $\Hom_{kE}(J,M)$. 
Moreover, the map 
\[
\zeta\col S\to \Hom_{kE}(J,J)
\]
induced by right multiplication is a morphism of differential graded algebras.

\begin{theorem}
\label{th:bgg}
The map $\zeta$ is a quasi-isomorphism and the functor
\[
\Hom_{kE}(J,-)\col \KInj{kE} \to \sfD(S) 
\]
is an equivalence of triangulated categories.
\end{theorem}

\begin{proof}
The main ingredients in the proof are summarised in the following claims:

\smallskip

\noindent\emph{Claim}: As a complex of $kE$-modules, $J$ consists of injectives and $J^{i}=0$ for $i<0$.

\smallskip

Indeed, the degree $i$ component of $J$ is $kE\otimes_{k}S^{i}$, which is isomorphic to a direct sum of $\binom {r+i-1}i$ copies of $kE$. The claim follows, since $kE$ is self-injective.

\medskip

Set $w=z_{1}\cdots z_{r}$; this is a non-zero element in the socle of $kE$, and hence generates it. It is easy to verify that the map of graded $k$-vector spaces
\[
\eta\col k\to J\quad\text{defined by $\eta(1) = w \otimes 1$}
\]
is a morphism of  differential graded  $kE$-modules.

\smallskip

\noindent\emph{Claim}: $\eta\col k\to J$ is an injective resolution of $k$, as a  $kE$-module.

\smallskip

Given the previous claim, and that $\eta$ is $kE$-linear, it remains to prove that $\eta$ is a quasi-isomorphism. We have to prove that $H^{*}(\eta)$ is an isomorphism, and in verifying this we can, and will, ignore the $kE$-module structures on $k$ and $J$.

When $r=1$, complex $J$ can be identified with the complex 
\[
0\to k[z]/(z^{2}) \xra{z} k[z]/(z^{2}) \xra{z} \cdots 
\]
of $k[z]/(z^{2})$-modules. It is easy to verify that the only non-zero cohomology in this complex is in degree
$0$, where it is $kw$. Thus, $\eta$ is a quasi-isomorphism. 

For general $r$, writing $J(i)$ for the complex $k[z_{i}]/(z_{i}^{2})\otimes_{k}k[x_{i}]$ one has a natural isomorphism
\[
J\cong J(1)\otimes_{k}\cdots \otimes_{k}J(r)
\]
of complex of $k$-vector spaces. With this, and setting $\eta(i)$ to be the map $k\to J(i)$ defined by $1\mapsto z_{i}\otimes 1$, one gets an identification
\[
\eta = \eta(1)\otimes_{k}\cdots \otimes_{k}\eta(r)\col k  \to J\,.
\]
Since each $\eta(i)$ is a quasi-isomorphism, it follows from the K\"unneth formula that $H^{*}(\eta)$ is an isomorphism as well.

This completes the proof of the claim.

\medskip

Now we are ready to prove that the map $\zeta\col S\to \Hom_{kE}(J,J)$ induced by right multiplication is a quasi-isomorphism. In verifying that $\zeta$ is a quasi-isomorphism we can now ignore the $S$-module structures.  The composite morphism of $k$-vector spaces
\[
S\xra{\quad\zeta\quad}\Hom_{kE}(J,J)\xra{\,\Hom_{kE}(\eta,J)\,} \Hom_{kE}(k,J)
\]
sends an element $s\in S$ to the map $1\mapsto w\otimes s$, where $w=z_{1}\dots z_{r}$.  It is easy to see that the composite is  an isomorphism of complexes. It remains to note that $\Hom_{kE}(\eta,J)$ is a quasi-isomorphism:  $\eta$ is a quasi-isomorphism of  differential graded  modules over $kE$ and $\Hom_{kE}(-,J)$ preserves quasi-isomorphisms, because $J$ is a complex of injective $kE$-modules with $J^{i}=0$ for $i<0$.

Armed with these facts about $J$, the proof can be wrapped up as follows: It is easy to verify that the functor $\Hom_{kE}(J,-)$ from $\KInj{kE}$ to $\sfD(S)$ is exact. Since $J$ is an injective resolution of $k$, it is compact and generates the triangulated category $\KInj{kE}$; see Theorem~\ref{thm:kinj-generation}. It follows from this that the functor $\Hom_{kE}(J,-)$ preserves coproducts, since a quasi-isomorphism between  differential graded  $S$-modules is an isomorphism in $\sfD(S)$. As $S$ is quasi-isomorphic to $\Hom_{kE}(J,J)$ and it is a compact generator for $\sfD(S)$, it remains to apply Exercise~23 in Chapter~\ref{ch:Monday} to conclude that $\Hom_{kE}(J,-)$ is an equivalence.
\end{proof}

\begin{remark}
\label{rem:bgg}
Let $k$ be a field, $\Lambda$ an exterior algebra on generators of degree $1$, and $S$ a polynomial algebra on generators of degree $1$. Bernstein, Gelfand, and Gelfand~\cite{Bernstein/Gelfand/Gelfand:1978} proved that there is an equivalence of triangulated categories
\[
\sfD^{\sfb}(\grmod \Lambda)\simeq \sfD^{\sfb}(\grmod S) 
\]
The functor inducing the equivalence is similar to the one in Theorem~\ref{th:bgg}.

Over a field of characteristic $2$, the group algebra of an elementary abelian $2$ group is an exterior algebra.
However, the result above cannot be applied directly to our context. One point being that it deals with \emph{graded modules} over an exterior algebra of degree one, but this is not crucial. The main issue is that it deals with \emph{bounded} derived categories, while for our purposes we need a statement at the level of the full homotopy category of complexes of injective modules; see Section~\ref{ssec:recollement} for the discrepancy between the two.
\end{remark}

\subsection{Proof of Theorem~\ref{thm:2groups}}
\label{ssec:proof2groups}

We retain the notation from  Section~\ref{ssec:BGG correspondence}. 

The basic idea of the proof is clear enough: $\KInj{kE}$ is equivalent to $\sfD(S)$, by Theorem~\ref{th:bgg}, so it suffices to prove that the latter is stratified, and for that one invokes Theorem~\ref{thm:stratify-da}. Two issues need clarification. 

One point is that Theorem~\ref{thm:stratify-da} concerns commutative rings and not differential graded algebras. However, the argument given for \emph{op.~cit.} carries over with minor modifications to $\sfD(S)$; see \cite[Theorem~5.2]{Benson/Iyengar/Krause:bik3}, and also \cite[Theorem~8.1]{Benson/Iyengar/Krause:bik2}, for a statement that applies to any formal differential graded algebra whose homology is commutative and noetherian.

The second point is this: By the discussion above,  $\KInj{kE}$ is stratified by the $S$-action  induced by the equivalence of categories in Theorem~\ref{th:bgg}. The latter result also gives that $H^{*}(E,k)\cong S$, as algebras; one then needs to check  that the induced action of $H^{*}(E,k)$ on $\KInj{kE}$ is the diagonal action. In fact, much less is required, namely, that the local cohomology functors, $\gam_{\fp}$, defined by the two actions coincide, and this is easier to verify; see the proof of \cite[Theorem~6.4]{Benson/Iyengar/Krause:bik3}.

For a systematic treatment of such ``change-of-categories'' results, which may serve to clarify the issue, see \cite[Section~7]{Benson/Iyengar/Krause:bik4}.

\section{Stratification for arbitrary finite groups}
\label{sec:Friday3}

The goal for this lecture is as follows. Let $G$ be a finite group and $k$ a field of characteristic $p$. We shall assume that $\KInj{kE}$ is stratified by $H^*(E,k)$ for all elementary abelian subgroups $E$ of $G$  (for $p=2$ this was proved in the previous lecture) and we shall prove that $\KInj{kG}$ is stratified by $H^*(G,k)$.

Our main tool is Quillen's Stratification Theorem; so we
begin by examining Quillen's theorem in some detail.

\subsection{Quillen Stratification}
\index{Quillen!stratification theorem}

Quillen \cite{Quillen:1971a,Quillen:1971b} 
(1971) described the \emph{maximal} ideal spectrum\index{maximal ideal spectrum}
of $H^*(G,k)$ in terms of elementary abelian subgroups.
We're really interested in \emph{prime} ideals, but let's first describe
Quillen's original theorem. Note that in a finitely generated 
graded commutative $k$-algebra such as $H^*(G,k)$ every prime ideal
is the intersection of the maximal ideals containing it, by
a version of Hilbert's Nullstellensatz.\index{Hilbert's Nullstellensatz}

Let $G$ be a finite group, and $k$ a field of characteristic $p$. If $H$ is a subgroup of $G$, there's a \emph{restriction map}\index{restriction!map} $H^*(G,k)\to H^*(H,k)$ which is a ring homomorphism. Writing $V_G=\max\,H^*(G,k)$, this induces a map of varieties $V_H\to V_G$.

\begin{theorem}
An element $u\in H^*(G,k)$ is nilpotent if and only if $\res_{G,E}(u)$ is
nilpotent for every elementary abelian $p$-subgroup $E\le G$. \qed
\end{theorem}

It thus makes sense to look at the product of the restriction maps
\[ 
H^*(G,k) \to \prod_E H^*(E,k). 
\]
The theorem implies that the kernel of this map is nilpotent.  What is the image?\medskip

If an element $(u_E)$ is in the image then
\begin{enumerate}
\item for each conjugation $c_g\colon E' \to E$ in $G$, 
$c_g^*(u_E)=u_{E'}$ and
\item for each inclusion $i\colon E' \to E$ in $G$, $i^*(u_E)=u_{E'}$.
\end{enumerate}

Conversely, if $(u_E)$ satisfies these conditions, Quillen showed
that for some $t\ge 0$ the element $(u_E^{p^t})$ is in the image.

\begin{definition}
We define $\varprojlim_{E} H^*(E,k)$ to be the elements $(u_E)$ of the direct product $\prod_E H^*(E,k)$ satisfying
conditions (1) and (2) above.
\end{definition}

\begin{definition}
A homomorphism $\phi\colon R\to S$ of graded commutative $k$-algebras 
is an \emph{$F$-isomorphism}\index{F@$F$-isomorphism} or 
\emph{inseparable isogeny}\index{inseparable isogeny} if
\begin{enumerate}
\item the kernel of $\phi$ consists of nilpotent elements, and
\item for each $s\in S$ there exists $t\ge 0$ such that $s^{p^t}\in\Im \phi$.
\end{enumerate}
\end{definition}

We can now rephrase Quillen's theorem as follows:

\begin{theorem}
\pushQED{\qed}
The restriction maps induce an $F$-isomorphism
\[ 
H^*(G,k)\to\varprojlim_{E} H^*(E,k)\,.\qedhere
\]
\end{theorem}

Now if $\phi\colon R\to S$ is an $F$-isomorphism of finitely generated 
graded commutative $k$-algebras then $\phi^*\colon \max S \to \max R$
is a \emph{bijection}.

For the moment, let's assume that $k$ is algebraically closed. Recalling that $V_G=\max H^*(G,k)$, this says that $\varinjlim_{E}V_E\to V_G$, as a map of varieties, is \emph{bijective} (but not necessarily invertible!).

Here, $\varinjlim_{E}V_E$ is the quotient of the disjoint union of the $V_E$ by the equivalence relation induced
by the conjugations and inclusions.

\subsection{A more concrete view}

Let us continue for a while with the assumption that $k$ is algebraically closed, a restriction which we shall later lift. Let us describe $V_G$ more explicitly in terms of the elementary abelian $p$-subgroups.  If $E$ is an elementary abelian $p$-group of rank $r$ then $H^*(E,k)$ modulo nilpotents is a polynomial ring in $r$
variables, and so 
\[ 
V_E=\max\,H^*(E,k)\cong\bbA^r(k), 
\] 
affine space of dimension $r$.

If $E'\le E$ then $\res_{E,E'}^*$ identifies $V_{E'}$ as a linear
subspace of $V_E$ determined by linear equations with coefficients in
the ground field $\bbF_p$. Furthermore, each such linear subspace
is the image of $\res_{E,E'}^*$ for a suitable $E'\le E$.

\begin{definition}
We set $V_E^+=V_E\setminus\bigcup_{E'<E}V_{E'}$, so that
$V_E^+$ is obtained from $V_E$ by removing all $(p^r-1)/(p-1)$ of
the codimension one subspaces 
in $V_E$ defined over $\bbF_p$. The set $V_E^+$ is a dense open
subset of $V_E$.

We consider the following subvarieties of $V_G$:
\[
V_{G,E}=\res_{G,E}^*(V_E)\qquad\text{and}\qquad V_{G,E}^+=\res_{G,E}^*(V_E^+)\,.
\]
Thus $V_{G,E}$ is a closed subvariety and $V_{G,E}^+$ is a locally closed subvariety.
\end{definition}

\begin{theorem}
\label{th:QST-max}
The following statements hold
\begin{enumerate}
\item
$V_G$ is the disjoint union of locally closed subsets $V_{G,E}^+$, one
for each conjugacy class of elementary abelian subgroups $E$ of $G$.
\item
$V_{G,E}^+$ is ($F$-isomorphic to) the quotient of the locally
closed variety $V_E^+$ by the free action
of $N_G(E)/C_G(E)$.
\item The layers $V_{G,E}^+$ of $V_G$ are glued together via the
inclusions $E'<E$ in $G$.\qed
\end{enumerate}
\end{theorem}

\subsection{Prime ideals}

Now we discuss prime ideals in $H^*(G,k)$, and we drop the assumption that $k$ is algebraically
closed. The discussion below is based on \cite[Section~9]{Benson/Iyengar/Krause:bik3}.

Write $\mcV_G$ for the (homogeneous) prime spectrum of $H^*(G,k)$. By Quillen's theorem, for each $\fp \in\mcV_G$ there exists an elementary abelian $p$-subgroup $E\le G$ such that $\fp$ is in the image of $\res_{G,E}^*$.
We say that $\fp$ \emph{originates}\index{originates} in such an $E$ if there does not exist a proper
subgroup $E'$ of $E$ such that $\fp$ is in the image of $\res_{G,E'}^*$.

\begin{theorem}
\label{th:origin}
For each $\fp \in\mcV_G$, the pairs $(E,\fq)$ where $\fq\in\mcV_E$,  $\fp=\res_{G,E}^*(\fq)$ and
such that $\fp$ originates in $E$ are all $G$-conjugate. \qed
\end{theorem}

\begin{warning} 
In contrast with part (2) of Theorem \ref{th:QST-max},
if we fix $\fp$ and $E$ such that $\fp$ originates in $E$, then
$N_G(E)/C_G(E)$ acts transitively but not necessarily 
freely on the set of primes
$\fq \in\mcV_E$ such that $\res_{G,E}^*(\fq)=\fp$.

Indeed, let $E$ be the normal four group of the alternating group $A_4$, and let $k$ be an algebraically closed field of characteristic two. Let $\omega$ be a primitive cube root of unity in $k$. Write $H^{*}(E,k)=k[x,y]$ where $x$ and $y$ are defined over $\bbF_2$. Then the prime ideal $(x+\omega y)$ is invariant under the action of $G$ but there is no inhomogeneous maximal ideal containing it that is fixed, because each point in the line is multiplied by $\omega$ or $\omega^2$ when acted on by an element of $A_{4}$ not in $E$.

\end{warning}

For any subgroup $H\le G$, formal properties of restriction and induction
prove the following:

\begin{lemma}
Let $\fp \in\mcV_G$ and set $\mcU=(\res_{G,H}^*)^{-1}\{\fp \}$.
\begin{enumerate}
\item For any $X$ in $\KInj{kG}$, $(\gam_\fp X){\downarrow_H}\cong
\bigoplus_{\fq \in\mcU}\gam_\fq(X{\downarrow_H})$.
\item For any $Y$ in $\KInj{kH}$, $\gam_\fp (Y{\uparrow^G})\cong
\bigoplus_{\fq \in\mcU}(\gam_\fq Y){\uparrow^G}$. \qed
\end{enumerate}
\end{lemma}

Using this, we get:

\begin{theorem}
If $X\in\KInj{kG}$ and $Y\in\KInj{kH}$ then
\begin{enumerate}
\item $\mcV_G(X{\downarrow_H\uparrow^G})\subseteq\mcV_G(X)$, and
\item $\mcV_G(Y{\uparrow^G})=\res_{G,H}^*\mcV_H(Y)$. \qed
\end{enumerate}
\end{theorem}

The following version of the subgroup theorem is for elementary
abelian groups. The analogue for arbitrary finite groups also holds,
and its proof makes use of this more limited version.

\begin{theorem}
\label{th:subgroup}
Let $E'\le E$ be elementary abelian $p$-groups, $X \in \KInj{kE}$. Then
\[ 
\mcV_{E'}(X{\downarrow_{E'}})=(\res_{E,E'}^*)^{-1}\mcV_E(X)\,. 
\]
\end{theorem}

\begin{proof}
Let $\fq \in\mcV_{E'}$, $\fp=\res_{E,E'}^*(\fq)$. By the previous theorem
we have
\[ 
\mcV_E(\gam_\fq k{\uparrow^E})=\{\fp \}=\mcV_E(\gam_\fp k),. 
\] 
By the classification of localising subcategories for $kE$ this implies
that
\[ 
\Loc(\gam_\fq k{\uparrow^E})=\Loc(\gam_\fp k)\,. 
\]
Thus 
\begin{align*} 
\gam_\fq(X{\downarrow_{E'}})\ne 0 
&\Iff\gam_\fq k\otimes X{\downarrow_{E'}}\ne 0 \\
&\Iff\gam_\fq k{\uparrow^E}\otimes X \ne 0 \\
&\Iff\gam_\fp k\otimes X\ne 0 \\
&\Iff \gam_\fp X \ne 0. 
\qedhere
\end{align*}
\end{proof}

\subsection{Chouinard's Theorem for $\KInj{kG}$}

The next ingredient in the proof of the classification theorem is
a version of Chouinard's Theorem \ref{th:Chouinard} 
for the category $\KInj{kG}$.

\begin{theorem}\label{th:Chouinard-Kinj}
An object $X$ in $\KInj{kG}$ is zero if and only if $X{\downarrow_E}$
is zero for every elementary abelian $p$-subgroup $E\le G$.
\end{theorem}

\begin{proof}
One direction is obvious, so assume $X\ne 0$. Look at the triangle 
\[
pk\otimes X \to X \to tk\otimes X\to\,.
\]
Either $tk \otimes X\ne 0$ or $pk \otimes X \ne 0$.  In the first case we're in $\StMod(kG)$ and we can use Chouinard's theorem for $\StMod(kG)$. In the second case $X$ is not acyclic so its restriction to any subgroup is non-zero. The trivial subgroup is elementary abelian, so we are done.
\end{proof}

\subsection{The main theorem}

\begin{theorem}
\label{thm:main-theorem}
As a tensor triangulated category, $\KInj{kG}$ is stratified by the canonical action of $H^*(G,k)$.
\end{theorem}

\begin{proof}
We must prove that if $\fp \in\mcV_G$ then $\gam_\fp \KInj{kG}$ is minimal among tensor ideal localising subcategories of $\KInj{kG}$. 

Let $0\ne X \in \gam_\fp \KInj{kG}$. By Theorem \ref{th:Chouinard-Kinj}, there exists $E_0$ elementary abelian with 
$X{\downarrow_{E_0}}\ne 0$. Choose $\fq_0\in\mcV_{E_0}(X{\downarrow_{E_0}})$. 
We have
\[
\res_{G,E_0}^*(\fq_0)\in\mcV_G(X{\downarrow_{E_0}\uparrow^G}) \subseteq\mcV_G(X)=\{\fp \}, 
\] 
so $\res_{G,E_0}^*(\fq_0)=\fp$. 

So there exists $(E,\fq)$ with $E\le E_0$, $\fq \in\mcV_E$, $\res_{E_0,E}^*\fq=\fq_0$, and $\fp$ originates in $\fq$. By the Theorem \ref{th:subgroup} we have $\fq \in\mcV_E(X{\downarrow_E})$, i.e., $\gam_\fq X{\downarrow_E} \ne 0$. 

By Theorem \ref{th:origin}, all $(E,\fq)$ where $\fp$ originate in $E$ are
conjugate. It follows that if $\gam_\fq X{\downarrow_E}\ne 0$ for
one of these then the same holds for all of them.
So if we choose one, then every $0\ne X\in\gam_\fp \KInj{kG}$
has $\gam_\fq X{\downarrow_E}\ne 0$, and $X{\downarrow_E}$ is a
direct sum of $N_G(E)$-conjugates of this. 

Let $0\ne Y\in\gam_\fp \KInj{kG}$ and set $Z=ik_E{\uparrow^G}$,
namely the injective resolution of the permutation module on the
cosets of $E$. Then we have
\begin{align*} 
\Hom^*_{kG}(X\otimes_k Z,Y)&=\Hom^*_{kG}(X\otimes_k ik_E{\uparrow^G},Y)\\
&\cong\Hom^*_{kG}((X{\downarrow_E}\otimes_k ik){\uparrow^G},Y)\\
&\cong\Hom^*_{kG}(X{\downarrow_E\uparrow^G},Y)\\
&\cong\Hom^*_{kE}(X{\downarrow_E},Y{\downarrow_E}) \\
&\cong \bigoplus_{\fq}
\Hom^*_{kE}(\gam_\fq X{\downarrow_E},\gam_\fq Y{\downarrow_E}).
\end{align*}
Since both $\gam_\fq X{\downarrow_E}$ and $\gam_\fq Y{\downarrow_E}$ are 
both non-zero, by the classification of localising subcategories
of $\KInj{kE}$ and Lemma \ref{le:minimality} we have
$\Hom^*_{kE}(\gam_\fq X{\downarrow_E},\gam_\fq Y{\downarrow_E})\ne 0$.
 
So using the minimality test of
Lemma \ref{le:minimality-tensor} for the tensor triangulated category
$\KInj{kG}$, we deduce that 
$\gam_\fp \KInj{kG}$ is minimal among tensor ideal localising subcategories. 
\end{proof}

Using the discussion of Section~\ref{sec:Wednesday3}, this implies the following result for the stable module category.

\begin{theorem}
As a tensor triangulated category, $\StMod(kG)$ is stratified by the canonical action of $H^*(G,k)$. \qed
\end{theorem}

\section{Exercises}
\label{exer:Friday}
In what follows $\sfT$ denotes a compactly generated $R$-linear triangulated category. 
Ideals and elements in $R$ will be assumed to be homogeneous.

\begin{enumerate}[\quad\rm(1)]
\item 
Assume $\sfT$ is a tensor triangulated category. Let $\fa$ and $\fb$ be ideals in $R$. Prove that if $\fa\subseteq\fb$, then $\kos X\fb$ is in $\Thick_{\sfT}(\kos X\fa)$. Deduce that if $\sqrt{\fa}=\sqrt{\fb}$, then
\[
\Thick_{\sfT}(\kos X\fa)= \Thick_{\sfT}(\kos X\fb)\,.
\]
\item 
Prove that for each ideal $\fa$ of $R$ and object $X$ in $\sfT$ one has 
\[
\gam_{\mcV(\fa)}X\in \Loc_{\sfT}(\kos X\fa)\quad\text{and}\quad \kos X\fa \in \Thick_{\sfT}(\gam_{\mcV(\fa)}X)\,.
\]
Here $\sfT$ need not be tensor triangulated.\smallskip

\noindent Hint: Induce on the number of generators of $\fa$. Use the
description of $\gam_{\mcV(r)}X$ from
Proposition~\ref{prop:hocolim-local}.

\item Using the previous exercise prove that for each specialisation closed subset $\mcV$ of $\Spec R$, and any decomposition $\bigcup_{i\in I}\mcV(\fa_{i})=\mcV$, and set $G$ of compact generators for $\sfT$, there are equalities
\[
\sfT_{\mcV} = \Loc_{\sfT}(\gam_{\mcV(\fa_{i})}C\mid C\in G) =
\Loc_{\sfT}(\kos C{\fa_{i}} \mid C\in G)
\]
This proves that $\sfT_{\mcV}$ is compactly generated, so  the exact functors $\gam_{\mcV}$ and $L_{\mcV}$ are smashing, meaning that they commute with coproducts in $\sfT$.

\item Let $E=\langle g\mid g^p=1\rangle$ and $k$ be a field of characteristic $p$. Write down a basis for the Koszul construction on $kE$ with respect to $z=g-1$, namely the differential graded algebra $A=kE\langle y\rangle$ with $|y|=-1$, $y^2=0$ and $dy=z$. 

Prove that the inclusion of the exterior algebra on the element $z^{p-1}y$ of degree $-1$ into $A$ is a quasi-isomorphism.

\item Using tensor products and the K\"unneth theorem, show that for a general elementary abelian $p$-group $E=\langle g_1,\dots,g_r\rangle$, the inclusion of an exterior algebra on the elements $z_i^{p-1}y_i$ ($i\le i \le r$) into the Koszul construction is a quasi-isomorphism.

\item Let $A$ be a local ring, with residue field $k$. Prove that a
  complex $M\in\sfD(A)$ is in $\Thick(k)$ if and only if the length of
  the $A$-module $H^{*}M$ is finite.\smallskip

\noindent Hint: For the converse, first consider the case where $M$ is a module; when $M$ is a complex, induce on the number of non-zero cohomology modules of $M$.
\end{enumerate}

\begin{appendix}
\renewcommand{\thechapter}{Appendix A.}

\chapter[\hspace{1.6cm}Support for modules over commutative rings]
{Support for modules over commutative rings}\label{Appendix}

\renewcommand{\thechapter}{A}
\markboth{Support for modules over commutative rings}
{Support for modules over commutative rings}

Let $A$ be a commutative noetherian ring. We consider the category $\Mod A$\index{ModA@$\Mod A$} of $A$-modules and its full subcategory $\mod A$\index{modA@$\mod A$} which is formed by all finitely generated $A$-modules. Note that an $A$-module is finitely generated if and only if it is noetherian.

The \emph{spectrum} \index{spectrum} $\Spec A$ \index{specA@$\Spec A$} of $A$ is the set of prime ideals in it. A subset of $\Spec A$ is \emph{Zariski closed} if it is of the form
\[
\mcV(\fa)=\{\fp\in\Spec A\mid \fa\subseteq\fp\}
\]
for some ideal $\fa$ of $A$. A subset $\mcV$ of $\Spec A$ is \emph{specialisation closed}\index{specialisation!closed} if for any pair $\fp\subseteq\fq$ of prime ideals, $\fp\in \mcV$ implies $\fq\in \mcV$. The \emph{specialisation closure}\index{specialisation!closure} of a subset $\mcU\subseteq\Spec A$ is the subset
\[
\cl\mcU=\{\fp\in\Spec A\mid\text{there exists $\fq\in\mcU$ with $\fq\subseteq \fp$}\}\,.
\]
This is the smallest specialisation closed subset containing $\mcU$.

\section{Big support}
The \emph{big support}\index{big support} of an $A$-module $M$ is the subset
\[
\Supp_{A} M=\{\fp\in\Spec A\mid M_\fp\neq 0\}\,.
\]
Observe that this is a specialisation closed subset of $\Spec A$.

\begin{lemma}
\label{le:ideal}
One has $\Supp_{A} A/\fa=\mcV(\fa)$ for each ideal $\fa$ of $A$.
\end{lemma}

\begin{proof}
Fix $\fp\in\Spec A$ and let $S=A\setminus\fp$. Recall that for any $A$-module $M$, an element $x/s$ in $S^{-1}M=M_\fp$ is zero iff there exists $t\in S$ such that $tx=0$. Thus we have $(A/\fa)_\fp=0$ iff there exists $t\in S$ with $t(1+\fa)=t+\fa=0$ iff $\fa\not\subseteq \fp$.
\end{proof}

\begin{lemma}
\label{le:exact}
If $0\to M'\to M\to M''\to 0$ is an exact sequence of $A$-modules, then $\Supp_{A} M=\Supp_{A} M'\cup\Supp_{A} M''$.
\end{lemma}

\begin{proof}
The sequence $0\to M'_\fp\to M_\fp\to M''_\fp\to 0$ is exact for each $\fp$ in $\Spec A$.
\end{proof}

\begin{lemma}
\label{le:sum}
 Let $M=\sum_i M_i$ be an $A$-module, written as a sum of submodules
 $M_i$.  Then $\Supp_{A} M=\bigcup_i \Supp_{A} M_i$.
\end{lemma}

\begin{proof}
The assertion is clear if the sum $\sum_iM_i$ is direct, since
\[
\bigoplus_i(M_i)_\fp=(\bigoplus_iM_i)_\fp.
\] 
As $M_i\subseteq M$ for all $i$ one gets $\bigcup_i\Supp_{A} M_i\subseteq \Supp_{A} M$, from Lemma~\ref{le:exact}.  On the other hand, $M=\sum_iM_i$ is a factor of $\bigoplus_iM_i$, so $\Supp_{A} M\subseteq\bigcup_i \Supp_{A} M_i$.
\end{proof}

We write $\ann_{A}M$\index{annaM@$\ann_{A}M$} for the ideal of elements in $A$ that annihilate $M$.

\begin{lemma}
\label{le:fg}
One has $\Supp_{A} M\subseteq\mcV(\ann_{A}M)$, with equality when $M$
is in $\mod A$.
\end{lemma}

\begin{proof}
Write $M=\sum_i M_i$ as a sum of cyclic modules $M_i\cong A/{\fa_i}$. Then
\[
\Supp_{A} M=\bigcup_i \Supp_{A} M_i=\bigcup_i
\mcV(\fa_i)\subseteq\mcV(\bigcap_i\fa_i)=\mcV(\ann_{A}M),
\]
and equality holds if the sum is finite.
\end{proof}

\begin{lemma}
\label{le:sub0}
Let $M\neq 0$ be an $A$-module.  If $\fp$ is maximal in the set of ideals which annihilate a non-zero element of $M$, then $\fp$ is prime.
\end{lemma}

\begin{proof}
Suppose $0\neq x\in M$ and $\fp x=0$. Let $a,b\in A$ with $ab\in\fp$ and $a\not\in\fp$. Then $(\fp,b)$ annihilates $a x\neq 0$, so the maximality of $\fp$ implies $b\in\fp$. Thus $\fp$ is prime. 
\end{proof}

\begin{lemma}
\label{le:sub}
Let $M\neq 0$ be an $A$-module.  There exists a submodule of $M$ which
is isomorphic to $A/\fp$ for some prime ideal $\fp$.
\end{lemma}
\begin{proof}
The set of ideals annihilating a non-zero element has a maximal
element, since $A$ is noetherian. Now apply Lemma~\ref{le:sub0}.
\end{proof}

\begin{lemma}
\label{le:filtr}
For each $M$ in $\mod A$ there exists a finite filtration
\[
0=M_0\subseteq M_1\subseteq \ldots \subseteq M_n=M
\]
such that each factor $M_i/M_{i-1}$ is isomorphic to $A/{\fp_i}$ for
some prime ideal $\fp_i$. In that case one has
$\Supp_{A}M=\bigcup_{i}\mcV(\fp_{i})$.
\end{lemma}

\begin{proof}
Repeated application of Lemma~\ref{le:sub} yields a chain of submodules
$0=M_0\subseteq M_1\subseteq  M_2\subseteq \ldots$  of $M$ such that each
$M_{i}/M_{i-1}$ is isomorphic to $A/\fp_{i}$ for some $\fp_{i}$.
This chain stabilises since $M$ is noetherian, and therefore $\bigcup_{i} M_{i}=M$.

The last assertion follows from Lemmas~\ref{le:exact} and \ref{le:ideal}.
\end{proof}

\section{Serre subcategories}
A full subcategory $\sfC$ of $A$-modules is called \emph{Serre subcategory}\index{Serre subcategory} if for every exact sequence $0\to M'\to M\to M''\to 0$ of $A$-modules, $M$ belongs to $\sfC$ if and only if $M'$ and $M''$
belong to $\sfC$.  We set
\[
\Supp_{A}\sfC=\bigcup_{M\in\sfC}\Supp_{A} M.
\]

\begin{proposition}
\label{pr:serre}
The assignment $\sfC\mapsto \Supp_{A}\sfC$ induces a bijection between
\begin{enumerate}[\quad\rm(1)]
\item the set of Serre subcategories of
$\mod A$, and 
\item the set of specialisation closed subsets of $\Spec A$.
\end{enumerate}
Its inverse takes $\mcV\subseteq \Spec A$ to $\{M\in\mod A\mid
\Supp M\subseteq \mcV\}$.
\end{proposition}

\begin{proof}
Both maps are well defined by Lemmas~\ref{le:exact} and \ref{le:fg}.
If $\mcV\subseteq \Spec A$ is a specialisation closed subset, let
$\sfC_\mcV$ denote the smallest Serre subcategory containing
$\{A/\fp\mid\fp\in \mcV\}$. Then we have $\Supp\sfC_\mcV=\mcV$, by
Lemmas~\ref{le:ideal} and \ref{le:exact}. Now let $\sfC$ be a Serre
subcategory of $\mod A$. Then \[\Supp\sfC=\{\fp\in\Spec A\mid
A/\fp\in\sfC\}\] by Lemma~\ref{le:filtr}. It follows that
$\sfC=\sfC_\mcV$ for each Serre subcategory $\sfC$, where
$\mcV=\Supp\sfC$. Thus $\Supp\sfC_1=\Supp\sfC_2$ implies
$\sfC_1=\sfC_2$ for each pair $\sfC_1,\sfC_2$ of Serre subcategories.
\end{proof}

\begin{corollary}
Let $M$ and $N$ be in $\mod A$. Then $\Supp_{A} N\subseteq \Supp_{A}
M$ if and only if $N$ belongs to the smallest Serre subcategory
containing $M$.
\end{corollary}

\begin{proof}
With $\sfC$ denoting the smallest Serre subcategory containing $M$,
there is an equality $\Supp_{A}\sfC=\Supp_{A} M$ by
Lemma~\ref{le:exact}. Now apply Proposition~\ref{pr:serre}.
\end{proof}

\section{Localising subcategories}

A full subcategory $\sfC$ of $A$-modules is said to be \emph{localising}\index{localising subcategory} if it is a Serre subcategory and if for any family of $A$-modules $M_i\in\sfC$ the sum $\bigoplus_iM_i$ is in $\sfC$.
The result below is from \cite[p.~ 425]{Gabriel:1962a}.

\begin{corollary}
The assignment $\sfC\mapsto \Supp_{A}\sfC$ gives a bijection between 
\begin{enumerate}[\quad\rm(1)]
\item the set of localising subcategories of $\Mod A$, and
\item the set of specialisation closed subsets of $\Spec A$.
\end{enumerate}
Its inverse takes $\mcV\subseteq \Spec A$ to $\{M\in\Mod A\mid \Supp_{A} M\subseteq \mcV\}$.
\end{corollary}

\begin{proof}
The proof is essentially the same as the one of
Proposition~\ref{pr:serre} if we observe that any $A$-module $M$ is
the sum $M=\sum_i M_i$ of its finitely generated submodules. Note that
$M$ belongs to a localising subcategory $\sfC$ if and only if all $M_i$
belong to $\sfC$.  In addition, we use that $\Supp_{A} M=\bigcup_i \Supp_{A}
M_i$; see Lemma~\ref{le:sum}.
\end{proof}

\section{Injective modules}

The following proposition collects the basic properties of injective modules over a commutative noetherian ring; for a proof see \cite[\S18]{Matsumura:1986}. For each $\fp\in\Spec A$ we denote $E(A/\fp)$ the injective hull of $A/\fp$.

\begin{proposition}
\pushQED{\qed}
\begin{enumerate}[\quad\rm(1)]
\item An arbitrary direct sum of injective modules is injective.
\item Every injective module decomposes essentially uniquely as a
  direct sum of injective indecomposables.
\item $E(A/\fp)$ is indecomposable for each $\fp$ in $\Spec A$.
\item Each injective indecomposable is isomorphic to $E(A/\fp)$ for a unique prime ideal $\fp$.
\qedhere
\end{enumerate}
\end{proposition}

Let $\fp$ a prime ideal in $A$ and let $M$ be an $A$-module.  The module $M$ is said to be \emph{$\fp$-torsion}\index{ptorsion@$\fp$-torsion} if each element of $M$ is annihilated by a power of $\fp$; equivalently:
\[
M =\{x\in M\mid \text{there exists an integer $n\geq 0$ such that $\fp^n\cdot x=0$}\}.
\]
The module $M$ is \emph{$\fp$-local}\index{fplocal@$\fp$-local} if the natural map $M\to M_\fp$ is bijective.

For example, $A/\fp$ is $\fp$-torsion, but it is $\fp$-local only if
$\fp$ is a maximal ideal, while $A_\fp$ is $\fp$-local, but it is
$\fp$-torsion only if $\fp$ is a minimal prime ideal.  The $A$-module
$E(A/\fp)$ is both $\fp$-torsion and $\fp$-local. Using this
observation the following  is easy to prove.

\begin{lemma} 
\label{le:inj-local}
\pushQED{\qed}
Let $\fp$ and $\fq$ be prime ideals in $A$. Then
\[
E(A/\fp)_\fq = \begin{cases} 
E(A/\fp)  & \text{if $\fq\in \mcV(\fp)$,}\\
0 & \text{otherwise.}
\end{cases} \qedhere
\]
\end{lemma}

\section{Support}
\label{se:appendix-supp}

Each $A$-module $M$ admits a minimal injective resolution
\[
0\to M\to I^0\to I^1 \to I^2 \to \cdots
\]
and such a resolution is unique, up to isomorphism of complexes of $A$-modules.  We say that $\fp$ \emph{occurs} in a minimal injective resolution $I$ of $M$, if for some integer $i\in\mathbb Z$, the module $I^i$ has a direct summand isomorphic to $E(A/\fp)$. We call the set
\[
\supp_{A} M=
\left\{\fp\in\Spec A \left|
\begin{gathered}
\text{$\fp$ occurs in a minimal} \\
\text{ injective resolution of $M$}
\end{gathered}
\right.\right\}
\] 
the \emph{support of $M$}\index{support}. In the literature, it is sometimes referred to as the `small support' or the `cohomological support', to distinguish it from the big support $\Supp_{A} M$.

\begin{lemma}
\label{le:supp-local}
Let $M$ be an $A$-module and $\fp\in\Spec A$. If $I$ is a minimal injective resolution of $M$, then  $I_\fp$ is a minimal injective resolution of $M_\fp$. Therefore
\[
\supp_{A} (M_\fp)=\supp_{A} M\cap \{\fq\in\Spec A\mid \fq\subseteq \fp\}\,.
\]
\end{lemma}

\begin{proof}
For the first assertion, see for example Lemmas~5 and 6 in \cite[\S18]{Matsumura:1986}. The formula for the support of $M_{\fp}$ then follows from Lemma~\ref{le:inj-local}. 
\end{proof}

We write $k(\fp)$ for the residue field $A_\fp/\fp A_\fp$ at $\fp\in\Spec A$.

\begin{lemma}
\label{le:supp-tests}
Let $M$ be an $A$-module and $\fp\in\Spec A$. The following are equivalent:
\begin{enumerate}[\quad\rm(1)]
\item $\fp\in\supp_{A} M$;
\item $\Ext^*_{A_{\fp}}(k(\fp),M_{\fp})\ne 0$;
\item $\Tor^{A_{\fp}}_{*}(k(\fp),M_{\fp})\ne 0$.
\end{enumerate}
\end{lemma}

\begin{proof}
For the equivalence of (1) and (2) see \cite[Theorem~18.7]{Matsumura:1986}. The
equivalence of (2) and (3) is more involved, and was proved by Foxby~\cite{Foxby:1979}.
\end{proof}

\begin{lemma}
\label{le:supp-ann}
For each $A$-module $M$ one has
\[
\supp_{A} M\subseteq \cl(\supp_{A} M) =\Supp_{A} M \subseteq\mcV(\ann M)\,,
\]
and equalities hold when $M$ is finitely generated. 
\end{lemma}

\begin{proof}
The equality follows from Lemma~\ref{le:supp-local}, while the inclusions are obvious.
 
Suppose now $M$ is finitely generated. Given Lemma~\ref{le:fg}, to
prove that equalities hold, it remains to verify $\Supp_{A}M\subseteq
\supp_{A} M$. If $M_{\fp}\ne 0$ for some $\fp\in\Spec A$, then
$k(\fp)\otimes_{A_{\fp}} M_{\fp}\ne 0$ by Nakayama's Lemma, for
$M_{\fp}$ is a finitely generated module over the local ring
$A_{\fp}$. In particular, $\Tor^{A_{\fp}}_{*}(k(\fp),M_{\fp})\ne 0$,
and hence $\fp$ is in $\supp_{A}M$, by Lemma~\ref{le:supp-tests}.
\end{proof} 

\section{Specialization closed sets}  
Given a subset $\mcU\subseteq\Spec A$, we consider the full subcategory
\[
\sfM_\mcU=\{M\in \Mod A\mid \supp_{A} M\subseteq \mcU\}.
\] 
The next result does not hold for arbitrary subsets of $\Spec A$. In fact, it can be used to characterise the property that $\mcV$ is specialisation closed.

\begin{lemma}
\label{le:specialisation}
Let $\mcV$ be a specialisation closed subset of $\Spec A$. Then for each $A$-module $M$, one has
\[
\supp_{A} M\subseteq \mcV\quad\Longleftrightarrow \quad M_\fq=0\text{ for
  each $\fq$ in $\Spec A\setminus \mcV$}.
  \] 
The subcategory $\sfM_\mcV$ of $\Mod A$ is closed under set-indexed direct sums, and in any exact sequence $0\to M'\to M\to M''\to 0$ of $A$-modules, $M$ is in $\sfM_\mcV$ if and only if $M'$ and $M''$ are in $\sfM_\mcV$.
\end{lemma}

\begin{proof} 
Since $\mcV$ is specialisation closed, it contains $\supp_{A} M$ if and only if it contains $\cl(\supp_{A} M)$. Thus the first statement is a consequence of Lemma \ref{le:supp-ann}.  Given this the second statement follows, since for each $\fq$ in $\Spec A$, the functor taking an $A$-module $M$ to $M_\fq$ is exact and preserves set-indexed direct sums.
\end{proof}

Torsion modules and local modules can be recognised from their supports.

\begin{lemma} 
\label{le:torsion-local}
Let $M$ be an $A$-module and $\fp\in\Spec A$. Then
\begin{enumerate}[\quad\rm(1)]
\item $M$ is $\fp$-local if and only if $\supp_{A}  M\subseteq \{\fq\in \Spec A\mid \fq\subseteq \fp\}$, and
\item $M$ is $\fp$-torsion if and only if  $\supp_{A}  M\subseteq\mcV(\fp)$.
\end{enumerate}
\end{lemma}

\begin{proof}
Let $I$ be a minimal injective resolution of $M$.

(1) Since $I_\fp$ is a minimal injective resolution of $M_\fp$, and minimal injective resolutions are unique up to isomorphism, $M\xra{\sim} M_\fp$ if and only if $I\xra{\sim} I_\fp$. This implies the desired equivalence, by Lemma~\ref{le:inj-local}.

(2) When $\supp_{A}  M\subseteq\mcV(\fp)$, then, by definition of support, one has that $I^0$ is isomorphic to a direct sum of copies of $E(A/\fq)$ with $\fq\in\mcV(\fp)$.  Since each $E(A/\fq)$ is $\fp$-torsion, so is $I^0$, and hence the same is true of $M$, for it is a submodule of $I^0$.

Conversely, when $M$ is $\fp$-torsion, $M_\fq =0$ for each $\fq$ in $\Spec A$ with $\fq\not\supseteq\fp$. This implies $\supp_{A} M\subseteq \mcV(\fp)$, by Lemma~\ref{le:specialisation}.
\end{proof} 

\begin{lemma}
\label{le:anti-specialisation}
Let $\fp$ be a prime ideal in $A$ and set $\mcU=\{\fq\in\Spec A\mid \fq\subseteq \fp\}$. 
\begin{enumerate}[{\quad\rm(1)}]
\item 
Restriction along the morphism $A\to A_\fp$ identifies $\Mod A_\fp$ with the subcategory $\sfM_\mcU$ of $\Mod A$. Therefore $\sfM_\mcU$ is closed under taking kernels, cokernels, extensions, direct sums, and products.
\item 
If $M,N$ are $A$-modules and one of them belongs to $\sfM_\mcU$, then $\Hom_A(M,N)$ is in $\sfM_\mcU$.
\end{enumerate}
\end{lemma}

\begin{proof}
(1) The objects in the subcategory $\sfM_\mcU$ are precisely the $\fp$-local $A$-modules, by Lemma~\ref{le:torsion-local}.  Thus the inclusion functor has a left and a right adjoint. It follows that $\sfM_\mcU$ is an exact abelian and extension closed subcategory of $\Mod A$, closed under set-indexed direct sums and products.

(2) The action of $A$ on $\Hom_A(M,N)$ factors via $\End_A(M)$ and $\End_A(N)$. If $M$ or $N$ is $\fp$-local, then this action factors through the map $A\to A_\fp$.
\end{proof} 

\end{appendix}

\clearpage
\phantomsection
\addcontentsline{toc}{chapter}{Index}

\printindex

\begin{thebibliography}{1}
\thispagestyle{empty}
\markboth{Bibliography}{Bibliography}
\addcontentsline{toc}{chapter}{Bibliography}

\bibitem{Avramov/Buchweitz:2000a} L.~L. Avramov and R.-O. Buchweitz, \emph{Support varieties and cohomology over complete intersections}, Invent. Math. \textbf{142} (2000), 285--318.

\bibitem{Avramov/Buchweitz/Iyengar/Miller:2010a} L.~L. Avramov, R.-O. Buchweitz, S.~B.~Iyengar, and C.~Miller, \emph{Homology of perfect complexes}, Adv. Math. \textbf{223} (2010), 1731--1781.
[Corrigendum: Adv. Math. \textbf{225} (2010), 3576--3578.]

\bibitem{BBD:1983} A. Beilinson, J. Bernstein, and P. Deligne, \emph{Faisceaux pervers}, Ast\'erisque \textbf{100}, Soc. Math. France, 1983. 

\bibitem{Benson:1991a} D.~J. Benson, \emph{{Representations and
      Cohomology I: Basic representation theory of finite groups and
      associative algebras}}, Cambridge Studies in Advanced
  Mathematics, vol.~30, Cambridge University Press, 1991, reprinted in
  paperback, 1998.

\bibitem{Benson:1991b}
D.~J. Benson, \emph{{Representations and Cohomology II: Cohomology of groups
  and modules}}, Cambridge Studies in Advanced Mathematics, vol.~31, Cambridge
  University Press, 1991, reprinted in paperback, 1998.

\bibitem{Benson:1995a}
D.~J. Benson, \emph{{Cohomology of modules in the principal block of a finite
  group}}, New York Journal of Mathematics \textbf{1} (1995), 196--205.

\bibitem{Benson:2002a}
D.~J. Benson, \emph{{The nucleus, and extensions between modules for a finite
  group}}, Representations of Algebras, Proceedings of the Ninth International
  Conference (Beijing 2000), Beijing Normal University Press, 2002, second
  volume.

\bibitem{Benson/Carlson/Rickard:1996a}
D.~J. Benson, J.~F. Carlson, and J.~Rickard, \emph{{Complexity and varieties
  for infinitely generated modules, II}}, Math.\ Proc.\ Camb.\ Phil.\ Soc.
  \textbf{120} (1996), 597--615.

\bibitem{Benson/Carlson/Rickard:1997a}
D.~J. Benson, J.~F. Carlson, and J.~Rickard, \emph{{Thick subcategories of the stable module category}}, Fundamenta
  Mathematicae \textbf{153} (1997), 59--80.

\bibitem{Benson/Iyengar/Krause:2008a}
D.~J. Benson, S.~B. Iyengar, and H.~Krause, \emph{{Local cohomology and support
  for triangulated categories}}, Ann.\ Scient.\ \'Ec.\ Norm.\ Sup.\ (4)
  \textbf{41} (2008), 1--47.

\bibitem{Benson/Iyengar/Krause:bik3}
D.~J. Benson, S.~B. Iyengar, and H.~Krause, \emph{{Stratifying modular representations of finite groups}},
Ann. of Math. \textbf{175} (2012), to appear.

\bibitem{Benson/Iyengar/Krause:bik2}
D.~J. Benson, S.~B. Iyengar, and H.~Krause, \emph{{Stratifying triangulated categories}}, J. Topology, to appear.

\bibitem{Benson/Iyengar/Krause:bik4}
D.~J. Benson, S.~B. Iyengar, and H.~Krause, \emph{{Colocalizing subcategories and cosupport}}, J. Reine Angew. Math., to appear.

\bibitem{Benson/Krause:2008a}
D.~J. Benson and H.~Krause, \emph{{Complexes of injective $kG$-modules}},
  Algebra \& Number Theory \textbf{2} (2008), 1--30.

\bibitem{Benson/Krause/Schwede:2004a}
D.~J. Benson, H.~Krause, and S.~Schwede, \emph{{Realizability of modules over
  Tate cohomology}}, Trans.\ Amer.\ Math.\ Soc. \textbf{356} (2004),
  3621--3668.

\bibitem{Benson/Krause/Schwede:2005a}
D.~J. Benson, H.~Krause, and S.~Schwede, \emph{{Introduction to realizability of modules over Tate
  cohomology}}, Fields Inst. Comm., vol.~45, American Math.\ Society, 2005,
  pp.~81--97.

\bibitem{Bernstein/Gelfand/Gelfand:1978} 
I.~N.~Bernstein, I.~M.~Gelfand, S.~I.~Gelfand, \emph{Algebraic vector bundles on $P\sp{n}$ and problems of linear algebra}, Funct.  Anal. Appl. \textbf{12} (1978), 212--214.

\bibitem{Bokstedt/Neeman:1993a} 
M.~B{\"o}kstedt and A.~Neeman, \emph{Homotopy colimits in triangulated categories}, Compositio Math. \textbf{86} (1993), 209--234.

\bibitem{Brown:1962a} E. H. Brown, Jr., \emph{Cohomology theories},
  Ann. of Math. (2) {\bf 75} (1962), 467--484.

\bibitem{Bruns/Herzog:1998a}  W.~Bruns and J.~Herzog, \emph{Cohen-Macaulay rings},
  Cambridge Studies in Advanced Mathematics, \textbf{39}. Cambridge
  University Press, Cambridge, 1998.  Revised edition.

\bibitem{Buchweitz:1986}
R.-O.~Buchweitz, \emph{Maximal Cohen-Macaulay modules and Tate-cohomology over Gorenstein rings},
preprint, Univ. Hannover 1986;\newline http://hdl.handle.net/1807/16682.

\bibitem{Buehler:2010}
T.~B\"uhler, \emph{Exact categories}, Expo. Math., \textbf{28} (2010) 1--69.

\bibitem{Carlson:1981b}
J.~F. Carlson, \emph{{The complexity and varieties of modules}}, Integral
  representations and their applications, Oberwolfach, 1980, Lecture Notes in
  Mathematics, vol. 882, Springer-Verlag, Ber\-lin/New York, 1981,
  pp.~415--422.

\bibitem{Carlson:1983a}
J.~F. Carlson, \emph{{The varieties and cohomology ring of a module}}, J.~Algebra
  \textbf{85} (1983), 104--143.

\bibitem{Cartan/Eilenberg:1956}
H.~Cartan and S.~Eilenberg, \emph{Homological Algebra}, Princeton University Press, Princeton, NJ, 1956.

\bibitem{Chouinard:1976a}
L.~Chouinard, \emph{{Projectivity and relative projectivity over group rings}},
  J. Pure Appl. Algebra \textbf{7} (1976), 278--302.

\bibitem{Dade:1978b}
E.~C. Dade, \emph{{Endo-permutation modules over $p$-groups, II}}, Ann.\ of
  Math. \textbf{108} (1978), 317--346.

\bibitem{Elmendorf/Kriz/Mandell/May:1996a}
A.~D. Elmendorf, I.~K{\v{r}}{\'{\i}}{\v{z}}, M.~A. Mandell, and J.~P. May,
  \emph{{Rings, modules and algebras in stable homotopy theory}}, Surveys and
  Monographs, vol.~47, American Math.\ Society, 1996.

\bibitem{Evens:1961a}
L.~Evens, \emph{{The cohomology ring of a finite group}}, Trans.\ Amer.\ Math.\
  Soc. \textbf{101} (1961), 224--239.

\bibitem{Foxby:1979}
H.-B.~Foxby, \emph{Bounded complexes of flat modules}, J. Pure Appl. Algebra \textbf{15} (1979), 149--172.

\bibitem{Gabriel:1962a} P.~Gabriel, \emph{Des cat\'egories ab\'eliennes},
Bull. Soc. Math. France \textbf{90} (1962), 323--448.

\bibitem{Gabriel/Zisman:1967a} P.~Gabriel and M.~Zisman,
  \emph{Calculus of fractions and homotopy theory}, Ergebnisse der
  Mathematik und ihrer Grenzgebiete, vol. 35, Springer-Verlag, New
  York, 1967.

\bibitem{Happel:1988a}
D.~Happel, \emph{{Triangulated categories in the representation theory of
  finite dimensional algebras}}, London Math.\ Soc.\ Lecture Note Series, vol.
  119, Cambridge University Press, 1988.

\bibitem{Hartshorne:1967}
R. Hartshorne, \emph{Local cohomology: A seminar given by A. Grothendieck (Harvard, 1961)}, Lecture Notes in Math. 41, Springer-Verlag, 1967.

\bibitem{Hopkins:1987a}
M.~J. Hopkins, \emph{{Global methods in homotopy theory}}, Homotopy Theory,
  Durham 1985, Lecture Notes in Mathematics, vol. 117, Cambridge University
  Press, 1987.

\bibitem{Hovey/Palmieri/Strickland:1997a}
M.~Hovey, J.~H. Palmieri, and N.~P. Strickland, \emph{{Axiomatic stable
  homotopy theory}}, Mem.\ AMS, vol. 128, American Math.\ Society, 1997.

\bibitem{Iyengar:2004} S.~Iyengar, \emph{Modules and cohomology over group algebras. One commutative algebraist's perspective}, in: Trends in commutative algebra (Berkeley 2002), Mathematical Sciences Research Inst. Publ. \textbf{51}, Cambridge Univ. Press, Cambridge, (2004) 51--86.

\bibitem{Iyengar:2006} S.~Iyengar, \emph{The classification of thick subcategories of perfect complexes over commutative noetherian rings}, in Thick subcategories - classifications and applications, Oberwolfach Report No. 8/2006.

\bibitem{Iyengar:2007} S.~Iyengar, G.~Leuschke, A.~Leykin, C.~Miller, E.~Miller, A.~Singh, and U.~Walther,
\emph{Twenty-four hours of local cohomology}, Graduate Stud. Math. \textbf{87}, American Mathematical Society, Providence, RI, 2007.

\bibitem{Keller:1994a}
B.~Keller, \emph{{Deriving DG categories}}, Ann.\ Scient.\ \'Ec.\ Norm.\ Sup.\ (4) \textbf{27} (1994), 63--102.

\bibitem{Krause:2005a}
H.~Krause, \emph{{The stable derived category of a noetherian scheme}}, Compositio Math. \textbf{141} (2005), 1128--1162.
  
\bibitem{Krause:2007a} H. Krause, \emph{Derived categories, resolutions, and
  Brown representability}, in \emph{Interactions between homotopy theory
    and algebra}, 101--139, Contemp. Math., 436 Amer. Math. Soc.,
  Providence, 2007.

\bibitem{Matsumura:1986} H. Matsumura, \emph{Commutative ring theory},
Cambridge University Press (1986). 

\bibitem{Neeman:1992a}
A.~Neeman, \emph{{The chromatic tower for $D(R)$}}, Topology \textbf{31} (1992), 519--532.

\bibitem{Neeman:1992b} A. Neeman, \emph{The connection between the
  $K$-theory localization theorem of Thomason, Trobaugh and Yao and
  the smashing subcategories of Bousfield and Ravenel},
  Ann. Sci. \'Ecole Norm. Sup. (4) {\bf 25} (1992), no.~5, 547--566.

\bibitem{Neeman:1996a} A.~Neeman, \emph{The Grothendieck duality
    theorem via Bousfield's techniques and Brown representability},
  J. Amer. Math. Soc. {\bf 9} (1996), 205--236.

\bibitem{Neeman:2001} A.~Neeman, \emph{Triangulated categories}, Annals of Mathematics Studies 148,
Princeton University Press, 2001.

\bibitem{Quillen:1971a}
D.~G. Quillen, \emph{{A cohomological criterion for $p$-nilpotence}}, J.~Pure
  \& Applied Algebra \textbf{1} (1971), 361--372.

\bibitem{Quillen:1971b}
D.~G. Quillen, \emph{{The spectrum of an equivariant cohomology ring, I}}, Ann.\ of
  Math. \textbf{94} (1971), 549--572.

\bibitem{Quillen:1973a} D. Quillen, \emph{Higher algebraic
    $K$-theory. I}, in {\it Algebraic $K$-theory, I: Higher
    $K$-theories (Proc. Conf., Battelle Memorial Inst., Seattle,
    Wash., 1972)}, 85--147. Lecture Notes in Math., 341, Springer,
  Berlin, 1973.

\bibitem{Ravenel:1984}
D.~C.~Ravenel, \emph{Localization with respect to certain periodic homology theories}, Amer. J. of Math., \textbf{106} (1984), 351--414.

\bibitem{Rickard:1989a}
J.~Rickard, \emph{{Derived categories and stable equivalence}}, J.~Pure Appl. Algebra \textbf{61} (1989), 303--317.

\bibitem{Rickard:1997a}
J.~Rickard, \emph{{Idempotent modules in the stable category}}, J.~London
  Math.\ Soc. \textbf{178} (1997), 149--170.
  
\bibitem{Verdier:1997a} J.-L. Verdier, \emph{Des cat\'egories d\'eriv\'ees des
cat\'egories ab\'eliennes}, Ast\'erisque No. 239 (1996), {\rm xii}+253
pp. (1997).

\bibitem{Weibel:1994} C.~Weibel, \emph{Homological algebra}, Cambridge Studies in
  Advanced Mathematics 38, Cambridge University Press, 1994.

\end{thebibliography}
\end{document}